\documentclass[11pt]{amsart} \usepackage{amsfonts}
\usepackage{bbding}
\usepackage{graphicx}
\usepackage{amsfonts,amssymb,amsmath,amsthm,amscd}
\allowdisplaybreaks
\usepackage{nicematrix}
\usepackage{mathrsfs}
\usepackage{mathdots}
\usepackage{xcolor}
\usepackage{empheq}
\usepackage{adforn}
\usepackage{cancel}
\usepackage{caption}
\usepackage{xr}
\usepackage{tikz-cd}
\usepackage[framemethod=tikz]{mdframed}
\usepackage{amsfonts}
\usepackage{mathdots}
\usepackage{pinlabel}
\usepackage{booktabs}
\usepackage{enumerate}
\usepackage{lmodern}
\usepackage{mdframed}
\usepackage{stmaryrd}
\usepackage{blindtext}
\usepackage{leftidx}
\usepackage{accents}
\usepackage{easytable}
\usepackage{enumitem}

\usepackage{fouriernc}
\usepackage[scaled=0.875]{helvet}
\usepackage{courier}
\usepackage[marginparwidth=2.5cm]{geometry}
\definecolor{navie}{RGB}{0, 70, 140}

\usepackage[backref=page]{hyperref}
\usepackage{xcolor}
\hypersetup{
	colorlinks,
	linkcolor={red!60!black},
	citecolor={navie},
	urlcolor={navie}
}
\PassOptionsToPackage{linktocpage}{hyperref}

\makeatletter
\newtheorem*{rep@theorem}{\rep@title}
\newcommand{\newreptheorem}[2]{%
	\newenvironment{rep#1}[1]{%
		\def\rep@title{#2 \ref{##1}}%
		\begin{rep@theorem}}%
		{\end{rep@theorem}}}
\makeatother

\makeatletter
\newtheorem*{rep@cor}{\rep@title}
\newcommand{\newrepcor}[2]{%
	\newenvironment{rep#1}[1]{%
		\def\rep@title{#2 \ref{##1}}%
		\begin{rep@cor}}%
		{\end{rep@cor}}}
\makeatother

\makeatletter
\newtheorem*{rep@prop}{\rep@title}
\newcommand{\newrepprop}[2]{%
	\newenvironment{rep#1}[1]{%
		\def\rep@title{#2 \ref{##1}}%
		\begin{rep@prop}}%
		{\end{rep@prop}}}
\makeatother
\newtheorem{corollary}{Corollary}[section]
\newtheorem{corx}{Corollary}

\newtheorem{theorem}[corollary]{Theorem}
\newtheorem{thmx}[corx]{Theorem}

\newtheorem{proposition}[corollary]{Proposition}

\newrepcor{cor}{Corollary}
\newreptheorem{theorem}{Theorem}
\newrepprop{prop}{Proposition}

\newtheorem*{theorem*}{Theorem}
\newtheorem{lemma}[corollary]{Lemma}

\theoremstyle{definition} 
\newtheorem{propdef}[corollary]{Proposition/Definition}

\newtheorem*{remark*}{Remark}
\newtheorem{definition}[corollary]{Definition}
\newtheorem*{definition*}{Definition}
 
\newtheorem{remark}[corollary]{Remark}

\theoremstyle{remark} 
\numberwithin{equation}{section}
\numberwithin{figure}{section}
\newcommand{\cir}[1]{\textcircled{\raisebox{-0.5pt}{\small{#1}}}}

\newcommand{\dt}[1]{\accentset{\,\mbox{\large\bfseries .}}{#1}}

\newcommand{\id}{\mathit{id}}
\newcommand{\diag}{\mathrm{diag}}
\newcommand{\Hom}{\mathrm{Hom}}
\newcommand{\Aut}{\mathrm{Aut}}
\newcommand{\R}{\mathbb{R}}
\newcommand{\C}{\mathbb{C}}
\renewcommand{\O}{\mathbf{O}}
\newcommand{\Z}{\mathbb{Z}}

\newcommand{\A}{\mathbb{A}}
\newcommand{\Fix}{\mathrm{Fix}}

\newcommand{\Hitchin}{{\scalebox{0.5}{Hitchin}}}
\newcommand{\hf}{F}

\newcommand{\ve}[1]{{\boldsymbol{#1}}}
\newcommand{\ca}[1]{{\mathcal{#1}}}
\newcommand{\T}{\mathsf{T}}

\newcommand{\TT}{\ca{L}}
\newcommand{\dif}{\mathsf{d}}
\newcommand{\dc}{\mathsf{d}^\mathrm{c}}

\newcommand{\vol}{\mathit{vol}}
\newcommand{\ac}[1]{{\boldsymbol{#1}}} 
\newcommand{\transp}[1]{\leftidx{^\mathsf{t}}\!{#1}}
\newcommand{\pa}{\partial}
\newcommand{\bpa}{{\bar{\partial}}}

\newcommand{\dz}{{\dif z}}
\newcommand{\dbz}{{\dif\bar z}}
\newcommand{\dx}{{\dif x}}
\newcommand{\dy}{{\dif y}}

\newcommand{\pazbz}{\partial_{z\bar z}^2}
\newcommand{\ima}{\boldsymbol{i}}
\newcommand{\jma}{\boldsymbol{j}}
\newcommand{\kma}{\boldsymbol{k}}
\newcommand{\lma}{\boldsymbol{l}}

\newcommand{\mat}[1]{\begin{pmatrix}#1\end{pmatrix}}

\newcommand{\cj}{\overline}

\newcommand{\vae}{\varepsilon}
\newcommand{\JT}{\ac{J}^{\ca{T}}\!\!}
\newcommand{\JN}{\ac{J}^\ca{N}\!\!}

\renewcommand{\H}{\mathbb{H}}

\newcommand{\GL}{\mathrm{GL}}
\newcommand{\SL}{\mathrm{SL}}

\newcommand{\SU}{\mathrm{SU}}

\newcommand{\SO}{\mathrm{SO}}

\newcommand{\Sp}{\mathrm{Sp}}
\newcommand{\II}{\ve{I\!\!I}}
\newcommand{\sth}{{\,\raisebox{0.8pt}{\scalebox{0.65}{$\bullet$}}\,}}
\newcommand{\wh}{\widehat}

\newcommand{\wt}{\widetilde}
\newcommand{\Sym}{\mathrm{Sym}}

\newcommand{\Sy}{X}

\DeclareMathOperator{\Ad}{Ad}
\DeclareMathOperator{\End}{End}
\DeclareMathOperator{\ad}{ad}
\DeclareMathOperator{\im}{\mathsf{Im}}
\DeclareMathOperator{\re}{\mathsf{Re}}
\DeclareMathOperator{\Tr}{Tr}

\begin{document}
\title[Cyclic Higgs bundles and minimal surfaces in pseudo-hyperbolic spaces]{Cyclic Higgs bundles and minimal surfaces\\ in pseudo-hyperbolic spaces} 
\author{Xin Nie}
\address{Shing-Tung Yau Center of Southeast University, Southeast University, Nanjing, China}
\maketitle

%

\begin{quote}
	\footnotesize
	\textsc{Abstract.}
	We introduce a type of minimal surface in the pseudo-hyperbolic space $\mathbb{H}^{n,n}$ (with $n$ even) or $\mathbb{H}^{n+1,n-1}$ (with $n$ odd) associated to cyclic $\mathrm{SO}_0(n,n+1)$-Higg bundles. By establishing the infinitesimal rigidity of these surfaces, we get a new proof, for $\mathrm{SO}_0(n,n+1)$, of Labourie's theorem that the holonomy map restricts to an immersion on the cyclic locus of Hitchin base, and extend it to Collier's components. This implies Labourie's former conjecture in the case of the exceptional group $G_2'$, for which we also show that these minimal surfaces are $\boldsymbol{J}$-holomorphic curves of a particular type in the almost complex $\mathbb{H}^{4,2}$.
\end{quote}

%

\tableofcontents

\section{Introduction}
Let $S_g$ be the orientable closed surface of genus $g\geq2$ and $G$ be a noncompact real semisimple Lie group of rank $\geq2$.
Higher Teichm\"uller Theory (see \cite{wienhard_invitation} for a survey) studies fundamental group representations $\rho:\pi_1(S_g)\to G$, especially the $\rho$'s in the \emph{Hitchin components} \cite{hitchin_lie}, the \emph{maximal components} \cite{burger-iozzi-wienhard} and more recently the \emph{$\Theta$-positive components} \cite{guichard-labourie-wienhard,guichard-wienhard_positivity}.  
An important object in the theory are $\rho$-equivariant minimal mappings $f:\wt{S}_g\to \Sy_G$, where $\wt{S}_g$ is the universal cover of $S_g$
and $\Sy_G$ is the Riemannian symmetric space associated to $G$. Such an $f$ is known to exist as long as $\rho$ is \emph{Anosov} (see \cite{labourie_cross}), a property shared by all those components.

The uniqueness of $f$ is a difficult issue and undergoes extensive studies. Labourie \cite{labourie_cross} conjectured that \emph{$f$ is unique when $\rho$ is Hitchin} and proved it when $G$ has rank $2$ in \cite{labourie_cyclic}. Markovic's recent result \cite{markovic} (see also \cite{markovic-sagman-smillie}) implies that the conjecture fails for hermitian $G$ of rank $\geq3$. The rank $2$ case has also been studied by Labourie \cite{labourie_cubic}, Loftin \cite{loftin_amer}, Collier \cite{collier_dedicata}, Alessandrini-Collier \cite{alessandrini-collier} and Collier-Toulisse-Tholozan \cite{collier-tholozan-toulisse}. Most of the existing works are based on investigating certain ``incarnation'' of $f$ as a special surface in other spaces, namely
\begin{itemize}
	\item (\cite{labourie_cubic,loftin_amer}) \emph{hyperbolic affine sphere} in $\R^3$ when $\rho$ is an $\SL(3,\R)$-Hitchin representation;
	\item (\cite{collier-tholozan-toulisse}) \emph{maximal surface} in  the pseudo-hyperbolic space $\H^{2,q}$ when $\rho$ is a maximal $\SO_0(2,q+1)$-representation (see also \cite{bonsante-schlenker_max,bonsante-seppi_survey} for the $q=1$ case);
	\item (\cite{labourie_cyclic}) \emph{cyclic surface} in a homogeneous space of the complexified group $G^\C$ when $\rho$ is in the cyclic locus of the $G$-Hitchin component (for any real split $G$; see also \cite{alessandrini-collier, collier_dedicata} for $\Sp(4,\R)$).
\end{itemize}

In this paper, we define and study yet another incarnation when $G$ is $\SO_0(n,n+1)$  or the exceptional simple Lie group $G_2'\subset\SO_0(3,4)$: a special type of spacelike (i.e.\ Riemannian) minimal surfaces in $\H^{n,n}$ (if $n$ is even) or $\H^{n+1,n-1}$ (if $n$ is odd) with a suitably defined Gauss map that recovers $f$. This generalizes the $n=2$ case contained in \cite{collier-tholozan-toulisse}, although when $n\geq3$, like Labourie's cyclic surfaces,  our construction does not apply to all $\rho$ in a component of $\SO_0(n,n+1)$-representations, but only to those in the cyclic locus. The definition of these minimal surfaces, which we proceed to explain, is related to Chern's Frenet frame \cite{chern_on} for minimal spheres and Bryant's superminimal surfaces \cite{bryant_conformal}.


\subsection{A-surfaces}\label{sec_intro1}
For a smooth curve $t\mapsto \gamma(t)$ in the Euclidean space $\mathbb{E}^n$, a classical fact due to Frenet, Serret and Jordan is that if the derivatives $\gamma',\gamma'',\cdots,\gamma^{(n-1)}$ are linearly independent for all $t$, then there is a unique orthonormal moving frame $t\mapsto (e_1(t),\cdots,e_n(t))$ along $\gamma$ such that $e_1=\gamma'/|\gamma'|$ is the unit tangent vector and the derivative of the frame takes the form
\begin{equation}\label{eqn_frenetcurve}
(e_1',\cdots,e_n')=(e_1,\cdots,e_n)\,
\scalebox{0.9}
{$
\mat{
0&-\kappa_2&&&\\[0.1cm]
\kappa_2&0&-\kappa_3&&\\[-0.1cm]
&\kappa_3&0&\ddots&\\
&&\ddots&\ddots&-\kappa_n\\[0.1cm]
&&&\kappa_n&0
}
$}
\end{equation}
for nowhere vanishing functions $\kappa_2,\cdots,\kappa_n$, which capture the extrinsic geometry of $\gamma$. 

Chern \cite{chern_on} gave an analogue of this theory for minimal surfaces and essentially showed that if $\Sigma$ is a minimal sphere in a Riemannian manifold $M$ of constant curvature and is not contained in a totally geodesic hypersurface (this forces $M$ to be even dimensional), then there is an orthogonal splitting
\begin{equation}\label{eqn_introdecomp}
\T M|_\Sigma=\TT_1\oplus \cdots \oplus\TT_n
\end{equation}
with the following properties:
\begin{itemize}
\item
each $\TT_i$ is a rank $2$ vector bundle on $\Sigma$ and $\TT_1$ is the tangent bundle $\T \Sigma\subset \T M|_\Sigma$;
\item
the decomposition of the Levi-Civita connection $\nabla$ on $\T M|_\Sigma$ under \eqref{eqn_introdecomp} has a similar form as the matrix in \eqref{eqn_frenetcurve}, only with the $i$th diagonal zero replaced by a connection on $\TT_i$ and the function $\kappa_j$ replaced by a $1$-form
$\ve\alpha_j\in\Omega^1(\Sigma,\Hom(\TT_{j-1},\TT_j))$;
\item every $\ve\alpha_j(X):\TT_{j-1}\to\TT_j$ ($j=2,\cdots,n$, $X\in C^\infty(\Sigma,\T \Sigma)$) is \emph{conformal} in the sense that it maps any circle in a fiber of $\TT_{j-1}$ to a circle in $\TT_j$ (possibly reduced to a point).
\end{itemize}

The proof of Chern makes use of the fact that some naturally defined holomorphic $k$-differentials on $\Sigma$ vanish because $\Sigma$ is a sphere (a famous idea of Hopf). For $\Sigma$ of general topological type, which we consider from now on, the existence of a splitting \eqref{eqn_introdecomp} with these properties defines a highly special class of minimal surfaces, namely the \emph{superminimal} ones (c.f.\ Remark \ref{rk_historysuperminimal}).

We may generalize this definition of superminimality to the case where $M$ is pseudo-Riemannian by further requiring every $\TT_i$ to be (positive or negative) definite. Relaxing the definition, we may call $\Sigma$ \emph{quasi-superminimal} if the condition in the last bullet point is only imposed for $j\leq n-1$ and not for $j=n$ (see Definition \ref{def_quasisuper}). In this case, if furthermore $\TT_i$ is positive and negative definite for all odd and even $i$ respectively, then we call $\Sigma$ an \emph{A-surface} (``A'' for the Alternating space-/time-likeness of the $\TT_i$'s).  
Note that this forces the signature $(p,q)$ of $M$ to be
\begin{equation}\label{eqn_pqintro}
(p,q)=
\Bigg\{\hspace{-0.1cm}
\begin{array}{ll}
	(n,n)&\text{if $n$ is even},\\
	(n+1,n-1)&\text{if $n$ is odd}.\\
\end{array}
\end{equation}
In this paper, we always assume \eqref{eqn_pqintro} whenever an A-surface is being considered.  We first show:
\begin{thmx}\label{thm_intro0}
Suppose $M$ is orientable and is locally modeled on the pseudo-hyperbolic space $\H^{p,q}$. Then any oriented A-surface $\Sigma\subset M$ satisfies (only rough statements are given here):
\begin{enumerate}[label=(\arabic*)]
	\item\label{item_intro1} 
	(Theorem \ref{thm_holo}; see also Remark \ref{rk_local}) Each $\TT_i$ in \eqref{eqn_introdecomp} can be interpreted as a hermitian holomorphic line bundle on $\Sigma$, such that every $\ve\alpha_j$ with $2\leq j\leq n-1$ is an $\Hom(\TT_{j-1},\TT_j)$-valued holomorphic $1$-form, while
		$\ve\alpha_n$ is the sum of holomorphic $1$-forms $\ve\alpha_n^+$ and $\ve\alpha_n^-$ with values in $\Hom(\TT_{n-1},\TT_n)$ and $\Hom(\TT_{n-1},\cj{\TT}_n)\cong\Hom(\TT_{n-1},\TT_n^{-1})$, respectively.
	\item\label{item_intro2}
		(Theorem \ref{thm_affinetoda}) The Gauss-Codazzi equations of $\Sigma$ boil down to the following affine Toda system involving the above $\ve\alpha_2,\cdots,\ve\alpha_{n-1},\ve\alpha_n^\pm$ and the hermitian metric $\ve{h}_i$ on $\TT_i$:
	$$
	\begin{cases}
		\pazbz\log h_1=\tfrac{1}{2}h_1-\tfrac{h_2}{h_1}|\alpha_2|^2,\\[0.2cm]
		\pazbz\log h_k=\tfrac{h_k}{h_{k-1}}|\alpha_k|^2-\tfrac{h_{k+1}}{h_k}|\alpha_{k+1}|^2\quad (2\leq k\leq n-2),\\[0.2cm]
		\pazbz\log h_{n-1}=\tfrac{h_{n-1}}{h_{n-2}}|\alpha_{n-1}|^2-\tfrac{h_{n}}{h_{n-1}}|\alpha_n^+|^2-\tfrac{|\alpha_n^-|^2}{h_{n-1}h_n}~,\\[0.2cm]
		\pazbz\log h_n=\tfrac{h_{n}}{h_{n-1}}|\alpha_n^+|^2-\tfrac{1}{h_{n-1}h_n}|\alpha_n^-|^2.
	\end{cases}
	$$
	\item\label{item_intro3} (Theorem \ref{thm_gaussmap}) The Gauss map from $\Sigma$ to the symmetric space
	$$
	\Sy_{\SO_0(p,q+1)}=\big\{\text{timelike totally geodesic $q$-spheres in $\H^{p,q}$}\big\},
	$$
	defined by assigning to every $z\in \Sigma$ the $q$-sphere passing through $z$ and tangent to the timelike part $\ca{N}^-:=\TT_2\oplus\TT_4\oplus \cdots$ of the normal bundle, is a conformal minimal immersion.
\end{enumerate}
\end{thmx}
\begin{remark}\label{rk_intro1}
We need $M$ to be oriented in order to endow $\TT_n$ with an orientation and hence view it as a holomorphic line bundle (but not for $\TT_1,\cdots,\TT_{n-1}$). Therefore, we work with the hyperboloid model of $\H^{p,q}$ in this paper rather than the projective model, as the latter is not orientable in even dimensions. If the orientation of $M$ is inverted, so is the one on $\TT_n$. This interchanges $\TT_n$ with  $\cj{\TT}_n\cong\TT_n^{-1}$ and interchanges $\ve\alpha_n^+$ with $\ve\alpha_n^-$.
\end{remark}

The data $\TT_i,\ve\alpha_j$ in \ref{item_intro1} and the equations in \ref{item_intro2} resemble those appearing in cyclic Higgs bundles (see \cite{baraglia_thesis,baraglia_dedicata, collier-li,dai-li_minimal,dai-li_on}). We establish next a link between them.

\subsection{Cyclic $\SO_0(n,n+1)$-Higgs bundles}
On a closed Riemann surface $\Sigma$, a \emph{Higgs bundle} is a pair $(\ca{E},\Phi)$, where $\ca{E}$ is a holomorphic vector bundle and $\Phi$ an $\End(\ca{E})$-valued holomorphic $1$-form. 
Given a real reductive group $G$, one may define a special class of $(\ca{E},\Phi)$'s called \emph{$G$-Higgs bundles}, whose Hitchin connections yield representations of $\pi_1(\Sigma)$ in $G$, and the \emph{non-abelian Hodge correspondence} claims that this roughly gives a bijection between the moduli spaces of such Higgs bundles and representations (see \cite{hitchin_selfduality,hitchin_lie,gothen,2009arXiv0909.4487G}). In this paper, we consider $(\ca{E},\Phi)$'s of the form
\begin{align*}
&\ca{E}=\TT_n^{-1}\oplus\TT_{n-1}^{-1}\oplus\cdots \oplus \TT_1^{-1}\oplus \ca{O}\oplus \TT_1\oplus  \cdots \oplus\TT_{n-1}\oplus \TT_n,\\
&\Phi=
\scalebox{0.85}{
$\left(
\begin{array}{ccccc|c|ccccc}
	0&&&&&&&&&\ve\alpha_n^-&\\
	\ve\alpha_n^+&0&&&&&&&&&\ve\alpha_n^-\\
	&\hspace{-0.1cm}\ve\alpha_{n-1}&\ddots&&&&&&&&\\
	&&\hspace{-0.2cm}\ddots&\ddots&&&&&&&\\
	&&&\hspace{-0.1cm}\ve\alpha_2&0&&&&&&\\
	\midrule
	&&&&\tfrac{1}{\sqrt{2}}&0&&&&&\\
	\midrule
	&&&&&\tfrac{1}{\sqrt{2}}&0&&&&\\
	&&&&&&\ve\alpha_2&\ddots&&&\\
	&&&&&&&\hspace{-0.2cm}\ddots&\ddots&&\\
	&&&&&&&&\hspace{-0.1cm}\ve\alpha_{n-1}&0&\\
	&&&&&&&&&\ve\alpha_n^+&0
\end{array}
\right)$
},
\end{align*}
where $\TT_1,\cdots, \TT_n$ are holomorphic line bundles with $\TT_1=\ca{K}^{-1}$.
The $\frac{1}{\sqrt{2}}$ can be brought to $1$ by a gauge transformation, but this exact constant is convenient for our purpose (see Remark \ref{rk_intro2}). We further assume that none of $\ve\alpha_2,\cdots,\ve\alpha_{n-1},\ve\alpha_n^+$ is identically zero. In this case, $(\ca{E},\Phi)$ is stable if and only if either $\ve\alpha_n^-\not\equiv0$ or $\deg(\TT_n)<0$ (see Prop./Def. \ref{propdef_harmonic}).

Such an $(\ca{E},\Phi)$ is a \emph{cyclic} $\SO_0(n,n+1)$-Higgs bundle, so its harmonic metric splits into hermitian metrics on the $\TT_i$'s (see \cite[Cor.\ 2.11]{collier-li}). Much of our work also relies on the following extra technical assumption on the divisors of the holomorphic forms $\ve\alpha_2,\cdots,\ve\alpha_{n-1},\ve\alpha_n^\pm$:
\begin{equation}\label{eqn_introprec}
	(\ve\alpha_2)\prec\cdots\prec(\ve\alpha_{n-1})\prec\min\big\{(\ve\alpha_n^+),(\ve\alpha_n^-)\big\},
\end{equation}
where we write ``$D_1\prec D_2$'' if the value of the divisor $D_1$ is strictly less than that of $D_2$ at every point where the former is nonzero (which always holds when $D_1=\emptyset$). In fact, by a generalization of Dai and Li's work \cite{dai-li_minimal}, the harmonic metrics on $\TT_1,\cdots,\TT_{n-1}$ have semi-negative Chern forms under this assumption (see \S \ref{subsec_hypo}), and some of our results can be viewed as new geometric applications of this property in addition to those in \cite{dai-li_minimal}.

We show that such  $(\ca{E},\Phi)$'s yield A-surfaces, where \eqref{eqn_introprec}
is needed in our proof of uniqueness:
\begin{thmx}[Theorem \ref{thm_higgsA}]\label{thm_intro1}
Given a stable $(\ca{E},\Phi)$ as above, with the $\ve\alpha_j$'s satisfying \eqref{eqn_introprec}, let $f:\wt{\Sigma}\to \Sy_{\SO_0(n,n+1)}$ be an equivariant conformal minimal mapping given by the harmonic metric. Then there is a unique antipodal pair $\pm F$ of A-surface immersions of $\wt{\Sigma}$ into $\H^{p,q}$ such that $f$ is given by post-composing $\pm F$ with their Gauss map. 
\end{thmx}
\begin{remark}\label{rk_intro2}
It is also shown that the structural data of $\hf$ in Theorem \ref{thm_intro0} \ref{item_intro1} (after possibly swapping $\pm\hf$) coincides with the data used to construct $(\ca{E},\Phi)$, as the notation already suggests. One may justify this by checking that the equations in Theorem \ref{thm_intro0} \ref{item_intro2} are exactly the Hitchin equation of $(\ca{E},\Phi)$. The constant $\frac{1}{\sqrt{2}}$ in $\Phi$ is required for this to be an exact match: otherwise the holomorphic data behind $F$ and $(\ca{E},\Phi)$ will differ by constants. Meanwhile, since the antipodal map inverts the orientation of $\H^{p,q}$, by Remark \ref{rk_intro1}, the structural data of $-\hf$ is the same except for a flip of $\TT_n^{\pm1}$ and $\ve\alpha_n^\pm$. This is related to a hidden symmetry in the definition of $(\ca{E},\Phi)$: if we interchange $\TT_n$ with $\TT_n^{-1}$ and interchange $\ve\alpha_n^\pm$, the resulting $(\ca{E},\Phi)$ is equivalent to the original one.
\end{remark}

The representation $\rho:\pi_1(\Sigma)\to\SO_0(n,n+1)$ given by the Hitchin connection of $(\ca{E},\Phi)$ belongs to the $\SO_0(n,n+1)$-Hitchin component when $\ve\alpha_2,\cdots,\ve\alpha_{n-1}$ and $\ve\alpha_n^+$ are nowhere zero (we have $\TT_i=\ca{K}^{-i}$ and $\ve\alpha_n^-\in H^0(\Sigma,\ca{K}^{2n})$ in this case). 
For the slightly broader class of $(\ca{E},\Phi)$'s with only $\ve\alpha_2,\cdots,\ve\alpha_{n-1}$ nowhere zero, $\rho$ belongs to the component $\ca{X}_d$ studied by Collier \cite{collier} with $d=|\deg(\TT_n)|\in[0,(2g-2)n]$ (after possibly interchanging $\TT_n^{\pm1}$ as in Remark \ref{rk_intro2} when $\ve\alpha_n^-\not\equiv0$, we may assume $d=\deg(\TT_n^{-1})$).
As the core results of this paper, we show next that the A-surfaces in Theorem \ref{thm_intro1} are infinitesimally rigid as minimal surfaces and draw consequences for $\ca{X}_d$.

\subsection{Infinitesimal rigidity}
%
It is well known that in a negatively curved pseudo-Riemannian manifold $M$, any spacelike minimal submanifold $\Sigma$ with timelike normal bundle  is locally volume maximizing (see Proposition \ref{prop_maximal}). This fact underlies the works \cite{bonsante-schlenker_max, collier-tholozan-toulisse, labourie-toulisse, labourie-toulisse-wolf} on minimal surfaces in $\mathbb{H}^{2,q}$ and explains the terminology ``maximal surface'' used in this setting.

When $\Sigma$ is an A-surfaces and $M$ is modeled on a pseudo-hyperbolic space of dimension $2n\geq 6$, we generalize this fact by showing that if $\TT_1,\cdots,\TT_{n-1}$ have semi-negative Chern forms, then $\Sigma$ is a saddle type critical point for the area functional: its variations along the spacelike and timelike parts of the normal bundle strictly increase and decrease the area, respectively (Theorem \ref{thm_Avariation}). As a consequence, $\Sigma$ does not admit any nonzero compactly supported Jacobi field (Corollary \ref{coro_Avariation1}). Thus, when $\Sigma$ is closed, we obtain:
\begin{thmx}[Corollary \ref{coro_Avariation2}]\label{thm_intro2}
In the setting of Theorem \ref{thm_intro0}, further assume that $\Sigma$ is closed and the holomorphic forms $\ve\alpha_2,\cdots,\ve\alpha_{n-1},\ve\alpha_n^\pm$ satisfy \eqref{eqn_introprec}. Then $\Sigma$ does not admit any nonzero Jacobi field.
\end{thmx}
Since the Jacobi fields of a minimal surface are the infinitesimal deformations (see \S \ref{subsec_variation}), this theorem signifies that $\Sigma$ is infinitesimally rigid as a minimal surface. 
Therefore, in the same way as how Labourie \cite[\S 7]{labourie_cyclic} uses the infinitesimal rigidity of cyclic surfaces to show that the holonomy map restricts to an immersion on the cyclic locus of the Hitchin base (a vector bundle over the Teichm\"uller space), we may deduce from Theorem \ref{thm_intro2} a generalization of Labourie's theorem to Collier's component $\ca{X}_d$ for $\SO_0(n,n+1)$:

\begin{thmx}[Theorem \ref{thm_main1}]\label{thm_intro3}
Fix $g\geq2$ and $n\geq3$. Given $1\leq d\leq (2g-2)n$, consider the natural mapping-class-group-equivariant smooth map 
$$
\Psi:\ca H_4\oplus\ca H_6\oplus\cdots\oplus \ca H_{2n-2}\oplus \ca F_{d}\to \ca X_d
$$
assigning to each point of the domain fiber bundle the holonomy of the cyclic $\SO_0(n,n+1)$-Higgs bundle constructed from that point
(see \S \ref{subsec_collier}). Then $\Psi$ restricts to an immersion on the subbundle $\ca F_{d}$.
\end{thmx}
\begin{remark}
For simplicity, we only treat in this paper the case $d>0$, as $\Psi$ is a smooth map between smooth manifolds in this case. When $d=0$, the domain is more complicated and has both orbifold and non-orbifold singularities (see \cite{alessandrini-collier,collier}), but Theorem \ref{thm_intro3} and its proof still hold on the smooth locus.
\end{remark}
The following cases of Theorem \ref{thm_intro3} are already known:
\begin{itemize}
\item 
The $d=(2g-2)n$ case is the original theorem of Labourie \cite{labourie_cyclic}. In this case, $\ca X_d$ is just the Hitchin component $\ca{X}_\Hitchin(\SO_0(n,n+1))$ and $\ca F_d=\ca H_{2n}$ is the cyclic locus of the Hitchin base. 
\item
When $n=2$, Theorem \ref{thm_intro3} still holds by virtue of the maximality (or the global uniqueness result \cite[Thm.\ 3.13]{collier-tholozan-toulisse}), but the statement and proof are already contained in \cite{alessandrini-collier,collier_dedicata,collier-tholozan-toulisse}.  In this case, the domain of $\Psi$ is $\ca{F}_d$ alone, and by an argument of Labourie, we can use a differential-topological result \cite[Thm.\ 8.1.1]{labourie_cyclic}, in combination with the relation between $\Psi$ and the energy functional \cite{labourie_cross}, to infer that $\Psi$ is a diffeomorphism. 
\item 
The first case above with $n=3$ is special: in this case, $\Psi$ sends $\ca{F}_{d}$ into the Hitchin component $\ca{X}_\Hitchin(G_2')$ of the exceptional group $G_2'$ (see below), which is contained in $\ca{X}_\Hitchin(\SO_0(3,4))$. The same argument as the last case shows that $\Psi$ is a diffeomorphism from $\ca{F}_d$ to $\ca{X}_\Hitchin(G_2')$.
\end{itemize}
The diffeomorphisms in the latter two cases prove Labourie's conjecture for the rank $2$ groups $\SO_0(2,3)$ and $G_2'$, while it remains elusive whether $\Psi$ is a diffeomorphism in general. Our final result is about the nature of the A-surfaces occurring in the last case.
\subsection{$G_2'$-Hitchin representations}
The rank $2$ exceptional complex simple Lie group $G_2^\C$ has two real forms: the compact one $G_2$ and the split one $G_2'$. The latter identifies with the group of isometries of $\H^{4,2}$ preserving a specific non-integrable almost complex structure $\ac{J}$ (see \S \ref{subsec_almostcomplex}). 

An interesting subclass of A-surfaces in this almost complex $\H^{4,2}$ are the \emph{spacelike $\ac{J}$-holomorphic curves with nowhere vanishing second fundamental form and timelike osculation lines} (where \emph{osculation line} refers to the $\ac{J}$-complex line in a normal space formed by the images of second fundamental form; see \S \ref{subsec_jholominimal}). The Gauss map of such an A-surface takes values in 
$$
X_{G_2'}=\big\{\text{timelike totally geodesic $\ac{J}$-holomorphic $2$-spheres in $\H^{4,2}$}\big\}\subset X_{\SO_0(4,3)}
$$
(see \S \ref{subsec_jholoA}).
We show that the A-surfaces from Theorem \ref{thm_intro1} belong to this subclass when $(\ca{E},\Phi)$ is in the  $G_2'$-Hitchin component (Proposition \ref{prop_g2A}). The infinitesimal rigidity of such A-surfaces imply Labourie's conjecture for $G_2'$, as explained in the last subsection.
The outcome can be summarized as:
\begin{corx}\label{coro_intro4}
For any Hitchin representation $\rho:\pi_1(S_g)\to G_2'$, there exist 
\begin{itemize}
\item
a unique orientation-preserving $\rho$-equivariant immersion of $\wt{S}_g$ into $\H^{4,2}$ as a $\ac{J}$-holomorphic curve with nowhere vanishing second fundamental form and timelike osculation lines; and
\item
a unique $\rho$-equivariant minimal immersion of $\wt{S}_g$ into the symmetric space $\Sy_{G_2'}$.
\end{itemize}
The latter is given by post-composing the former with its Gauss map, which assigns to every point of the $\ac{J}$-holomorphic curve the $2$-sphere generated by the osculation line at that point.
\end{corx}
It was Baraglia \cite[\S 3.6]{baraglia_thesis} who first observed that Higgs bundles in the $G_2'$-Hitchin component yield $\ac{J}$-holomorphic curves in $\H^{4,2}$. The above statement gives a more complete picture. 
 \begin{figure}[h]
	\includegraphics[width=10cm]{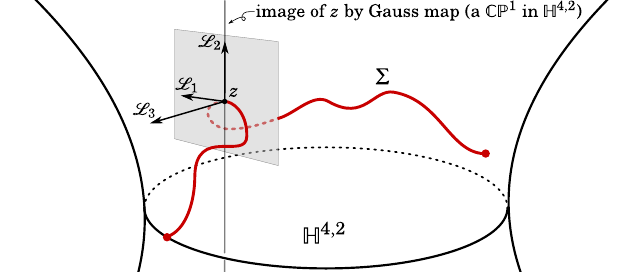}
    \caption{$\ac{J}$-holomorphic A-surface in $\H^{4,2}$}\label{figure_cartoon}
\end{figure} 
Figure \ref{figure_cartoon} is a cartoon of the theorem, where we use the familiar drawing for the anti-de Sitter space $\H^{2,1}$ to depict $\H^{4,2}$ schematically and show the A-surface $\Sigma$ as a ``Frenet curve''.
The ideal boundary $\pa_\infty\H^{4,2}$ carries, besides the conformal metric of signature $(3,2)$ induced by the metric of $\H^{4,2}$, a non-integrable $2$-plane distribution induced by $\ac{J}$ (see \cite{baez-huerta,sagerschnig}). Therefore, a natural question is how the frontier of $\Sigma$ in $\pa_\infty\H^{4,2}$ (depicted as two dots in Figure \ref{figure_cartoon}, but should actually be a Jordan curve) behaves with respect to these structures. Another intriguing problem is to geometrize $G_2'$-Hitchin representations by using $\Sigma$ in a similar way as what \cite{collier-tholozan-toulisse} did with maximal surfaces in $\H^{2,q}$.

Finally, we note that the recent work of Collier and Toulisse \cite{collier-toulisse}, independent of the current paper, uses cyclic surfaces to study spacelike $\ac{J}$-holomorphic curves in $\H^{4,2}$ coming from $G_2'$-representations that are not necessarily Hitchin.
\subsection*{Organization of the paper}
In \S \ref{sec_pseudo} and \S \ref{sec_minimal}, we review backgrounds on pseudo-hyperbolic spaces, the Lie group $G_2'$ and minimal submanifolds. In \S \ref{sec_asurface}, we discuss the definition of A-surfaces and establish their fundamental properties, namely Theorem \ref{thm_intro0}. In \S \ref{sec_infinitesimal}, we prove the infinitesimal rigidity result, Theorem \ref{thm_intro2}, with a crucial ingredient (Theorem \ref{thm_daili}) postponed to the appendix. In \S \ref{sec_cyclic}, we relate Higgs bundles to A-surfaces by proving Theorem \ref{thm_intro1}, and finally use the infinitesimal rigidity to deduce Theorem \ref{thm_intro3} and show that the $G_2'$ case yields $\ac{J}$-holomorphic curves.

\subsection*{Acknowledgments} We would like to thank Qiongling Li for her enormous help on Higgs bundles. We are also grateful to Andrea Seppi and Jun Wang for enlightening discussions, and to Brian Collier and J\'er\'emy Toulisse for informing us of their work \cite{collier-toulisse}.

\section{Pseudo-hyperbolic spaces and $G_2'$}\label{sec_pseudo}
In this section, we review the necessary backgrounds on the pseduo-hyperbolic space $\H^{p,q}$, the Lie group $G_2'$ and the almost complex structure on $\H^{4,2}$.
\subsection{Pseudo-hyperbolic spaces}\label{subsec_pseudohyperbolic}
Let $\R^{p,q+1}$ denote $\R^{p+q+1}$ endowed with the quadratic form $\langle\sth,\sth\rangle$ given by
$$
\langle x,x\rangle:=x_1^2+\cdots+x_p^2-x_{p+1}^2-\cdots-x_{p+q+1}^2.
$$ 
The \emph{pseudo-hyperbolic space} $\H^{p,q}$ is defined as the hypersurface 
$$
\H^{p,q}:=\big\{x\in\R^{p,q+1}\mid \langle x,x\rangle=-1\big\}
$$
endowed with the metric restricted from $\langle\sth,\sth\rangle$, which is a pseudo-Riemannian metric of signature $(p,q)$ with constant sectional curvature $-1$. In particular, $\H^{p,0}$ is the two-sheeted hyperboloid, namely two copies of the usual hyperbolic space $\H^p$, while $\H^{0,q}$ is the unit sphere $\mathbb{S}^q$ with metric multiplied by $-1$. Topologically, $\mathbb{H}^{p,q}$ identifies with the product of the $p$-dimensional ball with the $q$-sphere through the diffeomorphism
$$
\H^{p,q}\cong \mathbb{B}^p\times \mathbb{S}^q,\quad x\mapsto\Bigg(\frac{(x_1,\cdots,x_p)}{\scalebox{0.85}{$\sqrt{x_{p+1}^2+\cdots+x_{p+q+1}^2}$}}\,,\ \frac{(x_{p+1},\cdots,x_{p+q+1})}{\scalebox{0.85}{$\sqrt{x_{p+1}^2+\cdots+x_{p+q+1}^2}$}}\Bigg).
$$

\begin{remark}
More commonly used in the recent literature (e.g.\ \cite{bonsante-schlenker_max, collier-tholozan-toulisse,labourie-toulisse,labourie-toulisse-wolf}) is the \emph{projective model} of the pseudo-hyperbolic space. It identifies with the quotient of our hyperboloid model $\H^{p,q}$ by the antipodal map $x\mapsto -x$. The crucial advantage of the $\H^{p,q}$ here is that it is orientable, while the quotient $\H^{p,q}/(x\sim -x)$ with $q\geq1$ is orientable only when $p+q$ is odd.
\end{remark}

A linear subspace $L\subset\R^{p,q+1}$ is said to be \emph{nondegnerate} or \emph{positive/negative definite} if the quadratic form $\langle\sth,\sth\rangle|_L$ is.
The intersection of $\H^{p,q}$ with some nondegenerate $L$ is called a \emph{pseudo-hyperbolic subspace}, which is a totally geodesic copy of $\H^{p',q'}$ in $\H^{p,q}$ ($(p',q'+1)$ being the signature of $\langle\sth,\sth\rangle|_L$). We refer to the copies of $\H^{0,n}$ in $\H^{p,q}$ as \emph{timelike totally geodesic $n$-spheres} (when $n=1$, these are the complete  timelike  geodesics).
In particular, all the timelike totally geodesic $q$-spheres are homotopic to $\{0\}\times\mathbb{S}^q\subset\mathbb{B}^p\times\mathbb{S}^q\cong\H^{p,q}$, hence can be oriented in a uniform way (corresponding to a choice of generator for the $q$th homology $H_q(\mathbb{H}^{p,q},\mathbb{Z})\cong\mathbb{Z}$).

The automorphism group $\mathrm{O}(p,q+1)$ of $\R^{p,q+1}$ acts on $\H^{p,q}$ as the isometry group. It has $4$ connected components. The index $2$ subgroup $\SO(p,q+1)$ consists of the orientation-preserving isometries, whereas the identity component $\SO_0(p,q+1)$ consists of the isometries preserving the above orientation of timelike totally geodesic $q$-spheres (or equivalently, acting trivially on homology). 

Let $e_1=(1,0,\cdots,0),\cdots, e_{p+q+1}=(0,\cdots,0,1)$ be the standard basis of $\R^{p,q+1}$ and put
$$
V:=\R e_1\oplus\cdots\oplus\R e_p\,,\quad W:=\R e_{p+1}\oplus\cdots\oplus\R e_{p+q+1}.
$$
These are maximal positive and negative definite subspaces of $\R^{p,q+1}$ orthogonal to each other. The stabilizer of either $V$ or $W$ in $\SO_0(p,q+1)$, namely 
$$
\SO(V)\times\SO(W)=\SO(p)\times\SO(q+1)\subset\SO_0(p,q+1),
$$
is a maximal compact subgroup of $\SO_0(p,q+1)$. Meanwhile, the $\SO_0(p,q+1)$-action on the space of maximal (i.e.\ $(q+1)$-dimensional) negative definite subspaces is transitive. Thus, the Riemannian symmetric space of $\SO_0(p,q+1)$ can be described as
\begin{align*}
X_{\SO_0(p,q+1)}:=&\SO_0(p,q+1)\big/\big(\SO(p)\times\SO(q+1)\big)\\
=&\big\{\text{$(q+1)$-dimensional negative definite subspaces of $\R^{p,q+1}$}\big\}\\
=&\big\{\text{timelike totally geodesic $q$-spheres in $\H^{p,q}$}\big\}.
\end{align*}

The \emph{pseudosphere} $\mathbb{S}^{q,p}$ is the counter part of pseudo-hyperbolic spaces defined by
$$
\mathbb{S}^{q,p}:=\big\{x\in\R^{q+1,p}\mid \langle x,x\rangle=1\big\}.
$$
It has curvature $+1$ and is clearly \emph{anti-isometric} to $\H^{p,q}$: there is an obvious diffeomorphism $\H^{p,q}\cong\mathbb{S}^{q,p}$ identifying the metric of the former with $-1$ times that of the latter. We sometimes use $\mathbb{S}^{q,p}$ with metric multiplied by $-1$ as an alternative model for $\H^{p,q}$.

\subsection{Split octonions}
The unit $6$-sphere $\mathbb{S}^6$ is known to carry a natural almost complex structure $\ac{J}$ coming from the algebra $\O$ of octonions. The group $G_2\subset\SO(7)$ of the isometries of $\mathbb{S}^6$ preserving $\ac{J}$, which is also the automorphism group of $\O$, is a maximal compact subgroup of the $14$-dimensional complex simple Lie group $G_2^\C$.
A close analogue holds for the pseudo-hyperbolic space $\H^{4,2}$: it carries an almost complex structure $\ac{J}$ coming from the \emph{split octonions} $\O'$, whose automorphism group 
$G_2'=\Aut(\O')=\Aut(\H^{4,2},\ac{J})\subset\SO_0(4,3)$
is the split real form of $G_2^\C$. We now briefly review the constructions.

The split octonions $\O'$ is the (noncommutative and nonassociative) algebra over $\R$ generated as a vector space by the symbols $1,\ima,\jma,\kma,\lma,\ima\lma,\jma\lma,\kma\lma$, such that $1$ is the identity and the products of the other basis vectors are given by Table \ref{table_multiplication}.
\begin{table}[ht]
	\scalebox{0.9}{
	\begin{TAB}(c)[4pt]{|c|c|c|c|c|c|c|c|}{|c|c|c|c|c|c|c|c|}
		&$\ima$&$\jma$&$\kma$&$\lma$&$\ima\lma$&$\jma\lma$&$\kma\lma$\\
		$\ima$&$-1$&$\kma$&$-\jma$&$\ima\lma$&$-\lma$&$-\kma\lma$&$\jma\lma$\\
		$\jma$&$-\kma$&$-1$&$\ima$&$\jma\lma$&$\kma\lma$&$-\lma$&$-\ima\lma$\\
		$\kma$&$\jma$&$-\ima$&$-1$&$\kma\lma$&$-\jma\lma$&$\ima\lma$&$-\lma$\\
		$\lma$&$-\ima\lma$&$-\jma\lma$&$-\kma\lma$&$1$&$-\ima$&$-\jma$&$-\kma$\\
		$\ima\lma$&$\lma$&$-\kma\lma$&$\jma\lma$&$\ima$&$1$&$\kma$&$-\jma$\\
		$\jma\lma$&$\kma\lma$&$\lma$&$-\ima\lma$&$\jma$&$-\kma$&$1$&$\ima$\\
		$\kma\lma$&$-\jma\lma$&$\ima\lma$&$\lma$&$\kma$&$\jma$&$-\ima$&$1$
	\end{TAB}
}
\caption{Multiplication table of $\O'$.}
	\label{table_multiplication}
\end{table}
Thus, $\O'$ is an extension of the algebras of complex numbers $\C=\R\oplus\R\ima$ and quaternions $\mathbf{H}=\R\oplus\R\ima\oplus\R\jma\oplus\R\kma$. The conjugation $x\mapsto \cj{x}$ (i.e.\ the linear involution fixing $1$ and sending any other basis vector to the opposite) and the notion of real and imaginary parts $\re(x):=\frac{x+\cj{x}}{2},\im(x):=\frac{x-\cj{x}}{2}$
naturally extend from $\C$ and $\mathbf{H}$ to $\O'$, and so does the quadratic form 
$$
\langle x,y\rangle:=\re(x\cj{y})=\tfrac{1}{2}(x\cj{y}+y\cj{x}).
$$
The conjugation and the quadratic form are still compatible with the multiplication in the sense that
\begin{equation}\label{eqn_compositionalgebra}
	\cj{xy}=\cj{y}\ \cj{x},\quad
	\langle xy,xy\rangle=\langle x,x\rangle\langle y,y\rangle. 
\end{equation}
However, $\langle\sth,\sth\rangle$ is not positive definite on $\O'$: we have
$$
\langle x,x\rangle=x_0^2+x_1^2+x_2^2+x_3^2-x_4^2-x_5^2-x_6^2-x_7^2.
$$
for $x=x_0+x_1\ima+x_2\jma+x_3\kma+x_4\lma+x_5\ima\lma+x_6\jma\lma+x_7\kma\lma\in \O'$. It follows that the imaginary split octonions 
$$
\im\O':=\big\{x\in\O'\mid \re(x)=0\big\}=1^\perp,
$$
endowed with the restriction of the quadratic form $\langle\sth,\sth\rangle$,
identifies with $\R^{3,4}$.

\subsection{The cross product on $\R^{3,4}$}\label{subsec_crossproduct}
The Lie group $G_2'$ can be defined as the group $\Aut(\O')$ of algebra automorphisms of $\O'$. However, instead of working with $\O'$, we shall rather focus on $\im\O'\cong \R^{3,4}$ and the cross product ``$\times$'' defined on it by
$$
x\times y:=\im(xy)=\tfrac{1}{2}(xy-yx)\quad \text{ for all } x,y\in\im\O'
$$
(the last equality is because $\cj{x}=-x$ for any $x\in\im\O'$). This is a skew-symmetric bilinear operation extending the familiar cross product on the Euclidean $3$-space and satisfying the same identities as the latter with respect to the quadratic form $\langle\sth,\sth\rangle$ (see Lemma \ref{lemma_crossproduct} below).

It is useful to notice that for any $x,y\in\im\O'$, the following conditions are equivalent:
$$
\langle x,y\rangle=0\ \Leftrightarrow\ xy=-yx\ \Leftrightarrow\ \re(xy)=0\ \Leftrightarrow\ x\times y=xy.
$$
As a consequence, we get from Table \ref{table_multiplication} the table
 below for the cross product of the basis vectors $\ima,\jma,\cdots,\kma\lma\in\im\O'\cong\R^{3,4}$, re-denoted here by
 $e_1,\cdots,e_7$.
\begin{table}[ht]
	\scalebox{0.9}{
\begin{TAB}(c)[4pt]{|c|c|c|c|c|c|c|c|}{|c|c|c|c|c|c|c|c|}
	&$e_1$&$e_2$&$e_3$&$e_4$&$e_5$&$e_6$&$e_7$\\
	$e_1$&$0$&$e_3$&$-e_2$&$e_5$&$-e_4$&$-e_7$&$e_6$\\
	$e_2$&$-e_3$&$0$&$e_1$&$e_6$&$e_7$&$-e_4$&$-e_5$\\
	$e_3$&$e_2$&$-e_1$&$0$&$e_7$&$-e_6$&$e_5$&$-e_4$\\
	$e_4$&$-e_5$&$-e_6$&$-e_7$&$0$&$-e_1$&$-e_2$&$-e_3$\\
	$e_5$&$e_4$&$-e_7$&$e_6$&$e_1$&$0$&$e_3$&$-e_2$\\
	$e_6$&$e_7$&$e_4$&$-e_5$&$e_2$&$-e_3$&$0$&$e_1$\\
	$e_7$&$-e_6$&$e_5$&$e_4$&$e_3$&$e_2$&$-e_1$&$0$
\end{TAB}
}
\caption{The cross product on $\R^{3,4}$}
\label{table_crossproduct}
\end{table}
 We may take this table as the definition of the cross product on $\R^{3,4}$ without involving the algebra $\O'$. The main properties of this cross product are:
\begin{lemma}\label{lemma_crossproduct}
	For any $x,y,z\in\R^{3,4}$, we have
	\begin{align}
		\langle x\times y,x\rangle&=0,\label{eqn_crossproduct1}\\
		\langle x\times y,z\rangle&=\langle y\times z,x\rangle,\label{eqn_crossproduct2}\\
		\langle x\times y,x\times y\rangle&=\langle x,x\rangle\langle y,y\rangle-\langle x,y\rangle^2,\label{eqn_crossproduct3}\\
		x\times(x\times y)&=-\langle x,x\rangle y+\langle x,y\rangle x.\label{eqn_crossproduct4}
	\end{align}
Moreover, if $x$ is orthogonal to both $y$ and $z$, then  
	we further have
	\begin{equation}\label{eqn_crossproduct5}
	x\times(y\times z)=-y\times (x\times z).
	\end{equation}
\end{lemma}
\begin{proof}
These actually hold in a more general setting: Given an algebra $\mathbb{A}$ (with identity) over $\R$ endowed with a nondegenerate quadratic form $\langle\sth,\sth\rangle$ satisfying \eqref{eqn_compositionalgebra}, one may consider the cross product $\times$ on $\im\mathbb{A}:=1^\perp$ defined in the same way as here. In the survey \cite{karigiannis}, it is shown in Lemma 3.39, Prop.\ 3.33 (see also Def.\ 3.34 and Eq.(3.45)) and Cor.\ 3.49, respectively, that \eqref{eqn_crossproduct1}, \eqref{eqn_crossproduct2} and \eqref{eqn_crossproduct3} hold in this setting (although $\langle\sth,\sth\rangle$ is assumed to be positive define in \cite{karigiannis}, nondegeneracy is adequate for these identities). Also, \cite[Cor.\ 3.49]{karigiannis} shows that for any $x,y,z\in\im\mathbb{A}$ we have
$$
x\times(y\times z)=-\langle x,y\rangle z+\langle x,z\rangle y-\tfrac{1}{2}[x,y,z],
$$
where the \emph{associator} $[x,y,z]:=(xy)z-x(yz)$ is known to be skew-symmetric in $x$, $y$ and $z$ (this means that $\mathbb{A}$ is an \emph{alternative} algebra; see \cite[Prop.\ 3.33]{karigiannis}). This implies \eqref{eqn_crossproduct4} and \eqref{eqn_crossproduct5}.
\end{proof}

Finally, note that the cross product is encoded in the $3$-form $\varphi\in\bigwedge^3(\R^{3,4})^*$ defined by
$$
\varphi(x,y,z):=\langle x\times y,z\rangle
$$
(it is skew-symmetric in $x$, $y$ and $z$ by \eqref{eqn_crossproduct2}). Using Table \ref{table_crossproduct} and noting that $\langle e_i,e_i\rangle=1$ (resp.\ $-1$) for $i=1,2,3$ (resp.\ $4,5,6,7$) and $\langle e_i,e_j\rangle=0$ for $i\neq j$, we get the explicit expression
\begin{equation}\label{eqn_3form}
\varphi=e^{123}-e^{145}+e^{167}-e^{246}-e^{257}-e^{347}+e^{356},
\end{equation}
where $(e^1,\cdots,e^7)$ is the basis of $(\R^{3,4})^*$ dual to $(e_1,\cdots,e_7)$ and we write $e^{ijk}:=e^i\wedge e^j\wedge e^k$.

\subsection{The Lie group $G_2'$}
It is known that the quadratic form $\langle\sth,\sth\rangle$ and the cross product $\times$ on $\im\O'$ completely capture the algebraic structure of $\O'$ in the sense that 
the natural homomorphism from $\Aut(\O')$ to the group of linear transformations of $\im\O'$ preserving $\langle\sth,\sth\rangle$ and $\times$ is an isomorphism. Therefore, we may henceforth forget about $\O'$ and use the following alternative definition of $G_2'$:
\begin{definition}\label{def_g2}
$G_2'$ is the group of all $A\in\mathrm{O}(3,4)$ whose action on $\R^{3,4}$ preserves the cross product $\times$ (or equivalently, preserves the $3$-form $\varphi$; see also Remark \ref{rk_g22} below). Vectors $\wt{e}_1,\cdots,\wt{e}_7\in\R^{3,4}$ are said to form a \emph{$G_2'$-basis} if they are successively the columns of such an $A$ (or in other words, $\wt{e}_i=A(e_i)$ for the standard basis $(e_1,\cdots,e_7)$ of $\R^{3,4}$).
\end{definition}
The following proposition and corollary give more concrete descriptions of $G_2'$ and its Lie algebra:
\begin{proposition}\label{prop_g2basis}
	Vectors $\wt{e}_1,\cdots,\wt{e}_7\in\R^{3,4}$ are a $G_2'$-basis if and only if they satisfy:
	\begin{enumerate}[label=(\alph*)]
		\item\label{item_lemmag2basis1} $\wt{e}_1$, $\wt{e}_2$ and $\wt{e}_4$ are orthogonal to each other and moreover
		$$
		\langle \wt{e}_1,\wt{e}_1\rangle=\langle \wt{e}_2,\wt{e}_2\rangle=1,\quad \langle \wt{e}_4,\wt{e}_4\rangle=-1,\quad \langle \wt{e}_1\times \wt{e}_2,\wt{e}_4\rangle=0;
		$$
		\item\label{item_lemmag2basis2} the other $\wt{e}_i$'s are determined from $\wt{e}_1$, $\wt{e}_2$ and $\wt{e}_4$ by
		$$
		\wt{e}_3=\wt{e}_1\times \wt{e}_2,\quad \wt{e}_5=\wt{e}_1\times \wt{e}_4,\quad \wt{e}_6=\wt{e}_2\times \wt{e}_4,\quad \wt{e}_7=\wt{e}_3\times \wt{e}_4.
		$$
	\end{enumerate}
\end{proposition}
\begin{proof}
Being a $G_2'$-basis means that the endomorphism $A$ of $\R^{3,4}$ with $A(e_i)=\wt{e}_i$ satisfies:
	\begin{itemize}
		\item $\langle\sth,\sth\rangle$-preserving $\Leftrightarrow$ $\langle \wt{e}_i,\wt{e}_i\rangle=1$ (resp.\ $-1$) for $i=1,2,3$ (resp.\ $4,5,6,7$) and $\langle\wt{e}_i,\wt{e}_j\rangle=0$ for $i\neq j$;
		\item $\times$-preserving $\Leftrightarrow$ the cross product $\wt{e}_i\times \wt{e}_j$ matches Table \ref{table_crossproduct} (with each $e_i$ replaced by $\wt{e}_i$).
	\end{itemize}
	These properties clearly imply conditions \ref{item_lemmag2basis1} and \ref{item_lemmag2basis2}. Conversely, if  \ref{item_lemmag2basis1} and \ref{item_lemmag2basis2} hold, it is straightforward to verify the first bullet point and then the second by using the identities in Lemma \ref{lemma_crossproduct}.
\end{proof}
\begin{corollary}\label{coro_g2}
$G_2'$ is connected (hence contained in the identity component $\SO_0(3,4)$ of $\mathrm{O}(3,4)$) and its Lie algebra $\frak{g}_2'\subset\frak{so}(2,3)$ consists of matrices of the following form, with $a_1,\cdots,a_6, b_1,\cdots,b_6,c_1,c_2\in\R$:
$$
\scalebox{0.9}{$
\left(
\begin{array}{ccc|cccc}
	0&a_1&a_2&b_4&c_1&b_1&b_2\\[0.1cm]
	-a_1&0&a_3&b_5&b_1-b_6&c_2&b_3\\[0.1cm]
	-a_2&-a_3&0&b_6&b_2+b_5&b_3-b_4&-c_1-c_2\\
	\midrule
	b_4&b_5&b_6&0&a_4&a_5&a_6\\[0.2cm]
	c_1&b_1-b_6&b_2+b_5&-a_4&0&a_1+a_6&a_2-a_5\\[0.1cm]
	b_1&c_2&b_3-b_4&-a_5&-a_1-a_6&0&a_3+a_4\\[0.1cm]
	b_2&b_3&-c_1-c_2&-a_6&-a_2+a_5&-a_3-a_4&0
\end{array}
\right)$}.
$$
\end{corollary}
\begin{proof}
By Proposition \ref{prop_g2basis}, $G_2'$ is diffeomorphic to the space of triples $(\wt{e}_1,\wt{e}_2,\wt{e}_4)$ in $\R^{3,4}$ satisfying Condition \ref{item_lemmag2basis1}, which is a sphere bundle over the space of orthonormal pairs $(\wt{e}_1,\wt{e}_2)$. The latter space is in turn an $\mathbb{S}^{1,4}$-bundle over $\mathbb{S}^{2,4}$, hence connected. The description of $\frak{g}_2'$ can be deduced from Proposition \ref{prop_g2basis} by differentiating the relations therein at the identity $(\wt{e}_1,\cdots,\wt{e}_7)=(e_1,\cdots,e_7)$. We omit the detailed calculations.
\end{proof}

\begin{remark}\label{rk_g21}
The compact group $G_2$ is given by an almost identical construction as above except for some changes of signs. Namely, to get $G_2$ instead of $G_2'$, it suffices to replace the split octonions $\O'$ by the octonions $\O$, whose multiplication is given by Table \ref{table_multiplication} with the signs in the lower-right $4\times 4$ block inverted; or we may consider the corresponding cross product on the Euclidean space $\mathbb{E}^7\cong \im\O$. 
\end{remark}

\begin{remark}\label{rk_g22}
Although we defined $G_2'$ as the group of linear transformations $A\in\GL(\R^{3,4})$ preserving both the metric $\langle\sth,\sth\rangle$ and the $3$-form $\varphi$, it is a nontrivial fact that if $A$ preserves $\varphi$, then it automatically preserves $\langle\sth,\sth\rangle$. This can be proved by showing directly, with  cumbersome computations (best done with a computer algebra system), that if a matrix $X\in\mathfrak{gl}_7\R$  preserves $\varphi$ infinitesimally, then it has the form given in Corollary \ref{coro_g2}. Alternatively, a conceptual proof of the analogous fact for $G_2$ is given in \cite{bryant_metrics} (see also \cite[Theorem 4.4]{karigiannis}), which can be adapted to $G_2'$. This fact is of fundamental importance in $G_2$-geometry (see \cite[Remark 4.10]{karigiannis}).
\end{remark}

\begin{remark}\label{rk_g23}
By the last remark, $G_2$ and $G_2'$ are the subgroups of $\GL(7,\R)$ preserving two specific $3$-forms on $\R^7$ respectively. This is related to the fact that the $\GL(7,\R)$-action on $\bigwedge^3\R^{7*}$ has exactly two open orbits (see \cite[p.332]{sagerschnig}): those two $3$-forms belong to the two orbits respectively.
\end{remark}

\subsection{Maximal compact subgroup}\label{subsec_maximalcompact}
It is known that $G_2'$ is a simple Lie group, and
if we consider the orthogonal splitting 
$\R^{3,4}=V\oplus W$ from \S \ref{subsec_pseudohyperbolic},
then the intersection (between subgroups of $\SO_0(3,4)$)
$$
K:=G_2'\cap\big(\SO(V)\times \SO(W)\big)
$$ 
is a maximal compact subgroup.
In order to describe $K$ more concretely, we use the $3$-form $\varphi$ from  \S \ref{subsec_crossproduct} to identify every $v\in V$ as a $2$-form on $W$ via the map 
\begin{equation}\label{eqn_mapvw}
	V\to \textstyle{\bigwedge^2}W^*,\quad v\mapsto (v\lrcorner \varphi)|_{W}.
\end{equation}
In particular, the basis vectors $e_1,e_2,e_3$ of $V$ correspond to the $2$-forms
\begin{equation}\label{eqn_e1e2e3}
	\left\{\hspace{-0.1cm}
	\begin{array}{l}
		e_1\mapsto -e^{45}+e^{67},\\ 
		e_2\mapsto -e^{46}-e^{57},\\ 
		e_3\mapsto -e^{47}+e^{56},
	\end{array}
	\right.
\end{equation}
where we view $(e^4,e^5,e^6,e^7)$ as a basis of $W^*$ and set $e^{ij}:=e^i\wedge e^j$. By a well known linear-algebraic construction, the Hodge star $\star:\bigwedge^\sth W^*\to \bigwedge^\sth W^*$ (defined under the Euclidean metric $-\langle\sth,\sth\rangle|_W$) is an involution on $\bigwedge^2W^*$ and the $2$-forms in \eqref{eqn_e1e2e3} are an orthonormal basis for the $(-1)$-eigenspace $\bigwedge^2_-W^*\subset\bigwedge^2W^*$ of $\star$ (whose elements are known as \emph{anti-self-dual} $2$-forms).
Therefore, the map \eqref{eqn_mapvw} induces a natural identification 
$V\cong\bigwedge^2_-W^*$. It enables us to describe $K$ as follows:

\begin{proposition}\label{prop_maximalcompact}
	Consider the homomorphism $\varrho:\SO(W)\to\SO(V)$ given by the natural action of $\SO(W)$ on $V\cong \bigwedge^2_-W^*$. Then we have
	$$
	K=\big\{(\varrho(A_W),A_W)\mid A_W\in\SO(W)\big\}\cong \SO(4).
	$$
	%
\end{proposition}
\begin{proof}
	By definition, $A\in \mathrm{O}(3,4)$ is in $G_2'$ if and only if 
	\begin{equation}\label{eqn_proofmaximalcompact}
		\varphi(A(e_i),A(e_j),A(e_k))=\varphi(e_i,e_j,e_k)\ \text{ for all } 1\leq i<j<k\leq 7.
	\end{equation}
	Now assume $A=(A_V,A_W)\in\SO(V)\times\SO(W)$. By the expression \eqref{eqn_3form} of $\varphi$, we infer that 
	\begin{itemize}
		\item if $(i,j,k)=(1,2,3)$, then both sides of \eqref{eqn_proofmaximalcompact} are equal to $1$;
		\item if $i\leq 3$ and $j,k\geq4$, then the left-hand side of \eqref{eqn_proofmaximalcompact} can be rewritten as $\varphi(A_V(e_i),A_W(e_j),A_W(e_k))$;
		\item if $(i,j,k)$ is in neither of the above cases, then both sides of \eqref{eqn_proofmaximalcompact} vanish.
	\end{itemize}
As a consequence, such an $A$ belongs to $K$ if and only if
	\begin{equation}\label{eqn_proofmaximalcompact2}
		\varphi(A_V(v),A_W(w_1),A_W(w_2))=\varphi(v,w_1,w_2)\ \text{ for all } v\in V,\ w_1,w_2\in W.
	\end{equation}
	On the other hand, by the definition \eqref{eqn_mapvw} of the identification $V\cong\bigwedge^2_-W^*$, we have $A_V=\varrho(A_W)$ if and only if
	$\big(A_V(v)\lrcorner\varphi\big)\big|_W =A_W^*(v\lrcorner\varphi)|_W$ for all $v\in V$,
	or equivalently,
	$$
	\varphi(A_V(v),w_1,w_2)=\varphi(v,A_W^{-1}(w_1),A_W^{-1}(w_2))\ \text{ for all } v\in V,\ w_1,w_2\in W.
	$$
	This is equivalent to \eqref{eqn_proofmaximalcompact2}, so the required statement follows.
\end{proof}
\begin{remark}
	There are well known exceptional isomorphisms 
	$$
	\SO(4)\cong\big(\SU(2)\times\SU(2)\big)\big/\{\pm I\}\cong \SU(2)\rtimes \SO(3).
	$$ 
	The first one results from quaternions, while the second is given by the diagonal $\SO(3)\cong\SU(2)/\{\pm I\}$ in $(\SU(2)\times\SU(2))/\{\pm I\}$ and the first component $\SU(2)$. The homomorphism $\varrho$ in Proposition \ref{prop_maximalcompact} is essentially the projection $\SO(4)\cong\SU(2)\rtimes\SO(3)\to\SO(3)$ whose kernel is $\SU(2)$.
\end{remark}

\begin{remark}\label{rk_p}
By using the matrix representation of $\frak{g}_2'$ in Corollary \ref{coro_g2}, it can be checked that
the subspace $\frak{p}\subset\frak{g}_2'$ in the Cartan decomposition $\frak{g}_2'=\frak{k}\oplus\frak{p}$ (where $\frak{k}\cong\frak{so}(4)$ is the Lie algebra of $K$) is isomorphic, as a $K$-module, to the subspace $\Hom_0(W,V):=\big\{T\in\Hom(W,V)\,\big|\, \Tr_{12}(T)=0\big\}$ of
$\Hom(W,V)\cong\Hom(W,\bigwedge^2_-W^*)= W^*\otimes\bigwedge^2_-W^*\subset W^*\otimes W^*\otimes W^*$, where $\Tr_{12}$ is defined by
$$
\Tr_{12}:W^*\otimes W^*\otimes W^*\to W^*,\quad T(\sth,\sth,\sth)\mapsto \sum_{i=4}^7T(e_i,e_i,\sth)
$$
(the contraction of the first two slots of any $3$-tensor on $W$ with the metric $-\langle\sth,\sth\rangle|_W$). 
\end{remark}

\subsection{The almost complex $\H^{4,2}$ and the symmetric space $X_{G_2'}$}\label{subsec_almostcomplex}
The cross product on $\R^{3,4}$ induces an almost complex structure $\ac{J}$ on the pseudosphere $\mathbb{S}^{2,4}=\{x\in\R^{3,4}\mid \langle x,x\rangle=1\}
$ by
$$
\ac{J}(X):=x\times X\quad \text{ for all } x\in \mathbb{S}^{2,4},\ X\in \T_x\mathbb{S}^{2,4}\cong x^\perp.
$$
Using equations \eqref{eqn_crossproduct2} and \eqref{eqn_crossproduct4} in Lemma \ref{lemma_crossproduct}, one may check that $\ac{J}$ is indeed an almost complex structure (i.e.\ $\ac{J}$ preserves $x^\perp$ and satisfies $\ac{J}^2=-\id$) and is \emph{orthogonal} with respect to the pseudo-Riemannian metric in the sense that it is an isometry of each $\T_x\mathbb{S}^{2,4}$. Moreover, it is clear that the antipodal map $\tau(x)=-x$ of $\mathbb{S}^{4,2}$ is $\ac{J}$-anti-holomorphic, namely $\dif\tau\circ\ac{J}=-\ac{J}\circ\dif\tau$.


It is known that $\ac{J}$ is non-integrable and has the property (c.f.\ \S \ref{subsec_jholominimal} below)
$$
(\nabla_X\ac{J})Y=-(\nabla_Y\ac{J})X\quad \text{ for all } X,Y\in\T \mathbb{S}^{2,4},
$$ 
 where $\nabla$ is the Levi-Civita connection of $\mathbb{S}^{2,4}$. In fact, it can be shown by computations that
$$
(\nabla_X\ac{J})Y=(X\times Y)^{x^\perp}\quad\text{ for all } x\in\mathbb{S}^{2,4}\text{ and } X,Y\in \T_x\mathbb{S}^{4,2}\cong x^\perp,
$$
where $(y)^{x^\perp}:=y-\langle x,y\rangle x$ is the $x^\perp$-component of $y\in\R^{3,4}$ under the splitting $\R^{3,4}=x\R\oplus x^\perp$.

We henceforth consider $(\mathbb{S}^{2,4},\ac{J})$ with metric multiplied by $-1$ and call it \emph{the almost complex $\H^{4,2}$} (see the last paragraph of \S \ref{subsec_pseudohyperbolic}).
In this model of $\H^{4,2}$, timelike totally geodesic $2$-spheres $S\subset\H^{4,2}$ correspond to $3$-dimensional positive definite subspaces $V'\subset\R^{3,4}$ via the relation $S=V'\cap \H^{4,2}$. The almost complex structure $\ac{J}$ singles out a subclass of such spheres, namely the $\ac{J}$-holomorphic ones:
\begin{lemma}\label{lemma_sphere}
For any timelike totally geodesic $2$-sphere $S=V'\cap\H^{4,2}$, where $V'\subset\R^{3,4}$ is a $3$-dimensional positive definite subspace, the following conditions are equivalent to each other: 
\begin{enumerate}[label=(\roman*)]
	\item\label{item_lemmasphere1} $\T_x S\subset\T_x\H^{4,2}$ is preserved by $\ac{J}$ for some $x\in S$;
	\item\label{item_lemmasphere2} $S$ is a $\ac{J}$-holomorphic curve (i.e.\ $\T_x S$ is preserved by $\ac{J}$ for all $x\in S$);
	\item\label{item_lemmasphere3} $V'$ is closed under the cross product operation.
\end{enumerate}
\end{lemma}
\begin{proof}
The implication ``\ref{item_lemmasphere2}$\Rightarrow$\ref{item_lemmasphere1}'' is trivial and ``\ref{item_lemmasphere3}$\Rightarrow$\ref{item_lemmasphere2}'' follows easily from the definition of $\ac{J}$. To show ``\ref{item_lemmasphere1}$\Rightarrow$\ref{item_lemmasphere3}'', we assume \ref{item_lemmasphere1} and take an orthonormal basis $(y,z)$ of $\T_xS\cong V'\cap x^\perp$ such that $z=\ac{J}y=x\times y$ and $y=-\ac{J}z=-x\times z$. Then $V'$ is spanned by $(x,y,z)$, whereas $x\times y=z,x\times z=-y$ and $y\times z=-(x\times z)\times z=x$ (see \eqref{eqn_crossproduct4} in Lemma \ref{lemma_crossproduct}) are all in $V'$, hence \ref{item_lemmasphere3} holds.
\end{proof}
The most obvious instance of $V'$ satisfying \ref{item_lemmasphere3} is the above $V=\R e_1\oplus\R e_2\oplus \R e_3$, whose stabilizer in $G_2'$ is the maximal compact subgroup $K\cong\SO(4)$. Meanwhile, we have:
\begin{proposition}
$G_2'$ acts transitively on the space of $3$-dimensional positive definite subspaces of $\R^{3,4}$ that are closed under the cross product.
\end{proposition}
\begin{proof}
Let $V'$ be any such subspace. Pick an orthonormal pair of vectors $(\wt{e}_1,\wt{e}_2)$ in $V'$ and a vector $\wt{e}_4\in (V')^\perp$ with $\langle \wt{e}_4,\wt{e}_4\rangle=-1$. Then $\wt{e}_1$, $\wt{e}_2$ and $\wt{e}_4$ satisfy Condition \ref{item_lemmag2basis1} in Proposition \ref{prop_g2basis}, hence there exists $A\in G_2'$ sending the basis $e_1,e_2,e_3$ of $V$ to the basis $\wt{e}_1,\wt{e}_2,\wt{e}_3:=\wt{e}_1\times \wt{e}_2$ of $V'$. Therefore, $V'$ is conjugate to $V$ by the $G_2'$-action.
\end{proof}

Thus, we obtain the following description for the symmetric space of $G_2'$:
\begin{corollary}\label{coro_xg2}
The Riemannian symmetric space $X_{G_2'}:=G_2'/K$ is naturally identified as
\begin{align*}
	X_{G_2'}
	=&\big\{\text{3d positive definite subspaces of $\R^{3,4}$ closed under cross product}\big\}\\
	=&\big\{\text{timelike totally geodesic $\ac{J}$-holomorphic $2$-spheres in $\H^{4,2}$}\big\}
\end{align*}
and is a totally geodesic submanifold of $X_{\SO_0(3,4)}=\big\{\text{timelike totally geodesic $2$-spheres in $\H^{4,2}$}\big\}$.
\end{corollary}
Here, the claim that $X_{G_2'}$ is totally geodesic in $X_{\SO_0(3,4)}$ can be shown by using the criterion \cite[Chapt.\ IV, Thm.\ 7.2]{helgason} for totally geodesic submanifolds in symmetric spaces. 

\subsection{A matrix representation of $\frak{g}_2^\C$}\label{subsec_g2complex}
Let us turn to the complex Lie algebra $\frak{g}_2^\C:=\frak{g}_2'\otimes\C$.
In view of the matrix representation of $\frak{g}_2'$ in Corollary \ref{coro_g2}, we may identify $\frak{g}_2^\C$ as the algebra of complex matrices of the same form. However, with this form, it is difficult to single out a Cartan subalgebra and a principal $3$-dimensional subalgebra, which are important for Higgs bundles. We now exhibit another matrix representation which makes these subalgebras explicit (at the cost of making the real form $\frak{g}_2'$ implicit and dependent on a hermitian metric):
\begin{proposition}\label{prop_g2complex}
Consider the complex $3$-form $\varpi$ on $\C^7$ given by
$$
\varpi:=\ima\big(\vae^{147}+\vae^{246}-\vae^{345}-\scalebox{0.8}{$2\sqrt{2}$}\,\vae^{156}-\tfrac{1}{\sqrt{2}}\vae^{237}\big)\in\textstyle{\bigwedge^3}\C^{7*}
$$
(the factor $\ima$ is inessential here but needed for Proposition \ref{prop_g2realform}), where we let $(\vae_1,\cdots,\vae_7)$ be the standard basis of $\C^7$, $(\vae^1,\cdots,\vae^7)$ be the dual basis of $\C^{7*}$, and write $\vae^{ijk}:=\vae^i\wedge\vae^j\wedge\vae^k$. Then
an endomorphism $X\in\frak{gl}_7\C$ of $\C^7$ preserves $\varpi$ infinitesimally if and only if it has the form
\begin{equation}\label{eqn_g2complex}
	X=
\scalebox{0.9}{$
\left(
\begin{array}{ccc|c|ccc}
	t_1+t_2&x_1&x_2&x_3&x_4&x_5&0\\[0.25cm]
	y_1&t_1&x_6&\scalebox{0.8}{$-2\sqrt{2}$}\,x_2&\scalebox{0.8}{$-\sqrt{2}$}\,x_3&0&x_5\\[0.25cm]
	y_2&y_6&t_2&\scalebox{0.8}{$2\sqrt{2}$}\,x_1&0&\scalebox{0.8}{$-\sqrt{2}$}\,x_3&-x_4\\[0.1cm]
	\midrule
	y_3&-\tfrac{y_2}{\sqrt{2}}&\tfrac{y_1}{\sqrt{2}}&0&\scalebox{0.8}{$2\sqrt{2}$}\,x_1&\scalebox{0.8}{$2\sqrt{2}$}\,x_2&x_3\\
	\midrule
	y_4&-\tfrac{y_3}{2\sqrt{2}}&0&\tfrac{y_1}{\sqrt{2}}&-t_2&x_6&-x_2\\[0.25cm]
	y_5&0&-\tfrac{y_3}{2\sqrt{2}}&\tfrac{y_2}{\sqrt{2}}&y_6&-t_1&x_1\\[0.25cm]
	0&y_5&-y_4&y_3&-y_2&y_1&-t_1-t_2
\end{array}
\right)
$}
\end{equation}
with $t_1,t_2,x_1,\cdots,x_6,y_1,\cdots,y_6\in\C$. 
 The Lie algebra formed by all such $X$'s is isomorphic to $\frak{g}_2^\C$ and is contained in the Lie algebra $\frak{so}(Q)\cong\frak{so}(7,\C)$ infinitesimally preserving the quadratic form
$$
Q:=
\scalebox{0.8}{$
\mat{&&&&&&-1\\&&&&&1&\\&&&&-1&&\\&&&1&&&\\&&-1&&&&\\&1&&&&&\\-1&&&&&&}
$}.
$$
\end{proposition}
\begin{proof}
The first statement can be checked by direct computations with a computer algebra software. Alternatively, we may use the explicit conjugation in the proof of Proposition \ref{prop_g2realform} below to bring the $3$-form $\varphi$ considered earlier to $\varpi$ and check that the conjugation brings  the matrices in Corollary \ref{coro_g2} to the ones here. We omit the details. The fact that $X\in\frak{so}(Q)$ is checked by a simple computation. Finally, the statement that the Lie algebra formed by these $X$'s is isomorphic to $\frak{g}_2^\C:=\frak{g}_2'\otimes\C$ also follows from  Proposition \ref{prop_g2realform} below, where we exhibit explicit real forms isomorphic to $\frak{g}_2'$.
\end{proof}
In this matrix representation of $\frak{g}_2^\C$, the subspace $\frak{h}$ of diagonal elements is a Cartan subalgebra. It is easy to find out explicitly the corresponding root system and Chevalley basis, which we do not specify here. An $\frak{h}$-principal $3$-dimensional subalgebra $\frak{s}\subset\frak{g}_2^\C$ (see \cite{hitchin_lie,labourie_cyclic}) is spanned by
$$
\wt{e}=
\scalebox{0.85}{$
\mat{0&&&&&&\\1&0&&&&&\\&1&0&&&&\\&&\tfrac{1}{\sqrt{2}}&0&&&\\&&&\tfrac{1}{\sqrt{2}}&0&&\\&&&&1&0&\\&&&&&1&0}
$},\quad
a=
\scalebox{0.85}{$
\mat{3&&&&&&\\&2&&&&&\\&&1&&&&\\&&&0&&&\\&&&&\hspace{-0.2cm}-1&&\\&&&&&\hspace{-0.2cm}-2&\\&&&&&&\hspace{-0.2cm}-3}	
$},\quad
e=
\scalebox{0.85}{$
\mat{0&3&&&&&\\&0&5&&&&\\&&0&\hspace{-0.15cm}\scalebox{0.8}{$6\sqrt{2}$}&&&\\&&&0&\hspace{-0.15cm}\scalebox{0.8}{$6\sqrt{2}$}&&\\&&&&0&5&\\&&&&&0&3\\&&&&&&0}
$}.
$$

As fundamental facts in the construction of Hitchin components, to any principal $3$-dimensional subalgebra is associated an involution $\sigma$ of $\frak{g}_2^\C$, and every Cartan involution $\rho$ of $\frak{g}_2^\C$ which commutes with $\sigma$ determines an anti-linear involution $\lambda=\sigma\circ\rho$ whose fixed point set is a split real form  (see \cite{hitchin_lie,labourie_cyclic}). Furthermore, those $\rho$ arising from \emph{cyclic} Higgs bundles preserve $\frak{h}$, and such a $\rho$ is called an \emph{$\frak{h}$-Cartan involution}. For the above $\frak{s}$, the associated $\sigma$ is the usual one $\sigma(x)=-\wh{x}$, where $\wh{x}$ denotes the transpose of $x$ with respect to the diagonal from lower-left to upper-right. The following proposition characterizes all $\frak{h}$-Cartan involutions commuting with $\sigma$ and identifies the resulting split real forms with $\frak{g}_2'$ in a concrete way:
\begin{proposition}\label{prop_g2realform}
	Let $H$ be a hermitian metric on $\C^7$ and $\rho(x)=-x^*$ be the Cartan involution of $\frak{sl}_7\C$ induced by $H$ (which is also the Cartan involution of  the $\frak{so}(Q)$ in Proposition \ref{prop_g2complex}), where $x^*=\cj{H}{}^{-1}\transp{\cj{x}}\cj{H}$ denotes the $H$-adjoint of $x$. Then $\rho$ preserves both the above $\frak{g}_2^\C$ and its Cartan subalgebra $\frak{h}$, and furthermore commutes with $\sigma$, if and only if $H$ has the form
	$$
	H=\diag\big(\tfrac{1}{h_3}\,,\ \tfrac{1}{h_2}\,,\ \tfrac{1}{h_1}\,,\ 1\,,\ h_1\,,\ h_2\,,\ h_3\big)\ \text{ with }h_3=\tfrac{h_1h_2}{4}.
	$$
Moreover, given such an $H$, consider the anti-linear involution $\Lambda$ of $\C^7$ defined by 
	$$
	\Lambda(z_1,\cdots,z_7)=\big(h_3\cj{z}_7\,,\ h_2\cj{z}_6\,,\ h_1\cj{z}_5\,,\ \cj{z}_4\,,\ \tfrac{1}{h_1}\cj{z}_3\,,\ \tfrac{1}{h_2}\cj{z}_2\,,\ \tfrac{1}{h_3}\cj{z}_1\big).
	$$
	Then $Q$ and $\varpi$ from Proposition \ref{prop_g2complex} restrict to a real quadratic form and a real $3$-form on $\Fix(\Lambda)\cong\R^7$, respectively, and $(\Fix(\Lambda), Q,\varpi)$ is isomorphic to $\R^{3,4}$ endowed with the $3$-form $\varphi$ (see \S \ref{subsec_crossproduct}). In particular, $\big\{X\in\frak{g}_2^\C\mid \text{$X$ preserves $\Fix(\Lambda)$}\big\}$ is a real form of $\frak{g}_2^\C$ isomorphic to $\frak{g}_2'$.
\end{proposition}
\begin{proof}
It is easy to see that if $\sigma$ preserves $\frak{h}$ then $H$ is diagonal. The ``if and only if'' statement then follows from a simple computation. For the ``moreover'' statement, we consider the basis $(e_1,\cdots,e_7)$ of $\C^7$ given by the columns of the matrix
$$
B=\tfrac{1}{\sqrt{2}}
\left(
\scalebox{0.85}{$
\begin{array}{c|cc|cc|cc}
&&&&&\ima h_3^\frac{1}{2}&h_3^\frac{1}{2}\\[0.2cm]
&h_2^\frac{1}{2}&-\ima h_2^\frac{1}{2}&&&&\\[0.2cm]
&&&h_1^\frac{1}{2}&-\ima h_1^\frac{1}{2}&&\\
\midrule
\hspace{-0.2cm}\scalebox{0.9}{$\sqrt{2}$}&&&&&&\\
\midrule
&&&h_1^{-\frac{1}{2}}&\ima h_1^{-\frac{1}{2}}&&\\[0.2cm]
&h_2^{-\frac{1}{2}}&\ima h_2^{-\frac{1}{2}}&&&&\\[0.2cm]
&&&&&-\ima h_3^{-\frac{1}{2}}&h_3^{-\frac{1}{2}}
\end{array}
$}
\right),
$$
or in other words, $(e_1,\cdots,e_7):=(\vae_1,\cdots,\vae_7)B$. Clearly, every $e_i$ is fixed by $\Lambda$, and we have $Q(e_i,e_i)=1$ (resp.\ $-1$) for $i=1,2,3$ (resp.\ $4,5,6,7$) and $Q(e_i,e_j)=0$ for $i\neq j$. Also, noting that the basis $(e^1,\cdots,e^7)$ of $\C^{7*}$ dual to $(e_1,\cdots, e_7)$ satisfies $(\vae^1,\cdots,\vae^7)=(e^1,\cdots,e^7)\,\transp{B}$, we obtain by computations that
$\varpi:=\ima\big(\vae^{147}+\vae^{246}-\vae^{345}-\scalebox{0.8}{$2\sqrt{2}$}\,\vae^{156}-\tfrac{1}{\sqrt{2}}\vae^{237}\big)$ is equal to $e^{123}-e^{145}+e^{167}-e^{246}-e^{257}-e^{347}+e^{356}$, which is exactly the expression of $\varphi$ in the standard coordinates of $\R^{3,4}$. This implies the ``moreover'' statement.
\end{proof}

\begin{remark}\label{rk_conjugate}
The matrix representation of $\frak{g}_2^\C$ considered in this subsection is adapted for the Higgs bundles of the form in the introduction, and is not symmetric in the sense that $X\in\frak{g}_2^\C$ does not imply $\transp{X}\in\frak{g}_2^\C$. This is the reason why the coefficients of the $3$-form $\varpi$ lack symmetry. 
Conjugating this $\frak{g}_2^\C$ by a matrix of the form $\diag\big(c,c,c,1,\frac{1}{c},\frac{1}{c},\frac{1}{c}\big)$ (which preserves $Q$), one may bring  it to a symmetric representation as the one in \cite[p.89]{baraglia_thesis}, for which the condition $h_3=\tfrac{h_1h_2}{4}$ on $H$ becomes $h_3=h_1h_2$. One may also use another conjugation of this type to change the $\tfrac{1}{\sqrt{2}}$ in $\wt{e}$ into $1$, so that the corresponding Higgs bundles have a more familiar form.
\end{remark}

\section{Minimal submanifolds}\label{sec_minimal}
In this section, we collect some preliminary definitions and results about minimal submanifolds.
\subsection{Submanifolds of pseudo-Riemannian manifold}\label{subsec_submanifolds}
Let $(M,\ve{h})$ be a pseudo-Riemannian manifold. Given a smooth submanifold $\Sigma\subset M$, we consider the restricted tangent bundle $\T M|_\Sigma$, which is a vector bundle on $\Sigma$ endowed with the metric $\langle\sth,\sth\rangle$ and connection $\nabla$ induced by $\ve{h}$ and its Levi-Civita connection. Also let ${\ca{T}}:=\T \Sigma\subset \T M|_\Sigma$ be the tangent bundle of $\Sigma$ and $\ve{g}:=\langle\sth,\sth\rangle|_{{\ca{T}}}$ be the first fundamental form. $\Sigma$ is called a \emph{spacelike} (or \emph{Riemannian}) submanifold if $\ve{g}$ is positive definite, which in particular implies that we have an orthogonal splitting $\T M|_\Sigma={\ca{T}}\oplus \ca{N}$, where $\ca{N}$ is the normal bundle, namely the 
orthogonal complement of ${\ca{T}}$ in $\T M|_\Sigma$. 

Since $\nabla$ preseves $\langle\sth,\sth\rangle$, it decomposes as follows under the splitting $\T M|_\Sigma=\ca{T}\oplus\ca{N}$:
\begin{equation}\label{eqn_nablaTN}
\nabla=\mat{\nabla^{\ca{T}}&-A^\dagger\\A&\nabla^\ca{N}}.
\end{equation}
$\nabla^{\ca{T}}$ is just the Levi-Civita connection of $\ve{g}$, while $\nabla^\ca{N}$ is called the \emph{normal connection}. The $1$-forms $A\in\Omega^1(\Sigma,\Hom({\ca{T}},\ca{N}))$ and   $A^\dagger\in\Omega^1(\Sigma,\Hom(\ca{N},{\ca{T}}))$ are related by
$\langle A(X)Y,\xi\rangle=\langle Y,A^\dagger(X)\xi\rangle$ (here and below, we let $X,Y$ and $\xi$ be arbitrary sections of ${\ca{T}}$ and $\ca{N}$, respectively). We recall the following familiar names for  variants of $A$ and $A^\dagger$:
\begin{itemize}
	\item The \emph{second fundamental form} of $\Sigma$ is the $\ca{N}$-valued $2$-tensor $\II(\sth,\sth)$ defined by\footnote{Throughout this paper, given a decomposition $V=V_1\oplus V_2\oplus\cdots$ (of a  vector space or vector bundle), we put a superscript such as $V_1$ after $v\in V$ to denote the corresponding component of $v$.} $\II(X,Y):=A(X)Y:=(\nabla_XY)^\ca{N}$. An basic fact is that it is a \emph{symmetric} tensor. In analogy with the theory of Frenet curves, we call every normal vector in the image of $\II(\sth,\sth)$ an \emph{osculation} vector.
	\item The \emph{shape operator} of $\Sigma$ associated to a given $\xi$ is the $(1,1)$-tensor $\ve{S}_\xi(\sth)$ defined by $\ve{S}_\xi(X):=A^\dagger(X)\xi=-(\nabla_X\xi)^{\ca{T}}$. By the relation between $A$ and $A^\dagger$, we have $\langle \II(X,Y),\xi\rangle=\langle Y,\ve{S}_\xi(X)\rangle$.
\end{itemize}

The section $\Tr_{\ve g}\II(\sth,\sth):=\sum_{i=1}^n\II(e_i,e_i)$ of $\ca{N}$ (where $(e_1,\cdots,e_n)$ is any orthonormal local frame of ${\ca{T}}$) is called the \emph{mean curvature field} of $\Sigma$, and $\Sigma$ is said to be a \emph{minimal} submanifold if $\Tr_{\ve g}\II(\sth,\sth)\equiv0$.


\subsection{Jacobi operator}\label{subsec_jacobi}
Given a Riemannian manifold $(\Sigma,\ve{g})$ and a real vector bundle $\ca{E}$ on $\Sigma$ endowed with a metric $\langle\sth,\sth\rangle$ and a metric-preserving connection $\nabla^\ca{E}$, we have a Laplacian operator $\Delta^\ca{E}:C^\infty(\Sigma,\ca{E})\to	C^\infty(\Sigma,\ca{E})$ defined by
$$
	\Delta^\ca{E}\xi:=\sum_{i=1}^n\big(\nabla_{e_i}^\ca{E}\nabla_{e_i}^\ca{E}\xi-\nabla_{\nabla^{\T \Sigma}_{e_i}e_i}^\ca{E}\xi\big),
$$
where $(e_1,\cdots,e_n)$ is an orthonormal local frame of $\T\Sigma$ and $\nabla^{\T \Sigma}$ is the Levi-Civita connection of $\ve{g}$. For any $\xi_1,\xi_2\in C^\infty(\Sigma,\ca{E})$ with $\xi_1$ compactly supported, we have (see e.g.\ \cite[\S 2.1]{berline-getzler-vergne}):
\begin{equation}\label{eqn_laplacian}
	\int_\Sigma\langle \Delta^\ca{E}\xi_1,\xi_2\rangle\dif\vol_\ve{g}=\int_\Sigma\langle \xi_1,\Delta^\ca{E}\xi_2\rangle\dif\vol_\ve{g}=-\int_\Sigma\Tr_{\ve g}\langle\nabla^\ca{E}_\sth\xi_1,\nabla^\ca{E}_\sth\xi_2\rangle\dif\vol_\ve{g},
\end{equation}
where the first equality signifies that $\Delta^\ca{E}$ is a self-adjoint operator.

When $\Sigma$ is a spacelike minimal submanifold in a pseudo-Riemannian manifold $M$ as considered in \S \ref{subsec_submanifolds}, the \emph{Jacobi operator} (or \emph{stability operator}) $L_\Sigma:C^\infty(\Sigma,\ca{N})\to	C^\infty(\Sigma,\ca{N})$ is defined by
$$
L_\Sigma\xi:=\Delta^{\!\ca{N}}\xi+\wt{A}(\xi)+\Tr_{\ve g}\big(R_M(\sth,\xi)\sth\big),
$$
where $\Delta^{\!\ca{N}}$ is the Laplacian given by the first fundamental form $\ve{g}$ and the normal connection $\nabla^\ca{N}$, $R_M$ is the curvature tensor\footnote{As in \cite{anciaux,colding-minicozzi}, we follow do Carmo's sign convention of curvature tensor: $R(X,Y)Z:=\nabla_Y\nabla_XZ-\nabla_X\nabla_YZ+\nabla_{[X,Y]}Z$.} of $M$, and $\wt{A}\in C^\infty(\Sigma,\End(\ca{N}))$ is the \emph{Simons operator} of $\Sigma$ given by
$$
\wt{A}(\xi):=\sum_{i,j}^n\langle \II(e_i,e_j),\xi\rangle \II(e_i,e_j).
$$
The sections of $\ca{N}$ annihilated by $L_\Sigma$ are called \emph{Jacobi fields}.

The self-adjointness of $\Delta^\ca{N}$ impplies that $L_\Sigma$ is a self-adjoint operator as well: $\int_\Sigma\langle L_\Sigma\xi_1,\xi_2\rangle\dif\vol_\ve{g}=\int_\Sigma\langle \xi_1,L_\Sigma\xi_2\rangle\dif\vol_\ve{g}$ if $\xi_1$ has compact support.
Also, since it follows from the relation between $\II$ and $\ve{S}_\xi$ that $\langle \wt{A}(\xi),\xi\rangle=\Tr_{\ve g}\langle \ve{S}_\xi(\sth),\ve{S}_\xi(\sth)\rangle$, by \eqref{eqn_laplacian}, we have
\begin{equation}\label{eqn_variationintegral}
	-\int_\Sigma\langle L_\Sigma\xi,\xi\rangle\dif\vol_\ve{g}=\int_\Sigma\Tr_{\ve g}\Big(\langle \nabla_\sth^\ca{N} \xi,\nabla_\sth^\ca{N}\xi\rangle-\langle \ve{S}_\xi(\sth),\ve{S}_\xi(\sth)\rangle-\langle R_M(\sth,\xi)\sth,\xi\rangle\Big)\dif\vol_\ve{g}
\end{equation}
for any compactly supported section $\xi$ of $\ca{N}$.

\subsection{Variations of minimal submanifolds}\label{subsec_variation}
$L_\Sigma$ is related to variations of the minimal submanifold $\Sigma$. By definition, a \emph{variation} of $\Sigma$ is a smooth one-parameter family $(F_t)_{t\in(-\epsilon,\epsilon)}$  of embeddings $\Sigma\to M$ such that $F_0$ is the inclusion. We use the following standard terminologies concerning such a $(F_t)$:
\begin{itemize}
	\item
	The derivative $\dt{F}_0:=\tfrac{\dif}{\dif t}F_t\big|_{t=0}$, which is a section of $\T M|_\Sigma$, is called the \emph{variational vector field}.
	\item
	$(F_t)$ is said to be \emph{normal} if $\dt{F}_0$ is a section of $\ca{N}\subset\T M|_\Sigma$, and is said to be \emph{compactly supported} if $F_t=F_0$ for all $t$ outside of a compact subset of $\Sigma$.
	\item
	Let $\Sigma_t$ denote the image of $F_t$. When $(F_t)$ is compactly supported, $\Sigma_t$ only differs from $\Sigma$ within a compact set $C\subset M$, so the volume $\vol(\Sigma_t\cap C)$ is well defined and its derivatives in $t$ are independent of the choice of $C$. The second derivative $\tfrac{\dif^2}{\dif t^2}\vol(\Sigma_t\cap K)\big|_{t=0}$ is called the \emph{second variation of volume} associated to the variation $(F_t)$.
\end{itemize}

The following theorem summarizes the fundamental facts about variations that we will need. See e.g.\ \cite{anciaux,colding-minicozzi,spivak,xin} for proofs. Note that although most of the literature only treats the case where the ambient space $M$ is Riemannian, the results and proofs actually hold in the pseudo-Riemannian setting without much modification, and one may even assume that $\Sigma$ is a pseudo-Riemannian submanifold instead of a Riemannian one (see \cite{anciaux}).

\begin{theorem}\label{thm_variation}
	Let $\Sigma$ be a spacelike minimal submanifold in a pseudo-Riemannian manifold $M$ and $(F_t)_{t\in(-\epsilon,\epsilon)}$ be a variation of $\Sigma$.
	\begin{enumerate}[label=(\arabic*)]
		\item\label{item_thmvariation1} If $(F_t)$ is a compactly supported normal variation, then the associated second variation of volume is equal to the integral \eqref{eqn_variationintegral} with $\xi=\dt{F}_0$.
		\item\label{item_thmvariation2} If every $\Sigma_t$ is a minimal submanifold, then the normal component $\dt{F}_0^\ca{N}$ of $\dt{F}_0$ is a Jacobi field.
	\end{enumerate}
\end{theorem}

Part \ref{item_thmvariation1} is known as the \emph{second variation formula}. A simple consequence is:
\begin{proposition}\label{prop_maximal}
	In the setting of Theorem \ref{thm_variation} \ref{item_thmvariation1}, if furthermore  the signature of $M$ is $(p,q)$ with $p=\dim \Sigma$ (or equivalently, the normal bundle of $\Sigma$ is negative definite) and $M$ has negative sectional curvature, then the second variation of volume is negative whenever $\dt{F}_0$ is not identically zero.
\end{proposition}
In particular, every spacelike minimal surface in the pseudo-hyperbolic space $\H^{2,q}$  locally maximizes the area. This explains why they are called \emph{maximal surfaces} in the literature.
\begin{proof}
	Put $\xi:=\dt{F}_0$ and let $(e_1,\cdots,e_n)$ be an orthonormal local frame of $\Sigma$ with respect to the first fundamental form $\ve{g}=\langle\sth,\sth\rangle|_{{\ca{T}}}$. Since $\nabla_{e_i}^\ca{N}\xi$ and $\ve{S}_\xi(e_i)$ are sections of $\ca{N}$ and ${\ca{T}}$, respectively, we have
	\begin{equation}\label{eqn_propmaximal1}
		\Tr_{\ve g}\langle \nabla_\sth^\ca{N} \xi,\nabla_\sth^\ca{N}\xi\rangle=\sum_{i=1}^n\langle \nabla^\ca{N}_{e_i}\xi,\nabla^\ca{N}_{e_i}\xi\rangle\leq0,\quad \Tr_{\ve g}\langle \ve{S}_\xi(\sth),\ve{S}_\xi(\sth)\rangle=\sum_{i=1}^n\langle \ve{S}_\xi(e_i),\ve{S}_\xi(e_i)\rangle\geq0.
	\end{equation}
	On the other hand, at  any point $z\in\Sigma$ with $\xi(z)\neq0$, the sectional curvature of $M$ along the tangent $2$-plane spanned by $e_i$ and $\xi$ is
	$$
	K(e_i\wedge \xi):=\frac{\langle R_M(e_i,\xi)e_i,\xi\rangle}{\langle e_i,e_i\rangle\langle \xi,\xi\rangle-\langle e_i,\xi\rangle^2}=\frac{\langle R_M(e_i,\xi)e_i,\xi\rangle}{\langle \xi,\xi\rangle}.
	$$
	We have $K(e_i\wedge \xi)<0$ and $\langle\xi,\xi\rangle<0$ at $z$ by assumption. It follows that $\langle R_M(e_i,\xi)e_i,\xi\rangle>0$, and hence
	\begin{equation}\label{eqn_propmaximal2}
		\Tr_\ve{g}\langle R_M(\sth,\xi)\sth,\xi\rangle=\sum_{i=1}^n\langle R_M(e_i,\xi)e_i,\xi\rangle>0 \ \  \text{at $z$}.
	\end{equation}
	By \eqref{eqn_propmaximal1} and \eqref{eqn_propmaximal2}, the right-hand side of \eqref{eqn_variationintegral} is negative when $\xi$ is not identically zero. In view of Theorem \ref{thm_variation} \ref{item_thmvariation1}, this implies the required statement.
\end{proof}

For an application later on, we also recall the following formula for the first variation of the connection $\nabla$ on $\T M|_\Sigma$. Here $\Sigma$ does not need to be minimal.
\begin{lemma}\label{lemma_variationconnection}
	Let $\Sigma$ be a spacelike submanifold in a pseudo-Riemannian manifold $M$ and $(F_t)_{t\in(-\epsilon,\epsilon)}$ be a variation of $\Sigma$. Let $\nabla_t$ denote the connection on $\T M|_{\Sigma_t}\cong \T M|_\Sigma$, where the identification is given by parallel transportation along the paths $(F_t(x))_{t\in(-\epsilon,\epsilon)}$, $x\in\Sigma$. Then the derivative $\dt{\nabla}_0\in\Omega^1(\Sigma,\End(\T M|_\Sigma))$ of $\nabla_t$ at $t=0$ is given by
	\begin{equation}\label{eqn_variationconnection1}
	(\dt{\nabla}_0)_X\zeta=R_M(X,\dt{F}_0)\zeta\quad  \text{ for any }X\in C^\infty(\Sigma,\T\Sigma),\ \zeta\in C^\infty(\Sigma,\T M|_\Sigma).
	\end{equation}
	In particular, if $M$ has constant sectional curvature and $(F_t)$ is a tangential variation (i.e.\ $\dt{F}_0$ is a section of $\T\Sigma\subset\T M|_\Sigma$), then every component of the decomposition \eqref{eqn_nablaTN} is stationary at $t=0$ except for the $\nabla^{\ca{T}}$-component.
\end{lemma}
\begin{proof}
Let $x=(x_1,\cdots,x_n)$ be local coordinates of $\Sigma$. It suffices to prove \eqref{eqn_variationconnection1} at $x=0$ for $X=\pa_{x_1}$. To this end, for each fixed $x_1$, we parallel-translate $\zeta_0(x_1):=\zeta(x_1,0,\cdots,0)\in \T_{(x_1,0,\cdots,0)}M$ along the path $(F_t(x_1,0,\cdots,0))_{t\in(-\epsilon,\epsilon)}$ and get a vector $\zeta_t(x_1)\in \T_{F_t(x_1,0,\cdots,0)}M$ for every $t$. If we fix $t$ and let $x_1$ vary instead, these vectors form a vector field along $(F_t(x_1,0,\cdots,0))_{x_1\in(-\delta,\delta)}$, whose covariant derivative in $x_1$ at $x_1=0$ can be denoted by $\nabla_{\pa_1F_t(0)}\zeta_t(x_1)\in \T_{F_t(0)}M$. The left-hand side of \eqref{eqn_variationconnection1} at $x=0$ is just the covariant derivative of the last vector in $t$ at $t=0$:
$$
(\dt{\nabla}_0)_{\pa_{x_1}}\zeta=\nabla_{\dt{F}_0(0)}\nabla_{\pa_1F_t(0)}\zeta_t(x_1).
$$
On the other hand, by definition of the curvature tensor $R_M$ and the fact that the vector fields $\pa_1F_t(x_1,0)$ and $\pa_tF_t(x_1,0)$ commute, we have
$$
R_M(\pa_{x_1},\dt{F}_0(0))=\nabla_{\dt{F}_0(0)}\nabla_{\pa_1F_t(0)}\zeta_t(x_1)-\nabla_{\pa_{x_1}}\nabla_{\dt{F}_0(x_1,0)}\zeta_t(x_1),
$$
and the last term vanishes by the construction of $\zeta_t(x_1)$. This proves  \eqref{eqn_variationconnection1}. We deduce the ``In particular'' statement by noting that having constant curvature $\kappa$ means $R_M(u,v)w=\kappa(\langle u,w\rangle v-\langle v,w\rangle u)$.
\end{proof}

\subsection{$\ac{J}$-holomorphic curves as minimal surfaces}\label{subsec_jholominimal}
In an almost complex manifold $(M,\ac{J})$, a \emph{$\ac{J}$-holomorphic curve} is a surface $\Sigma$ whose tangent bundle $\T\Sigma\subset \T M$ is preserved by $\ac{J}$. Such a $\Sigma$ is a minimal surface for a pseudo-Riemannian metric $\ve{h}$ satisfying some compatibility condition with $\ac{J}$: 

\begin{proposition}\label{prop_kahler}
	Let $(M,\ve{h})$ be a pseudo-Riemannian manifold endowed with an almost complex structure $\ac{J}$ which is orthogonal (i.e.\ $\ve{h}(\sth,\sth)=\ve{h}(\ac{J}\sth,\ac{J}\sth)$), $\nabla$ be the Levi-Civita connection, and $\Sigma\subset M$ be a spacelike $\ac{J}$-holomorphic curve. Then 
	\begin{enumerate}[label=(\arabic*)]
		\item\label{item_propkahler1} A sufficient condition for $\Sigma$ to be a minimal surface is
		\begin{equation}\label{eqn_quasi}
			\nabla_{\ac{J}X}\ac{J}=-\ac{J}\nabla_X\ac{J}\ \ \text{ for all }X\in C^\infty(\Sigma,\T M).
		\end{equation}
		\item\label{item_propkahler2} The second fundamental form $\II(\sth,\sth)$ of $\Sigma$ is complex bilinear (i.e.\ $\II(\sth,\ac{J}\sth)=\ac{J}\circ \II(\sth,\sth)$) if and only if
		$(\nabla_X\ac{J})Y=0$ for any $X,Y$ tangent to $\Sigma$. A sufficient condition for this is
		\begin{equation}\label{eqn_nearly}
			(\nabla_X\ac{J})X=0 \ \ \text{ for all }X\in C^\infty(\Sigma,\T M).
		\end{equation}	
	\end{enumerate}
\end{proposition}
As a consequence, under condition \eqref{eqn_nearly}, for any point $z\in\Sigma$ where $\II(\sth,\sth)$ is nonzero, the space of osculation vectors $\big\{\II(X,Y)\mid X,Y\in\T_z\Sigma\big\}$ is a complex line in $\ca{N}_z$. We call it the \emph{osculation line} of $\Sigma$ at $z$.
This proposition will mainly be applied to the almost complex $\H^{4,2}$, which fulfills \eqref{eqn_nearly} (see \S \ref{subsec_almostcomplex}).
\begin{remark}\label{rk_kahler1}
	Condition \eqref{eqn_nearly} is equivalent to the condition that 
	$(\nabla_X\ac{J})Y$ is anti-symmetric in $X$ and $Y$, which is
	stronger than \eqref{eqn_quasi} because it implies 
	$$
	(\nabla_{\ac{J}X}\ac{J})Y=-(\nabla_Y\ac{J})\ac{J}X=\ac{J}\nabla_Y\ac{J})X=-\ac{J}(\nabla_X\ac{J})Y
	$$
	(note that $\ac{J}\nabla_X\ac{J}=-(\nabla_X\ac{J})\ac{J}$ because $\ac{J}^2=-\id$). Correspondingly, it is easy to check that the complex linearity of $\II(\sth,\sth)$ also implies the minimality condition $\Tr_\ve{g}\II(\sth,\sth)=0$.
\end{remark}
\begin{remark}\label{rk_nearlykahler}
	When $\ve{h}$ is Riemannian, $(M,\ve{h},\ac{J})$ satisfying \eqref{eqn_quasi} and \eqref{eqn_nearly} are known as \emph{quasi-K\"ahler} and \emph{nearly K\"ahler} manifolds, respectively. Along with the notion of \emph{almost-K\"ahler} manifolds, they were introduced by Gray \cite{gray_michigan,gray_jdg}, whose also proved a higher-dimensional generalization of Part \ref{item_propkahler1} (\cite[Theorem 5.7]{gray_michigan}).
\end{remark}
\begin{proof}[Proof of Proposition \ref{prop_kahler}]
\ref{item_propkahler1} 
We abuse the notation to let $\nabla$ and $\ac{J}$ also standard for the connection and complex structure induced on the vector bundle $\T M|_\Sigma$  (see \S \ref{subsec_submanifolds}). By assumption, $\ac{J}$ preserves the subbundles ${\ca{T}}$ and $\ca{N}$. Let $\JT $ and $\JN $ denote the restriction of $\ac{J}$ to these subbundles. In view of the decompositions
$$
	\nabla=
	\mat{\nabla^{\ca{T}}&-A^\dagger\\A&\nabla^\ca{N}},\quad \ac{J}=\mat{\JT &\\&\JN },
$$
we have the following expression for $\nabla \ac{J}$ as an $\End(\T M|_\Sigma)$-valued $1$-form on $\Sigma$:
\begin{equation}\label{eqn_proofkahler1}
\nabla\ac{J}=\mat{\nabla^{\ca{T}}\JT &\JT A^\dagger-A^\dagger \JN \\A\JT -\JN A&\nabla^\ca{N}\JN }.
\end{equation}
Since $\nabla^{\ca{T}}$ is the Levi-Civita connection of the first fundamental form $\ve{g}$, while $\ve{g}$ is a metric conformal to the complex structure $\JT $, we 
have $\nabla^{\ca{T}}\JT =0$. Therefore, the first column of \eqref{eqn_proofkahler1} gives
\begin{equation}\label{eqn_proofkahler2}
(\nabla_X\ac{J})Y=A(X)\JT Y-\JN A(X)Y=\II(X,\JT Y)-\JN \II(X,Y)
\end{equation}
for any tangent vector fields $X$ and $Y$ of $\Sigma$.

Assuming \eqref{eqn_quasi}, we have $(\nabla_X\ac{J})X+(\nabla_{\ac{J}X}\ac{J})\ac{J}X=0$ for any tangent vector field $X$ of $M$. Taking $X$ to be tangent to $\Sigma$ and applying \eqref{eqn_proofkahler2}, we get
\begin{align*}
	0&=(\nabla_X\ac{J})X+(\nabla_{\JT X}\ac{J})\JT X
	\\&=\big(\II(X,\JT X)-\JN \II(X,X)\big)+\big(-\II(\JT X,X)-\JN \II(\JT X,\JT X)\big)\\
	&=-\JN \big(\II(X,X)+\II(\JT X,\JT X)\big),
\end{align*}
This implies that $\Sigma$ is minimal, because $\II(X,X)+\II(\JT X,\JT X)$ is exactly the mean curvature field $\Tr_\ve{g}\II(\sth,\sth)$ when $X$ has unit length under $\ve{g}$.

\ref{item_propkahler2} 
The ``if and only if'' statement follows immediately from \eqref{eqn_proofkahler2}. Assuming \eqref{eqn_nearly}, we have $(\nabla_X\ac{J})Y+(\nabla_Y\ac{J})X=0$.
Taking $X$ and $Y$ to be tangent to $\Sigma$ and applying \eqref{eqn_proofkahler2} again, we get
\begin{equation}\label{eqn_proofkahler3}
0=(\nabla_X\ac{J})Y+(\nabla_Y\ac{J})X=\II(X,\JT Y)+\II(\JT X,Y)-2\JN \II(X,Y).
\end{equation}
Since we know that $\Tr_\ve{g}\II(\sth,\sth)=0$ by Part \ref{item_propkahler1}, we have $\II(X,\JT Y)=\II(\JT X,Y)$ (this can be checked for $(X,Y)=(e_1,e_1)$, $(e_2,e_2)$, $(e_1,e_2)$ and $(e_2,e_1)$ respectively, for an orthonormal local frame $(e_1,e_2)$ of ${\ca{T}}$ with $e_2=\JT e_1$). Thus, \eqref{eqn_proofkahler3} implies $\II(X,\JT Y)=\JN \II(X,Y)$, as required.
\end{proof}

\subsection{Minimal immersions and maps to symmetric space}\label{subsec_immersion}
Although $\Sigma$ is a submanifold of $(M,\ve{h})$ in all the discussions above, it is just a matter of notation to adapt everything to the setting of an immersion $f:\Sigma\to M$. Namely, instead of $\T M|_\Sigma$, we consider the vector bundle $f^*\T M$ on $\Sigma$, which comes with the metric $\langle\sth,\sth\rangle$ and connection $\nabla$ given by the ones on $\T M$. In addition, we view the differential $\dif f$ as a $1$-form on $\Sigma$ with values in this vector bundle, which tells how $\T \Sigma$ is embedded as a subbundle.

These objects can be used to study the more general setting where $\Sigma$ is endowed with a background metric or conformal structure unrelated to the first fundamental form $f^*\ve{h}$ and/or $f$ is merely a smooth map rather than an immersion. In particular, when $\Sigma$ is a Riemann surface, the notion of \emph{harmonic map} is defined via the covariant derivative of the $1$-form $\dif f$, and it is well known that $f$ is harmonic and weakly conformal at the same time if and only if it is a branched conformal minimal immersion (see \cite{gulliver-osserman-royden}).

In this paper, we will only encounter true immersions and do not need to worry about branching points. We will consider certain $f$ satisfying a strong condition which implies conformal minimality:
\begin{lemma}\label{lemma_conformalminimal}
Let $f$ be an immersion of a Riemann surface $\Sigma$ into a Riemannian manifold $X$. If the vector bundle $f^*\T X$ admits an orthogonal, parallel, complex structure $\ac{J}$ under which $\dif f$ is a $(1,0)$-form, then $f$ is a conformal minimal immersion. 
\end{lemma}
\begin{proof}
The conformality follows immediately from the assumption that $\dif f$ is a $(1,0)$-form and $\ac{J}$ is orthogonal. To show the minimality, we transfer back to the setting where $\Sigma$ is an embedded surface in $X$, and only need to show that if the vector bundle $\T X|_\Sigma$ on $\Sigma$ admits an orthogonal, parallel, complex structure $\ac{J}$ which preserves the subbundle $\T\Sigma$, then $\Sigma$ is minimal. Here, being parallel means $\nabla\ac{J}=0$, but by the computation of $\nabla\ac{J}$ in the proof of Proposition \ref{prop_kahler}, the condition that the $\T\Sigma$-component of $\nabla\ac{J}$ (given by \eqref{eqn_proofkahler2}) vanishes already implies that the second fundamental form $\II(\sth,\sth)$ is complex bilinear, which in turns implies the minimality.  
\end{proof}

Although the proof has some overlap with Proposition \ref{prop_kahler}, this lemma will be applied to a context different from that proposition, with $X$ being the Riemannian symmetric space $X_{\SO_0(n,n+1)}$, which does not admit a natural almost complex structure itself.

Finally, we recall the following well known description of $f^*\T X_G$ (as a metric vector bundle with connection) and $\dif f$ (as a $f^*\T X_G$-valued $1$-form) for a general symmetric spaces $X_G$:
\begin{lemma}\label{lemma_pullback}
Let $G$ be a noncompact semisimple Lie group with maximal compact subgroup $K$ and Cartan decomposition $\frak{g}=\frak{k}\oplus\frak{p}$. Let $X_G=G/K$ be the associated Riemannian symmetric space and $\omega\in \Omega^1(G,\frak{g})$ be the left Maurer-Cartan form on $G$. Given a smooth map $f$ from a manifold $\Sigma$ to $X_G$, if $f$ has a lift $\wt{f}:\Sigma\to G$, then $f^*\T X_G$ identifies with the trivial $\frak{p}$-bundle over $\Sigma$ endowed with metric given by the Killing form and the connection given by the $\End(\frak{p})$-valued $1$-form $\ad(\wt{f}^*\omega)^\frak{k}$ (where $\ad$ is the map $\frak{k}\to\End(\frak{p})$ corresponding to the $\frak{k}$-action on $\frak{p}$), while the $f^*\T X_G$-valued $1$-form $\dif f$ identifies with the $\frak{p}$-valued $1$-form $(\wt{f}^*\omega)^\frak{p}$.
\end{lemma}
\begin{proof}
If we view the projection $\pi:G\to X$ as a principal $K$-bundle endowed with the principal connection whose horizontal distribution is given by left-translating $\frak{p}\subset\frak{g}=\T_eG$ to every point of $G$, then $\T X_G$ identifies with the $\frak{p}$-bundle $G\times_{\Ad_K}\frak{p}$ associated to this principal $K$-bundle, while the Levi-Civita connection on $\T X_G$ is the associated connection. As a consequence, $\pi^*\T X_G$ identifies with the trivial $\frak{p}$-bundle over $G$ with the connection given by the $1$-form $\ad \omega^\frak{k}\in\Omega^1(G,\End(\frak{p}))$, while the $1$-form $\dif\pi\in\Omega^1(G,\frak{p})$ is just $\omega^\frak{p}$. The required statement follows.
\end{proof}
\begin{remark}
In Higgs bundle theory, one obtains a map $f:\Sigma\to X$ from the data of a principal $G$-bundle $\ca{P}$ over $\Sigma$, a flat $G$-connection $\nabla$ on $\ca{P}$ and a $K$-reduction of $\ca{P}$. We may write
$\nabla=\nabla^K+\Psi$,
where $\nabla^K$ is a $K$-connection and $\Psi$ is a $1$-form with values in the $\frak{p}$-bundle $\ca{P}_\frak{p}$ associated to the principal $K$-bundle resulting from the reduction. Lemma \ref{lemma_pullback} essentially means that $f^*\T X_G$ identifies with $\ca{P}_\frak{p}$ endowed with the connection $\nabla^K$ and the metric given by the Killing form, while $\dif f$ is just $\Psi$.
\end{remark}

\section{A-surfaces: definition and properties}\label{sec_asurface}
In this section, we give a detailed discussion of A-surfaces and establish their fundamental properties in Theorems \ref{thm_holo}, \ref{thm_affinetoda} and \ref{thm_gaussmap}, which are roughly stated as Theorem \ref{thm_intro0} in the introduction.
\subsection{The definition}
In analogy to Frenet curves, we introduce:
\begin{definition}\label{def_quasisuper}
	Let $M$ be a pseudo-Riemannian manifold of even dimension $2n\geq4$ and $\Sigma\subset M$ be a smooth, embedded, spacelike surface. Consider the vector bundle $\T M|_\Sigma$ on $\Sigma$ endowed with the metric $\langle\sth,\sth\rangle$ and connection $\nabla$ given by the metric and Levi-Civita connection of $M$.
	Then a \emph{Frenet splitting} of $\T M|_\Sigma$ is an orthogonal splitting 
	$\T M|_\Sigma=\TT_1\oplus \cdots \oplus \TT_n$ into rank $2$ subbundles
	such that
	\begin{itemize}
		\item $\TT_1=\T\Sigma\subset \T M|_\Sigma$ is the tangent bundle of $\Sigma$ (hence $\ca{N}=\TT_2\oplus\cdots\oplus \TT_n$ is the normal bundle);
		\item every $\TT_i$ is (positive or negative) definite;
		\item  $(\nabla_X\xi)^{\TT_i}=0$ for all $X\in C^\infty(\Sigma,\T\Sigma)$ and $\xi\in C^\infty(\Sigma,\TT_j)$ with $|i-j|\geq2$.
	\end{itemize}
\end{definition}

Since $\nabla$ preserves $\langle\sth,\sth\rangle$, the last bullet point implies that $\nabla$ decomposes under the splitting as
\begin{equation}\label{eqn_nablamatrix}
\nabla=
\scalebox{0.9}{$
\mat{
	\nabla_1&-\ve\alpha_2^\dagger&&&\\[0.1cm]
	\ve\alpha_2&\nabla_2&-\ve\alpha_3^\dagger&&\\
	&\ve\alpha_3&\nabla_3&\ddots&\\
	&&\hspace{-0.2cm}\ddots&\hspace{-0.2cm}\ddots&\hspace{-0.1cm}-\ve\alpha_n^\dagger\\[0.1cm]
	&&&\hspace{-0.3cm}\ve\alpha_n&\hspace{-0.1cm}\nabla_n
}$},
\end{equation}
where $\nabla_i$ is a connection on $\TT_i$ preserving the metric $\langle\sth,\sth\rangle|_{\TT_i}$, while $\ve\alpha_j$ and $\ve\alpha_j^\dagger$ are $1$-forms with values in $\Hom(\TT_{j-1},\TT_j)$ and $\Hom(\TT_j,\TT_{j-1})$, respectively,
related by $\langle \ve\alpha_j(X)\xi,\eta\rangle=\langle \xi,\ve\alpha_j^\dagger(X)\eta\rangle$ for any tangent vector field $X$ and sections $\xi$ and $\eta$ of $\TT_{j-1}$ and $\TT_j$ respectively.

\begin{remark}\label{rk_secondff}
The $1$-form $\ve\alpha_2$ is just the second fundamental form of $\Sigma$, as we have
$\II(X,Y)=\ve\alpha_2(X)Y$.
In particular, every osculation vector (see \S \ref{subsec_submanifolds}) is in $\TT_2$. 
By the relation $\langle\II(X,Y),\xi\rangle=\langle X,\ve{S}_\xi(Y)\rangle$, it follows that the shape operator $\ve{S}_\xi\in C^\infty(\Sigma,\End(\T \Sigma))$ assigned to a normal vector field $\xi$ is identically zero whenever $\xi$ is orthogonal to $\TT_2$.
\end{remark}

A linear map  $A:\mathbb{E}^2\to\mathbb{E}^2$ of the Euclidean plane is said to be \emph{conformal} if $A$ sends any circle centered at the origin to a circle (possibly reduced to $\{0\}$), or equivalently, if $A$ is either a scaled rotation, a scaled reflection, or the zero map.
A homomorphism between rank $2$ metric vector bundles is said to be conformal if it is a conformal linear map on each fiber. Our main object of study is:

\begin{definition}\label{def_quasisuperminimal}
Given $M$ and $\Sigma$ as above, with $\Sigma$ connected and orientable, we call $\Sigma$ \emph{superminimal} if $\T M|_\Sigma$ has a Frenet splitting such that $\ve\alpha_j(X):\TT_{j-1}\to \TT_j$ is conformal for all $X\in C^\infty(\Sigma,\T\Sigma)$ and $2\leq j\leq n$. More generally, we call $\Sigma$ \emph{quasi-superminimal} if $\Sigma$ is minimal and $\T M|_\Sigma$ has a Frenet splitting with $\ve\alpha_j(X)$ conformal for $2\leq j\leq n-1$ (the condition is removed for $j=n$ here).
If $\Sigma$ is quasi-superminimal and the $\TT_i$'s have alternating space-/time-likeness, namely $\langle\sth,\sth\rangle|_{\TT_i}$ is positive/negative definite for all odd/even $i$, then we call $\Sigma$ an \emph{A-surface}. 
\end{definition}

As mentioned in the introduction, by counting the number of the positive/negative definite $\TT_i$'s, we see that if $M$ admits an A-surface, then its signature $(p,q)$ is determined from $n=\tfrac{p+q}{2}$ by \eqref{eqn_pqintro}, so we will always let $(p,q)$ be given by that equation whenever an A-surface is being considered.

\begin{remark}\label{rk_superminimalisminimal}
It is easy to show that if a symmetric bilinear map  $B(\sth,\sth):\mathbb{E}^2\times\mathbb{E}^2\to\mathbb{E}^2$ has the property that $B(v,\sth):\mathbb{E}^2\to\mathbb{E}^2$ is conformal for all fixed $v$, then $\Tr B(\sth,\sth):=B(e_1,e_1)+B(e_2,e_2)$ vanishes. Therefore, if $\ve\alpha_2(X)=\II(X,\sth)$  is conformal for all $X\in C^\infty(\Sigma,\T\Sigma)$, then $\Sigma$ is minimal. In particular, superminimal surfaces are minimal, and when $n\geq3$ we may omit the condition ``$\Sigma$ is minimal'' in the definition of quasi-superminimal surfaces as well. On the other hand, when $n=2$, the conformality condition in the definition of quasi-superminimal surfaces is vacuous, so an A-surface in this case is just a spacelike minimal surface with rank $2$ negative definite normal bundle.
\end{remark}

\begin{remark}\label{rk_historysuperminimal}
The notion of superminimal surfaces arose from the study of minimal $2$-spheres. Chern \cite{chern_on} essentially showed that any minimal $2$-sphere $S$ in a Riemannian manifold $M$ of constant curvature is superminimal in the sense of Definition \ref{def_quasisuperminimal} (after replacing $M$ by the smallest totally geodesic submanifold containing $S$), while such spheres are also studied via twistor lift in \cite{calabi,barbosa,michelsohn} for $M=\mathbb{S}^{2n}$. Bryant \cite{bryant_conformal} coined the term ``superminimal'' to describe minimal surfaces of general topological type in $\mathbb{S}^4$ with similar properties, again via twistor lift. When $M$ is a Riemannian $4$-manifold, the equivalent definition via the conformality of $\ve\alpha_2^\dagger(X)$ (which is the same as the conformality of $\ve\alpha_2(X)$) appears in \cite{friedrich, forstneric_proper, forstneric_calabiyau} and is attributed to K.\ Kommerell \cite{kommerell}.
\end{remark}
\subsection{Holomorphic interpretation of $\TT_i$ and $\ve\alpha_j$}\label{subsec_holo}
Given an A-surface $\Sigma\subset M$ with Frenet splitting $\T M|_\Sigma=\TT_1\oplus\cdots\oplus\TT_n$, it is convenient to consider the positive definite metric
$$
\langle\sth,\sth\rangle_i:=(-1)^{i+1}\langle\sth,\sth\rangle|_{\TT_i}
$$ 
on each $\TT_i$. The adjoint $\ve\alpha_j^*(X)$ of $\ve\alpha_j(X):\TT_{j-1}\to\TT_j$ with respect to these metrics is opposite to the adjoint $\ve\alpha_j^\dagger(X)$ in \eqref{eqn_nablamatrix} with respect to the original $\langle\sth,\sth\rangle$, so  \eqref{eqn_nablamatrix} can be rewritten as 
\begin{equation}\label{eqn_nablamatrix2}
	\nabla=
	\scalebox{0.9}{$
	\mat{
		\nabla_1&\ve\alpha_2^*&&&\\[0.1cm]
		\ve\alpha_2&\nabla_2&\ve\alpha_3^*&&\\
		&\ve\alpha_3&\nabla_3&\ddots&\\
		&&\hspace{-0.2cm}\ddots&\hspace{-0.2cm}\ddots&\hspace{-0.1cm}\ve\alpha_n^*\\[0.1cm]
		&&&\hspace{-0.3cm}\ve\alpha_n&\hspace{-0.1cm}\nabla_n
	}$}.
\end{equation}

Now assume that $M$ is the pseudo-hyperbolic space $\H^{p,q}$. Since this is a real analytic manifold, any minimal surface therein is analytic as well (see e.g.\ \cite[p.\ 117]{eells-sampson} for a stronger result in this regard).
Our first main result is a holomorphic interpretation of $\TT_i$ and $\ve\alpha_j$. 
Let us recall some backgrounds in order to give the statement. 

On a Riemann surface $\Sigma$, the following two types of objects can be considered as the same thing:
\begin{itemize}
\item 
rank $r$ hermitian holomorphic vector bundle $(\ca{E},\ve{h})$;
\item
  rank $2r$ real vector bundle $\ca{E}$ endowed with a (positive definite) metric $\langle\sth,\sth\rangle$, an orthogonal complex structure $\ac{J}$ and a connection $\nabla$ preserving $\langle\sth,\sth\rangle$ and $\ac{J}$.
\end{itemize}
Namely, given the former, we obtain the latter by letting $\langle\sth,\sth\rangle$ and $\nabla$ be the real part and Chern connection, respectively, of the hermitian metric $\ve{h}$; conversely, given the latter, the $(0,1)$-part of $\nabla$ is a Dolbeault operator (or ``partial/pseudo connection'') that provides $\ca{E}$ with a holomorphic atlas (see e.g.\ \cite{guichard_introduction}). 

When $r=1$, since the Euclidean plane admits exactly two orthogonal complex structures, corresponding to the two orientations, we conclude that a hermitian holomorphic line bundle $(\TT,\ve{h})$ on $\Sigma$ is the same as an oriented real rank $2$ vector bundle $\TT$ endowed with a metric $\langle\sth,\sth\rangle$ and a metric-preserving connection $\nabla$. It can be shown that the orientation-reversed vector bundle $\cj{\TT}$ (endowed with the same $\langle\sth,\sth\rangle$ and $\nabla$) corresponds to the inverse hermitian holomorphic line bundle $(\TT^{-1},\ve{h}^{-1})$. As an example, the anti-canonical line bundle $\ca{K}^{-1}$ endowed with a hermitian metric identifies with the tangent bundle $\T\Sigma$ endowed with the natural orientation (underlying the Riemann surface structure), a conformal Riemannian metric, and the Levi-Civita connection of that metric.

Returning to the context of A-surfaces, we have:
\begin{theorem}\label{thm_holo}
	 Let $\Sigma$ be an A-surface in $\H^{p,q}$ with $n=\tfrac{p+q}{2}\geq2$. Suppose $\Sigma$ is not contained in any pseudo-hyperbolic subspace of codimension $4$. Then
\begin{enumerate}[label=(\arabic*)]
	\item\label{item_thmholo1}
	 The Frenet splitting $\T \H^{p,q}|_\Sigma=\TT_1\oplus\cdots\oplus \TT_n$ is unique, and none of the $1$-forms $\ve\alpha_2,\cdots,\ve\alpha_{n-1}$ in \eqref{eqn_nablamatrix2} is identically zero (but $\ve\alpha_n$ can be zero, namely when $\Sigma$ is in a codimension $2$ subspace).
	 \item\label{item_thmholo2} 
	 Once an orientation of $\TT_1=\T\Sigma$ is chosen, we can endow each of $\TT_2,\cdots,\TT_n$ with an orientation as well, such that if $\Sigma$ is endowed with the Riemann surface structure given by the first fundamental form $\langle\sth,\sth\rangle_1$ and the orientation, and every $\TT_i$ is viewed as a hermitian holomorphic line bundle on $\Sigma$ by means of the metric $\langle\sth,\sth\rangle_i$, the connection $\nabla_i$ and the orientation (in particular, $\TT_1$ gets identified with the anti-canonical bundle $\ca{K}^{-1}$), then 
		\begin{enumerate}[label=(\roman*)]
			\item\label{item_gammaj1}
			each $\ve\alpha_j$ with $j\leq n-1$ is a $\Hom_\C(\TT_{j-1},\TT_j)$-valued holomorphic $(1,0)$-form;
			\item\label{item_gammaj2}
			$\ve\alpha_n=\ve\alpha_n^++\ve\alpha_n^-$, where $\ve\alpha_n^+$ and $\ve\alpha_n^-$ are holomorphic $(1,0)$-forms with values in the line bundles $\Hom_\C(\TT_{n-1},\TT_n)$ and $\Hom_\C(\TT_{n-1},\cj{\TT}_n)=\Hom_\C(\TT_{n-1},\TT_n^{-1})$, respectively.
		\end{enumerate}
\end{enumerate}
\end{theorem}
\begin{remark}\label{rk_orientation}
One of the nontrivial claims in Part \ref{item_thmholo2} is the \emph{existence} of an orientation on $\TT_i$ for $i\geq2$. On the other hand, the \emph{uniqueness} is relatively simple: On each $\TT_i$ with $2\leq i\leq n-1$, the orientation with the required properties \ref{item_gammaj1} \ref{item_gammaj2} is unique, because reversing it would violate \ref{item_gammaj1}. These orientations all get reversed if the initial orientation of $\TT_1$ is reversed. On the other hand, the orientation of $\TT_n$ is inessential for the properties \ref{item_gammaj1} \ref{item_gammaj2}, as reversing it just switches $\TT_n$ and $\cj{\TT}_n$ and switches $\ve\alpha_n^\pm$.  
\end{remark}
\begin{remark}
\label{rk_insubspace}
$\Sigma$ not being contained in a pseudo-hyperbolic subspace of codimension $4$ is necessary for the uniqueness of the Frenet splitting and the nonzeroness of $\ve\alpha_2,\cdots,\ve\alpha_{n-1}$. For instance, if $\Sigma$ is a totally geodesic hyperbolic plane $\H^2\subset\H^{4,2}$, then $\T\H^{4,2}|_\Sigma$ has uncountably many Frenet splittings with $\ve\alpha_2$ identically zero: each totally geodesic copy of  $\H^{2,2}$ in $\H^{4,2}$ containing $\Sigma$ induces one (in this case, $\ve\alpha_3^\pm$ are zero as well). The nonzeroness of $\ve\alpha_2,\cdots,\ve\alpha_{n-1}$ is crucial in the proof, as we will use it to ``propagate'' the property of being a holomorphic $(1,0)$-form from $\ve\alpha_2$ to $\ve\alpha_3$, then to $\ve\alpha_4$, and so forth.
\end{remark}

Before proceeding with the proof, we need to carry the above discussion of vector bundles one step further. Given a hermitian holomorphic line bundle $\TT$ and a real vector bundle $\ca{V}$ endowed with a metric and a metric-preserving connection, 
the bundle of homomorphisms $\Hom_\R(\TT,\ca{V})$ is a hermitian holomorphic vector bundle in a natural way: its complex structure is given by pre-composing each homomorphism $\ca{L}\to\ca{V}$ with the complex structure of $\TT$, and its metric and connection are induced by those of $\TT$ and $\ca{V}$. If furthermore $\ca{V}$ is endowed with a parallel orthogonal complex structure (hence is itself a hermitian holomorphic vector bundle), then we have a holomorphic decomposition
\begin{equation}\label{eqn_homLV}
\Hom_\R(\TT,\ca{V})=\Hom_\C(\TT,\ca{V})\oplus\Hom_\C(\TT,\cj{\ca{V}})
\end{equation}
given by splitting each homomorphism into complex linear and anti-linear parts. In rank $2$, we have:
\begin{lemma}\label{lemma_orientation}
	Let $\ca{L}$ be a hermitian holomorphic line bundle and $\ca{V}$ be a real rank $2$  vector bundle endowed with a metric and a metric-preserving connection. Let $\ve\alpha$ be a $\Hom_\R(\ca{L},\ca{V})$-valued holomorphic $(1,0)$-form which is not identically zero, such that $\ve\alpha(X)$ is conformal for all $X\in\T\Sigma$. 
	Then $\ca{V}$ is orientable, and has a unique orientation such that $\ve\alpha$ takes values in the subbundle $\Hom_\C(\ca{L},\ca{V})$ of $\Hom_\R(\ca{L},\ca{V})$ defined by the orthogonal complex structure of $\ca{V}$ underlying the orientation.
\end{lemma}
\begin{proof}
It is a basic fact that given a holomorphic vector bundle $\ca{E}$ on a Riemann surface and an $\ca{E}$-valued holomorphic $(1,0)$-form $\ve\alpha$ which is not identically zero, there is a unique holomorphic line subbundle $\ca{A}\subset\ca{E}$ containing the image of $\ve\alpha$. In fact, if we write locally $\ve\alpha=\alpha(z)\dz$ for a holomorphic local section $\alpha(z)$ of $\ca{E}$, then $\ca{A}$ is just generated by $\alpha(z)$ away from the zeros of $\ve\alpha$, and this extends over any zero $z=z_0$ because we may write $(z-z_0)^d\wh\alpha(z)$ around $z_0$ for a nonzero holomorphic section $\wh\alpha(z)$, which equally generates $\ca{A}$.

In the current setting with $\ca{E}=\Hom_\R(\TT,\ca{V})$, since $\ca{V}$ admits locally defined orientations and hence complex structures, the splitting \eqref{eqn_homLV} exists locally, and $\ca{V}$ is globally orientable if and only if the two components of the splitting does not intertwine (hence give a global splitting). But the conformality assumption on $\ve\alpha(X)$ means that the image of $\ve\alpha$ is contain in one of the two components. Therefore, by the above fact, the splitting, and hence the orientation, is global. Also, after possibly reversing the orientation, the image of $\ve\alpha$ is contained in the first component rather than the second. 
\end{proof}

\begin{proof}[Proof of Theorem \ref{thm_holo}]
	We first formulate the Gauss-Codazzi equations of $\Sigma$ in a moving frame. Since $\H^{p,q}$ is formed by all $x\in \R^{p,q+1}$ with $\langle x,x\rangle=-1$, the tangent space $\T_x\H^{p,q}$ identifies with the orthogonal complement $x^\perp$ of $x$ in $\R^{p,q+1}$. Thus, around any point of $\Sigma$, we may choose locally defined maps $\iota,u_1,v_1,u_2,v_2,\cdots,u_n,v_n$ from $\Sigma$ to $\R^{p,q+1}$ such that $\iota$ is the inclusion, while $u_i$ and $v_i$ are perpendicular to $\iota$ and form an orthonormal frame of subspace $\TT_i$ of the tangent space of $\H^{p,q}$ at $f$ (with respect to the positive definite metric $\langle\sth,\sth\rangle_i$, which is $(-1)^{i+1}$ times the metric inherited from $\R^{p,q+1}$). These maps are the moving frame that we are going to use. In particular, $(u_1,v_1)$ is an orthonormal local frame of the tangent bundle $\TT_1=\T\Sigma$.

Consider each of $\iota,u_1,v_1,\cdots,u_n,v_n$ as a column vector and $(\iota,u_1,v_1,\cdots,u_n,v_n)$ as a (locally defined) matrix-valued function on $\Sigma$. In view of \eqref{eqn_nablamatrix2}, its differential is
\begin{equation}\label{eqn_thmholo1}
	\dif(\iota,u_1,v_1,\cdots,u_n,v_n)=(\iota,u_1,v_1,\cdots,u_n,v_n)\Omega\ \text{ with }\ 
	\Omega=
	\scalebox{0.9}{$
	\mat{
	&\transp{\theta}&&&&\\[0.1cm]
	\theta&\Omega_1&\transp{\ve\alpha_2}&&&\\[0.1cm]
	&\ve\alpha_2&\Omega_2&\transp{\ve\alpha_3}&&\\
	&&\ve\alpha_3&\Omega_3&\ddots&\\
	&&&\hspace{-0.2cm}\ddots&\hspace{-0.2cm}\ddots&\hspace{-0.1cm}\transp{\ve\alpha_n}\\[0.1cm]
	&&&&\hspace{-0.3cm}\ve\alpha_n&\hspace{-0.1cm}\Omega_n
	}$},
\end{equation}
where $\theta=\transp(\theta_1,\theta_2)$ is the local frame of $\T^*\Sigma$ dual to $(u_1,v_1)$, $\Omega_i$ is the matrix of the connection $\nabla_i$ under the frame $(u_i,v_i)$, and we abuse the notation to denote the matrix  of $\ve\alpha_j\in\Omega^1(\Sigma,\Hom(\TT_{j-1},\TT_j))$ still by $\ve\alpha_j$ (so $\Omega_i$ and $\ve\alpha_j$ are $2\times 2$ matrices of locally defined $1$-forms).


We observe that since $\Sigma$ is assumed not to be contained in any pseudo-hyperbolic subspace of codimension $4$ or higher, none of the $\ve\alpha_k$'s with $k\leq n-1$ is identically zero. In fact, if $\ve\alpha_k\equiv0$, then $\Omega$ is the direct sum matrix of a upper-left block $\Omega'$ of size $2k-1$ and a lower-right block $\Omega''$ of size $2(n-k+1)$, and by \eqref{eqn_thmholo1} we have
$\dif(\iota,u_1,v_1,\cdots,u_{k-1},v_{k-1})=(\iota,u_1,v_1,\cdots,u_{k-1},v_{k-1})\Omega'$. This implies that $\iota,u_1,v_1,\cdots,u_{k-1},v_{k-1}$ span a fixed nondegenerate subspace $L\subset\R^{p,q+1}$. In particular, $\Sigma$ is contained in the pseudo-hyperbolic subspace $L\cap \H^{p,q}$ of codimension $2(n-k+1)$, which violates the assumption unless $k=n$.

Under this moving frame, the Gauss-Codazzi equations of $\Sigma$ are just
\begin{equation}\label{eqn_omegaomega}
	\dif\Omega+\Omega\wedge\Omega=0.
\end{equation}
To get more explicit equations on $\Omega_i$ and $\ve\alpha_j$, note that since $(u_i,v_i)$ is an orthonormal frame of $\TT_i$ and $\nabla_i$ preserves the metric, $\Omega_i$ has the form
$$
\Omega_i=\scalebox{0.9}{$\mat{&\omega_i\\-\omega_i&}$}
$$ 
for a $1$-form $\omega_i$, hence $\Omega_i\wedge\Omega_i=0$. We also clearly have $\transp\theta\wedge\theta=0$. So we find by computation that
$$
\Omega\wedge\Omega=
\scalebox{0.95}{
$\mat{
	&\hspace{0.3cm}\transp\theta\wedge\Omega_1&\hspace{0.3cm}\transp\theta\wedge\transp \ve\alpha_2&&&&\\[0.4cm]
	\Omega_1\wedge\theta&\hspace{0.4cm}\scalebox{0.85}{\parbox{1.5cm}{$\theta\wedge\transp\theta$\\$+\transp \ve\alpha_2\wedge\ve\alpha_2$}}&\hspace{0.4cm}\scalebox{0.85}{\parbox{1.6cm}{$\Omega_1\wedge\transp \ve\alpha_2$\\$+\transp \ve\alpha_2\wedge\Omega_2$}}&\hspace{0.3cm}\transp \ve\alpha_2\wedge\transp \ve\alpha_3&&&\\[0.6cm]
	\ve\alpha_2\wedge\theta&\hspace{0.4cm}\scalebox{0.85}{\parbox{1.5cm}{$\ve\alpha_2\wedge\Omega_1$\\$+\Omega_2\wedge\ve\alpha_2$}}&\hspace{0.4cm}\scalebox{0.85}{\parbox{1.5cm}{$\ve\alpha_2\wedge\transp \ve\alpha_2$\\$+\transp \ve\alpha_3\wedge\ve\alpha_3$}}&\hspace{0.4cm}\scalebox{0.85}{\parbox{1.6cm}{$\Omega_2\wedge\transp \ve\alpha_3$\\$+\transp \ve\alpha_3\wedge\Omega_3$}}&\hspace{0.3cm}\transp \ve\alpha_3\wedge\transp \ve\alpha_4&&\\[0.6cm]
	&\hspace{0.3cm}\ve\alpha_3\wedge\ve\alpha_2&\hspace{0.4cm}\scalebox{0.85}{\parbox{1.5cm}{$\ve\alpha_3\wedge\Omega_2$\\$+\Omega_3\wedge\ve\alpha_3$}}&\hspace{0.4cm}\scalebox{0.85}{\parbox{1.5cm}{$\ve\alpha_3\wedge\transp \ve\alpha_3$\\$+\transp \ve\alpha_4\wedge\ve\alpha_4$}}&\ddots&\ddots&\\[0.5cm]
	&&\hspace{0.3cm}\ve\alpha_4\wedge\ve\alpha_3&\ddots&\ddots&\ddots&\scalebox{0.9}{$\transp \ve\alpha_{n-1}\wedge \transp \ve\alpha_n$}\\[0.5cm]
	&&&\ddots&\ddots&\scalebox{0.85}{\parbox{2.2cm}{$\ve\alpha_{n-1}\wedge\transp \ve\alpha_{n-1}$\\$+\transp \ve\alpha_n\wedge\ve\alpha_n$}}&\scalebox{0.85}{\parbox{2cm}{$\Omega_{n-1}\wedge\transp \ve\alpha_n$\\$+\transp\ve\alpha_n\wedge\Omega_n$}}\\[0.6cm]
	&&&&\scalebox{0.9}{$\ve\alpha_n\wedge\ve\alpha_{n-1}$}&\hspace{0.4cm}\scalebox{0.85}{\parbox{2.1cm}{$\ve\alpha_n\wedge\Omega_{n-1}$\\$+\Omega_n\wedge\ve\alpha_n$}}&\scalebox{0.9}{$\ve\alpha_n\wedge\transp \ve\alpha_n$}
}$}.
$$
Thus, the Gauss-Codazzi equations \eqref{eqn_omegaomega} can be rewritten as the following three sets of equations, which are respectively the diagonal part and the parts two and three steps away from the diagonal:
\begin{align}
&\left\{\hspace{-0.15cm}
\begin{array}{l}
		\dif \Omega_1+\theta\wedge\transp\theta+\transp \ve\alpha_2\wedge\ve\alpha_2=0,\\
		\dif \Omega_k+\ve\alpha_k\wedge\transp \ve\alpha_k+\transp \ve\alpha_{k+1}\wedge\ve\alpha_{k+1}=0\quad (2\leq k\leq n-1),\\
		\dif \Omega_n+\ve\alpha_n\wedge\transp \ve\alpha_n=0.
	\end{array}
\right.
\label{eqn_sss1}\\
&\left\{\hspace{-0.15cm}
 \begin{array}{l}
		\dif \theta+\Omega_1\wedge\theta=0,\\
		\dif \ve\alpha_k+\ve\alpha_{k-1}\wedge\Omega_k+\Omega_k\wedge\ve\alpha_k=0\quad(2\leq k\leq n).
\end{array}
\right.
\label{eqn_sss2}\\
&\left\{\hspace{-0.15cm}
\begin{array}{l}
		\ve\alpha_2\wedge\theta=0,\\
		\ve\alpha_k\wedge\ve\alpha_{k-1}=0\quad(3\leq k\leq n).
	\end{array}
\right.
\label{eqn_sss3}
\end{align}

We may now prove Part \ref{item_thmholo2} of the theorem by showing, for $i=2,\cdots,n$ successively, that $\TT_i$ can be oriented and $\ve\alpha_i$ has the required holomorphic description.

Let us first treat $\TT_2$ and $\ve\alpha_2$. Since $(u_1,v_1)$ is an orthonormal local frame of $\TT_1=\T\Sigma$ and the second fundamental form of $\Sigma$ is $\II(X,Y)=\ve\alpha_2(X)Y$ (see Remark \ref{rk_secondff}), the conditions that $\Sigma$ is a minimal surface and that $\II$ is a symmetric tensor amount to the following equations, respectively:
\begin{equation}\label{eqn_proofthmholo1}
	\ve\alpha_2(u_1)u_1+\ve\alpha_2(v_1)v_1=0,\quad \ve\alpha_2(u_1)v_1=\ve\alpha_2(v_1)u_1
\end{equation}
(the latter is effectively the first equation in \eqref{eqn_sss3}). We may further assume that the frame $(u_1,v_1)$ is compatible with the prescribed orientation of $\TT_1$, so that the complex structure $\ac{J}_1$ on $\TT_1$ is given by $\ac{J}_1u_1=v_1$. Then we deduce from \eqref{eqn_proofthmholo1} that
\begin{equation}\label{eqn_proofthmholo2}
	\ve\alpha_2(\ac{J}_1X)Y=\ve\alpha_2(X)\ac{J}_1Y\  \text{ for any }X,Y\in C^\infty(\Sigma,\T\Sigma).
\end{equation}
In fact, using \eqref{eqn_proofthmholo1}, we readily check that \eqref{eqn_proofthmholo2} holds when $(X,Y)$ is $(u_1,u_1)$, $(u_1,v_1)$, $(v_1,u_1)$ and $(v_1,v_1)$ respectively, but this implies that \eqref{eqn_proofthmholo2} holds in general.

Since $\TT_1$ is a hermitian holomorphic line bundle, we can view $\Hom_\R(\TT_1,\TT_2)$ as a rank $2$ hermitian holomorphic vector bundle as explained before the proof, and understand \eqref{eqn_proofthmholo2} as saying that $\ve\alpha_2$ is a $\Hom_\R(\TT_1,\TT_2)$-valued $(1,0)$-form.
On the other hand, the $k=2$ case of the second equation in \eqref{eqn_sss2} just means that the covariant exterior derivative $\dif^{\nabla^{\Hom_\R(\TT_1,\TT_2)}} \ve\alpha_2\in\Omega^2(\Sigma,\Hom_\R(\TT_1,\TT_2))$ vanishes. But it is a basic fact that a smooth $(1,0)$-form $\ve\alpha$ with values in a hermitian holomorphic vector bundle $\ca{E}$ is holomorphic if and only if $\dif^{\nabla^\ca{E}}\ve\alpha=0$ (where $\nabla^\ca{E}$ is the Chern connection): in fact, being holomorphic means $\ve\alpha$ is annihilated by the Dolbeault operator $\bpa_\ca{E}:\Omega^{1,0}(\Sigma,\ca{E})\to \Omega^{1,1}(\Sigma, \ca{E})=\Omega^2(\Sigma,\ca{E})$, but $\bpa_\ca{E}$ coincides with $\dif ^{\nabla^\ca{E}}$ on $\Omega^{1,0}(\Sigma,\ca{E})$ because $\nabla^\ca{E}$ is a connection of $(1,0)$-type and  $\Sigma$ does not admit $(2,0)$-forms. Therefore, we conclude that $\ve\alpha_2$ is a holomorphic  $\Hom_\R(\TT_1,\TT_2)$-valued $(1,0)$-form. We proceed to show that $\ve\alpha_2$ matches the description \ref{item_gammaj2} and \ref{item_gammaj1} when $n=2$ and $n\geq3$ respectively.

\textbf{Case $n=2$.} In this case we have $\T\H^{p,q}|_\Sigma=\TT_1\oplus \TT_2$. By the additivity of the first Stiefel-Whitney class, we know that the direct sum of an orientable vector bundle with a non-orientable one is again non-orientable. So $\TT_2$ is orientable because $\TT_1$ and $\H^{p,q}$ are. Pick any orientation and view $\TT_2$ as a hermitian holomorphic line bundle correspondingly. Since $\Hom_\R(\TT_1,\TT_2)$ is the direct sum of the holomorphic line subbundles $\Hom_\C(\TT_1,\TT_2)$ and $\Hom_\C(\TT_1,\cj{\TT}_2)$ (see the paragraph preceding Lemma \ref{lemma_orientation}), $\ve\alpha_2$ is the sum of holomorphic $(1,0)$-forms $\ve\alpha_2^+$ and $\ve\alpha_2^-$ with values in them respectively, as required.

\textbf{Case $n\geq3$.} 
In this case, the conformality condition in Definition \ref{def_quasisuperminimal} takes effect. In particular,  $\ve\alpha_2(X)$ is conformal for any $X$. So we may use Lemma \ref{lemma_orientation} to conclude that $\TT_2$ carries a unique orientation under which $\ve\alpha_2$ is a nonzero $\Hom_\C(\TT_1,\TT_2)$-valued holomorphic $(1,0)$-form, as required.

Thus, we have shown that $\TT_2$ can be oriented and $\ve\alpha_2$ has the required description.

Next, we assume $n\geq3$ and show that $\TT_3$ can be oriented and $\ve\alpha_3$ has the required description. As $\TT_2$ has already been oriented and hence is a hermitian holomorphic line bundle, $\Hom_\R(\TT_2,\TT_3)$ is a rank $2$ hermitian holomorphic vector bundle and  we may write $\ve\alpha_3=\ve\alpha_3'+\ve\alpha_3''$ under the decomposition into $(1,0)$- and $(0,1)$-parts
$$
\Omega^1(\Sigma,\Hom_\R(\TT_2,\TT_3))= \Omega^{1,0}(\Sigma,\Hom_\R(\TT_2,\TT_3))\oplus \Omega^{0,1}(\Sigma,\Hom_\R(\TT_2,\TT_3)).
$$
Since $\ve\alpha_2$ and $\ve\alpha_3'$ are $(1,0)$-forms with values in $\Hom_\C(\TT_1,\TT_2)$ and $\Hom_\R(\TT_2,\TT_3)$, respectively, their wedge-composition $\ve\alpha_3'\wedge\ve\alpha_2$ is a $\Hom_\R(\TT_1,\TT_3)$-valued $(2,0)$-form, which must vanish. Therefore, the $k=3$ case of the second equation in \eqref{eqn_sss3} gives
$\ve\alpha_3''\wedge\ve\alpha_2
=\ve\alpha_3\wedge\ve\alpha_2=0$.
This forces $\ve\alpha_3''$ to vanish on the set $\{\ve\alpha_2\neq 0\}\subset\Sigma$. But since $\ve\alpha_2$ is holomorphic and is not identically zero as noticed early, this set is dense. So we conclude by continuity that $\ve\alpha_3''$ vanishes everywhere, or in other words, $\ve\alpha_3=\ve\alpha_3'$ is a $(1,0)$-form. Now the rest of the argument is the same as the above one for $\ve\alpha_2$: we first use the second equation in \eqref{eqn_sss2} to infer that $\ve\alpha_3$ is holomorphic, then show in the cases $n=3$ and $n\geq4$ separately that $\TT_3$ can be oriented and $\ve\alpha_3$ has the required description. 

This scheme of argument carries on and eventually shows that $\TT_i$ can be oriented and $\ve\alpha_i$ has the required holomorphic description for all $i=2,\cdots,n$. The proof of Part \ref{item_thmholo2} is thus finished.

For Part \ref{item_thmholo1}, we have already seen that $\ve\alpha_2,\cdots,\ve\alpha_{n-1}$ are nonzero. To show the uniqueness of the Frenet splitting, we note that as a subbundle of $\T\H^{p,q}|_\Sigma$, $\TT_1=\T\Sigma$ is already completely determined, whereas each $\TT_i$ is determined from $\TT_1,\cdots,\TT_{i-1}$ because
\begin{itemize}
	\item if $i=n$, then $\TT_i$ is just the orthogonal complement of $\TT_1\oplus\cdots\oplus\TT_{i-1}$ in $\T\H^{p,q}|_\Sigma$;
	\item if $i\leq n-1$, $\TT_i$ is also determined because the image of the $\Hom_\C(\TT_{i-1},\TT_i)$-valued holomorphic $(1,0)$-form $\ve\alpha_i$ is full in $\TT_i$ except at the isolated zeros.
\end{itemize}
This implies the required uniqueness and completes the proof.
\end{proof}

In what follows, given an \emph{oriented} A-surface $\Sigma\subset\H^{p,q}$ not contained in a codimension $4$ pseudo-hyperbolic subspace, we will always view $\Sigma$ as a Riemann surface as in Theorem \ref{thm_holo} \ref{item_thmholo2} and refer to the collection
$(\TT_1,\cdots,\TT_n,\ve\alpha_2,\cdots,\ve\alpha_{n-1},\ve\alpha_n^\pm)$
of hermitian holomorphic line bundles and holomorphic forms as the \emph{structural data} of $\Sigma$. Moreover, we let $\ve{h}_i$ denote the hermitian metric on $\TT_i$. Note that $\TT_1$ is the anti-canonical line bundle $\ca{K}^{-1}$ and $\ve{g}:=\re\ve{h}_1=\langle\sth,\sth\rangle_1$ is the first fundamental form.

\subsection{Affine Toda system from Gauss-Codazzi equations}
In the above proof, we wrote the Gauss-Codazzi equations of an A-surface $\Sigma\subset\H^{p,q}$ as the three sets of equations \eqref{eqn_sss1}, \eqref{eqn_sss2} and \eqref{eqn_sss3}, and basically deduced Theorem \ref{thm_holo} from the latter two. We show next that the first set of equations \eqref{eqn_sss1} is an \emph{affine Toda system}, namely the same type of equation as the Hitchin equation of cyclic Higgs bundles (see \cite{baraglia_thesis,baraglia_dedicata,collier-li}). This is  an expected result, as A-surfaces will be linked  directly with cyclic Higgs bundles in \S \ref{sec_cyclic} below.

\begin{theorem}\label{thm_affinetoda}	
For any oriented A-surface $\Sigma\subset\H^{p,q}$ (with $n=\frac{p+q}{2}\geq2$) not contained in a pseudo-hyperbolic subspace of codimension $4$, the structural data satisfies 
\begin{align*}
&\text{if $n=2$:}\quad
	\begin{cases}
		\pazbz\log h_1=\tfrac{1 }{2}h_1-\tfrac{h_2}{h_1}|\alpha_2^+|^2-\frac{1}{h_1h_2}|\alpha_2^-|^2,\\[0.2cm]
		\pazbz\log h_2=\tfrac{h_2}{h_1}|\alpha_2^+|^2-\frac{1}{h_1h_2}|\alpha_2^-|^2;
	\end{cases}\\
&\text{if $n\geq3$:}\quad
\begin{cases}
	\pazbz\log h_1=\tfrac{1}{2}h_1-\tfrac{h_2}{h_1}|\alpha_2|^2,\\[0.2cm]
	\pazbz\log h_k=\tfrac{h_k}{h_{k-1}}|\alpha_k|^2-\tfrac{h_{k+1}}{h_k}|\alpha_{k+1}|^2\quad \text{for $2\leq k\leq n-2$},\\[0.2cm]
	\pazbz\log h_{n-1}=\tfrac{h_{n-1}}{h_{n-2}}|\alpha_{n-1}|^2-\tfrac{h_{n}}{h_{n-1}}|\alpha_n^+|^2-\tfrac{1}{h_{n-1}h_n}|\alpha_n^-|^2,\\[0.2cm]
	\pazbz\log h_n=\tfrac{h_{n}}{h_{n-1}}|\alpha_n^+|^2-\tfrac{1}{h_{n-1}h_n}|\alpha_n^-|^2;
\end{cases}
\end{align*}
where notations are as follows: we view $\ve\alpha_j$ ($2\leq j\leq n-1$) and $\ve\alpha_n^\pm$ as holomorphic sections of $\ca{K}\!\ca{L}_{j-1}^{-1}\ca{L}_j$ and $\ca{K}\!\ca{L}_{n-1}^{-1}\ca{L}_{n}^{\pm1}$, respectively, then pick a conformal coordinate $z$ of $\Sigma$ and a holomorphic local frame $\ell_i(z)$ of $\TT_i$ to write locally 
$h_i=\ve{h}_i(\ell_i,\ell_i)$, $\ve\alpha_j=\alpha_j \ell_{j-1}^{-1}\ell_j\dz$ and $\ve\alpha_n^\pm=\alpha_n^\pm \ell_{n-1}^{-1}\ell_n^{\pm1}\dz$.
\end{theorem}
\begin{remark}\label{rk_norm}
Despite having been exhibited in a local form, these equations make sense globally. In fact, after being divided by $h_1$, every term is a globally defined function on $\Sigma$: the left-hand side $\tfrac{1}{h_1}\pazbz\log h_i$ is a constant times the ratio between the curvature form $\bpa\pa\log h_i$ of the Chern connection of $\TT_i$ and the volume form $h_1\dx\wedge\dy$ of the first fundamental form $\ve{g}=\re\ve{h}_1$, whereas each term on the right-hand side is the (pointwise) squared norm of $\ve\alpha_j$ or $\ve\alpha_n^\pm$ under the hermitian metric on the respective line bundle induced by the $\ve{h}_i$'s, namely
$$
\|\ve\alpha_j\|^2=\tfrac{h_{j}}{h_1h_{j-1}}|\alpha_j|^2\quad (2\leq j\leq n-1),\quad \|\ve\alpha_n^+\|^2=\tfrac{h_n}{h_1h_{n-1}}|\alpha_n^+|^2,\quad \|\ve\alpha_n^-\|^2=\tfrac{1}{h_1h_{n-1}h_n}|\alpha_n^-|^2.
$$
\end{remark}
\begin{proof}
In the proof of Theorem \ref{thm_holo}, we obtained equations \eqref{eqn_sss1}, \eqref{eqn_sss2} and \eqref{eqn_sss3} after choosing an orthonormal local frame $(u_i,v_i)$ for each $\TT_i$. The matrices $\Omega_i$ and $\ve\alpha_j$ therein depend on the choice. Under the special choice
$(u_i,v_i)=h_i^{-\frac{1}{2}}(\ell_i,\ima\ell_i)$, we have
\begin{equation}\label{eqn_prooftoda1}
\ve\alpha_j
=\tfrac{\sqrt{h_j}}{\sqrt{h_{j-1}}}
\scalebox{0.9}{
	$\mat{\re(\alpha_j\dz)&-\im(\alpha_j\dz)\\[0.1cm]\im(\alpha_j\dz)&\re(\alpha_j\dz)}$}
\end{equation}
for $2\leq j\leq n-1$ and
\begin{equation}\label{eqn_prooftoda2}
	\ve\alpha_n^+=\tfrac{\sqrt{h_n}}{\sqrt{h_{n-1}}}
	\scalebox{0.9}{
		$\mat{\re(\alpha_n^+\dz)&-\im(\alpha_n^+\dz)\\[0.1cm]\im(\alpha_n^+\dz)&\re(\alpha_n^+\dz)}$},\quad
	\ve\alpha_n^-=\tfrac{1}{\sqrt{h_{n-1}h_n}}
	\scalebox{0.9}{
		$\mat{\re(\alpha_n^-\dz)&-\im(\alpha_n^-\dz)\\[0.1cm]-\im(\alpha_n^-\dz)&-\re(\alpha_n^-\dz)}$}.
\end{equation}

We shall show that \eqref{eqn_sss1} reduces to the required equations. After expanding the $1$-forms in \eqref{eqn_prooftoda1} and \eqref{eqn_prooftoda2} in the real coordinates $(x,y)$ with $z=x+\ima y$ (e.g.\ $\re(\alpha_j\dz)=\re(\alpha_j)\dx-\im(\alpha_j)\dy$, etc.), we obtain by computations that
\begin{align}
	\ve\alpha_j\wedge  \transp \ve\alpha_j&=-\transp \ve\alpha_j\wedge \ve\alpha_j=-\tfrac{2\,h_j}{h_{j-1}}|\alpha_j|^2\scalebox{0.8}{$\mat{&-1\\1&}$}\dx\wedge \dy\quad (2\leq j\leq n-1),\label{eqn_prooftoda5}\\
	\ve\alpha_n^+\wedge\transp\ve\alpha_n^+&=-\transp\ve\alpha_n^+\wedge\ve\alpha_n^+=-\tfrac{2\,h_n}{h_{n-1}}|\alpha_n^+|^2\scalebox{0.8}{$\mat{&-1\\1&}$}\dx\wedge \dy,\label{eqn_prooftoda51}\\
	\ve\alpha_n^-\wedge\ve\alpha_n^-&=\tfrac{2}{h_{n-1}h_n}|\alpha_n^-|^2\scalebox{0.8}{$\mat{&-1\\1&}$}\dx\wedge \dy,\label{eqn_prooftoda52}\\
	\ve\alpha_n^+\wedge\ve\alpha_n^-&=\transp\ve\alpha_n^+\wedge\ve\alpha_n^-=0.\label{eqn_prooftoda53}
\end{align}
Since $\ve\alpha_n^-$ is a symmetric and $\ve\alpha_n=\ve\alpha_n^++\ve\alpha_n^-$, we deduce from \eqref{eqn_prooftoda51}$\sim$\eqref{eqn_prooftoda53} that
\begin{align}
	\ve\alpha_n\wedge\transp \ve\alpha_n&=(\ve\alpha_n^++\ve\alpha_n^-)\wedge(\transp\ve\alpha_n^++\transp\ve\alpha_n^-)=2\left(-\tfrac{h_{n}}{h_{n-1}}|\alpha_n^+|^2+\tfrac{1}{h_{n-1}h_n}|\alpha_n^-|^2\right)\scalebox{0.8}{$\mat{&-1\\1&}$}\dx\wedge \dy,\label{eqn_prooftoda6}\\
	\transp \ve\alpha_n\wedge	\ve\alpha_n&=(\transp\ve\alpha_n^++\transp\ve\alpha_n^-)\wedge(\ve\alpha_n^++\ve\alpha_n^-)=2\left(\tfrac{h_{n}}{h_{n-1}}|\alpha_n^+|^2+\tfrac{1}{h_{n-1}h_n}|\alpha_n^-|^2\right)\scalebox{0.8}{$\mat{&-1\\1&}$}\dx\wedge \dy. \label{eqn_prooftoda7}
\end{align}

Next, since $\Omega_i\wedge\Omega_i=0$ (see the proof of Theorem \ref{thm_holo}), $\dif\Omega_i$ is the matrix of the curvature $2$-form $F_{\nabla_i}\in\Omega^2(\Sigma,\End_\C(\TT_i))$ of $\nabla_i$ under the frame $(u_i,v_i)$. But since $\nabla_i$ is the Chern connection of $\TT_i$, its curvature is expressed in the holomorphic frame $\ell_i$ as
$\bpa\pa\log h_i=2\ima \pazbz\,\log h_i\dx\wedge\dy$.
Changing to the real frame $(u_i,v_i)$, we get 
\begin{equation}\label{eqn_prooftoda8}
\dif\Omega_i=2\,\pazbz\log h_i\scalebox{0.8}{$\mat{&-1\\1&}$}\dx\wedge \dy.
\end{equation}

Finally, since $\theta=\transp(\theta_1,\theta_2)$ is the frame of $\T^*\Sigma$ dual to the orthonormal frame $(u_1,v_1)$ of $\TT_1=\T\Sigma$, $\theta_1\wedge\theta_2$ equals the volume form $h_1\dx\wedge\dy$ of first fundamental form $\ve g=\re\ve h_1$. It follows that 
\begin{equation}\label{eqn_prooftoda9}
\theta\wedge\transp\theta=-\scalebox{0.8}{$\mat{&-1\\1&}$}\theta_1\wedge\theta_2=-h_1\scalebox{0.8}{$\mat{&-1\\1&}$}\dx\wedge\dy.
\end{equation}

Inserting \eqref{eqn_prooftoda5} and \eqref{eqn_prooftoda6}$\sim$\eqref{eqn_prooftoda9} into \eqref{eqn_sss1}, we get the required equations.
\end{proof}

\subsection{Gauss map}\label{subsec_gauss}
The classical notion of Gauss map has a nice generalization to spacelike surfaces $\Sigma$ in $\H^{2,q}$, which plays a key role in \cite{bonsante-schlenker_max, collier-tholozan-toulisse}. In this case, the Gauss map $f$ of $\Sigma$ takes values in the symmetric space 
$\Sy_{\SO_0(2,q+1)}$ (see \S \ref{subsec_pseudohyperbolic}) and is defined in such a way that $f(z)$ is the timelike totally geodesic $q$-sphere $S\subset\H^{2,q}$ passing through $z$ with $\T_zS$ being the orthogonal complement of $\T_z\Sigma$ in $\T_z\H^{2,q}$.
For A-surfaces in $\H^{p,q}$, we introduce:
\begin{definition}\label{def_gauss}
Given an A-surface $\Sigma\subset\H^{p,q}$ not contained in any pseudo-hyperbolic subspace of codimension $4$, we let $\ca{N}^-:=\TT_2\oplus \TT_4\oplus\cdots$ denote the timelike part of the normal bundle of $\Sigma$ given by the Frenet splitting and define the Gauss map 
$$
f:\Sigma\to \Sy_{\SO_0(p,q+1)}=\big\{\text{timelike totally geodesic $q$-spheres in $\H^{p,q}$}\big\}
$$ 
in such a way that for every $z\in \Sigma$, $f(z)$ is the totally geodesic $q$-sphere $S\subset\H^{p,q}$ passing through $z$ with $\T_zS=\ca{N}^-_z\subset \T_z\H^{p,q}$.
\end{definition}
Note that since the definition requires the uniqueness of the Frenet splitting (Theorem \ref{thm_holo} \ref{item_thmholo1}), the assumption that $\Sigma$ is not in a subspace of codimension $4$ is indispensable (see Remark \ref{rk_insubspace}).

\begin{theorem}\label{thm_gaussmap}
	In the above setting, $f$ is a conformal minimal immersion, with pullback metric
	\begin{equation}\label{eqn_thmgauss}
	f^*\ve{g}_X=(8n-4)\big(\tfrac{1}{2}+\|\ve\alpha_2\|^2+\cdots+\|\ve\alpha_{n-1}\|^2+\|\ve\alpha_n^+\|^2+\|\ve\alpha_n^-\|^2\big)\ve{g},
	\end{equation}
	where $\ve{g}_X$ is the metric on $X:=\Sy_{\SO_0(p,q+1)}$ and $\ve{g}=\re\ve{h}_1$ is the first fundamental form of $\Sigma$.
\end{theorem}
Here, the squared norms in \eqref{eqn_thmgauss} are as define in Remark \ref{rk_norm}, while the metric $\ve{g}_X$ is the one induced by the Killing form $B(x,y)=(2n-1)\Tr(xy)$ of $\frak{so}(n,n+1)$.
\begin{proof}
We first describe the metric and connection on $f^*\T X$ (see \S \ref{subsec_immersion}) by using Lemma \ref{lemma_pullback} and the moving frame $(\iota,u_1,v_1,\cdots,u_n,v_n)$ from the proof of Theorem \ref{thm_holo}. Viewing $X$ as the space of $(q+1)$-dimensional negative definite subspaces of $\R^{p,q+1}$, by definition of the Gauss map $f$, we can write
$$
f(z)=\mathrm{Span}\big\{\iota(z),\,u_2(z),\,v_2(z),\,u_4(z),\,v_4(z),\,\cdots\big\}.
$$
In particular, the matrix-valued function
$$
\wt{f}:\Sigma\to G:=\SO_0(p,q+1),\quad \wt{f}:=(u_1,\,v_1,\,u_3,\,v_3,\,\cdots,\,\iota,\,u_2,\,v_2,\,u_4,\,v_4,\,\cdots)
$$
is a lift of $f$. Thus, by Lemma \ref{lemma_pullback}, $f^*\T X$ identifies locally with the trivial $\frak{p}$-bundle endowed with the metric given by the Killing form $B$ and the connection $\wh{\nabla}$ given by the $\End(\frak{p})$-valued $1$-form $\ad(\wt{f}^*\omega)^\frak{k}$, while $\dif f\in\Omega^1(\Sigma,f^*\T X)$ is locally the $\frak{p}$-valued $1$-form $(\wt{f}^*\omega)^\frak{p}$. In order to get more explicit expressions, we view $\frak{g}:=\frak{so}(p,q+1)$ and $G$ as matrix algebra and matrix group contained in $\R^{(p+q+1)\times(p+q+1)}$, so that the Maurer-Cartan form $\omega$ can be written as $g^{-1}\dif g$ (where $g\in G$ is the matrix coordinate), and hence $\wt{f}^*\omega=\wt{f}^{-1}\dif\wt{f}$. By equation \eqref{eqn_thmholo1} in the proof of Theorem \ref{thm_holo}, we obtain
$$
\wt{f}^*\omega=
\wt{f}^{-1}\dif\wt{f}=
\left(
\begin{array}{c|c}
	\scalebox{0.85}{$
	\begin{matrix}
		\Omega_1&&\\&\hspace{-0.2cm}\Omega_3&\\[-0.1cm]&&\hspace{-0.15cm}\ddots
	\end{matrix}
$}	
	&\transp\Gamma\\
	\midrule
	\Gamma&
		\scalebox{0.85}{$
	\begin{matrix}
		0&&&\\[-0.1cm]&\hspace{-0.1cm}\Omega_2&&\\&&\hspace{-0.2cm}\Omega_4&\\[-0.1cm]&&&\hspace{-0.1cm}\ddots
	\end{matrix}
$}
\end{array}
\right),
$$
where
\begin{equation}\label{eqn_Gamma}
	\Gamma=
	\scalebox{0.9}{$
	\mat{
		\transp\theta&&&&\\
		\ve\alpha_2&\transp\ve\alpha_3&&&\\
		&\ve\alpha_4&\transp\ve\alpha_5&&\\
		&&\hspace{-0.3cm}\ddots&\ddots&\\
		&&&\hspace{-0.3cm}\ve\alpha_{n-2}&\transp\ve\alpha_{n-1}\\
		&&&&\hspace{-0.2cm}\ve\alpha_n
	}$}
\text{ if $n$ is even,}\quad
\Gamma=
	\scalebox{0.9}{$
	\mat{
		\transp\theta&&&&\\
		\ve\alpha_2&\transp\ve\alpha_3&&&\\
		&\ve\alpha_4&\transp\ve\alpha_5&&\\
		&&\hspace{-0.3cm}\ddots&\ddots&\\
		&&&\hspace{-0.3cm}\ve\alpha_{n-1}&\transp\ve\alpha_n
	}$}
\text{ if $n$ is odd}.
\end{equation}
Meanwhile, since the components $\frak{k}$ and $\frak{p}$ of the Cartan decomposition $\frak{g}=\frak{k}\oplus\frak{p}$ are
$$
\frak{k}=\left\{
	\mat{
	\xi_1&\\
	&\xi_2
}\ \Bigg|\
\parbox{2cm}{$\xi_1\in\frak{so}(p)$,\\ $\xi_2\in\frak{so}(q+1)$}
\right\},\quad
\frak{p}=\left\{
\mat{
	&\transp \eta\\
	\eta&
}
\ \Bigg|\
\eta\in\R^{p\times(q+1)}
\right\},
$$
we may identify $\frak{p}$ with $\R^{p\times(q+1)}$ in such a way that
the action of $\diag(\xi_1,\xi_2)\in\frak{k}$ sends $\eta\in\R^{p\times(q+1)}$ to $\xi_2\eta-\eta\xi_1$,
while the Killing form $B$ on $\frak{p}$ is $(4n-2)$-times the standard metric of $\R^{p\times(q+1)}$.
Therefore, we conclude that $f^*\T X$ identifies locally with the trivial $\R^{p\times(q+1)}$-bundle endowed with the metric $\langle \eta_1,\eta_2\rangle=(4n-2)\Tr(\transp{\eta}_1 \eta_2)$ and the connection $\wh{\nabla}$ defined by
$$
\wh{\nabla} \eta=\dif \eta+
\scalebox{0.9}{$
\mat{	0&&&\\[-0.1cm]&\hspace{-0.1cm}\Omega_2&&\\&&\hspace{-0.2cm}\Omega_4&\\[-0.1cm]&&&\hspace{-0.1cm}\ddots}
$}
\eta-\eta
\scalebox{0.9}{$
\mat{\Omega_1&&\\&\hspace{-0.2cm}\Omega_3&\\[-0.1cm]&&\hspace{-0.15cm}\ddots}
$}
,
$$
while $\dif f\in\Omega^1(\Sigma,f^*\T X)$ is locally the $\R^{p\times(q+1)}$-valued $1$-form $\Gamma$ in \eqref{eqn_Gamma}.

Next, we verify the required expression \eqref{eqn_thmgauss} of $f^*\ve{g}_X$. Using the above identifications, we obtain by computations that
$$
f^*\ve g_X(\sth,\sth)=B(\dif f(\sth),\dif f(\sth))=(4n-2)\Tr\big[\transp\Gamma(\sth)\Gamma(\sth)\big]=(4n-2)\Big\{\transp\theta(\sth)\theta(\sth)+\sum_{j=2}^n\Tr\big[\transp\ve\alpha_j(\sth)\ve\alpha_j (\sth)\big]\Big\},
$$
where ``$\,\Tr\,$'' takes the matrix trace of each matrix-valued $2$-tensor to produce a scalar $2$-tensor. 
Since $\theta=\transp(\theta_1,\theta_2)$ is the frame of $\T^*\Sigma$ dual to the orthonormal frame $(u_1,v_1)$ of $\T\Sigma$, the term $\transp\theta(\sth)\theta(\sth)$ is just the first fundamental form $\ve g$. To compute the term $\Tr\big[\transp\ve\alpha_j(\sth)\ve\alpha_j (\sth)\big]$, we pick a coordinate $z$ of $\Sigma$ and a holomorphic local frame $\ell_i$ of $\TT_i$ as in Theorem \ref{thm_affinetoda} and let $(u_i,v_i)$ be as in the proof of that theorem. Using the expression of $\ve\alpha_j$ given in that proof, we get, for $2\leq j\leq n-1$,
\begin{align*}
	\Tr\big[\transp\ve\alpha_j(\sth)\ve\alpha_j(\sth)\big]&=\tfrac{h_j}{h_{j-1}}\Tr
	\scalebox{0.9}{
		$\mat{\re(\alpha_j\dz)&\im(\alpha_j\dz)\\[0.1cm]-\im(\alpha_j\dz)&\re(\alpha_j\dz)}$}
	\otimes
	\scalebox{0.9}{
		$\mat{\re(\alpha_j\dz)&-\im(\alpha_j\dz)\\[0.1cm]\im(\alpha_j\dz)&\re(\alpha_j\dz)}$}
	\\
	&=\tfrac{2\,h_j}{h_{j-1}}\big(\re(\alpha_j\dz)\otimes\re(\alpha_j\dz)+\im(\alpha_j\dz)\otimes\im(\alpha_j\dz)\big)\\
	&=\tfrac{2\,h_j}{h_{j-1}}|\alpha_j|^2(\dx^2+\dy^2)=\tfrac{2\,h_j}{h_1h_{j-1}}|\alpha_j|^2h_1(\dx^2+\dy^2)=2\|\ve\alpha_j\|^2\ve{g}.
\end{align*}
As for $\ve\alpha_n=\ve\alpha_n^++\ve\alpha_n^-$, we have
$$
\Tr\big[\transp\ve\alpha_n(\sth)\ve\alpha_n(\sth)\big]=\Tr\big[\transp\ve\alpha_n^+(\sth)\ve\alpha_n^+(\sth)\big]+\Tr\big[\transp\ve\alpha_n^-(\sth)\ve\alpha_n^-(\sth)\big]+\Tr\big[\transp\ve\alpha_n^+(\sth)\ve\alpha_n^-(\sth)\big]+\Tr\big[\transp\ve\alpha_n^-(\sth)\ve\alpha_n^+(\sth)\big].
$$
By similar computations, the last two terms vanish and the first two terms equal $2\|\ve\alpha_n^+\|^2\ve g$ and $2\|\ve\alpha_n^-\|^2\ve g$, respectively. Putting the results together, we get \eqref{eqn_thmgauss}. 

This already implies that $f$ is a conformal immersion. We finally show that $f$ is minimal via Lemma \ref{lemma_conformalminimal}. Note that the above local description of $f^*\T X$  effectively identifies $f^*\T X$ globally as the homomorphism bundle $\Hom_\R(\TT_1\oplus\TT_3\oplus\cdots,\ \underline{\R}\oplus\TT_2\oplus\TT_4\oplus\cdots)$ with metric and connection induced by those on the trivial line bundle $\underline{\R}$ and each $\TT_i$. 
In particular, 
\begin{equation}\label{eqn_homkl}
\Hom_\R(\TT_k,\TT_l),\ \  \Hom(\TT_k,\underline{\R})\quad  (k\in\{1,3,\cdots\},\ l\in\{2,4,\cdots\})
\end{equation}
 are parallel subbundles orthogonal to each other.
Meanwhile, by the expression \eqref{eqn_Gamma}, the only nonzero components of $\dif f=\Gamma$ are  $\ve\alpha_l\in\Omega^1(\Sigma,\Hom_\R(\TT_{l-1},\TT_l))$, $\ve\alpha_k^*\in\Omega^1(\Sigma,\Hom_\R(\TT_k,\TT_{k-1}))$ and $\ve{g}\in\Omega^1(\Sigma,\Hom_\R(\TT_1,\underline{\R}))=C^\infty(\Sigma,\T^*\Sigma\otimes\T^*\Sigma)$ (the last one corresponds to the term $\transp\theta$). To apply Lemma \ref{lemma_conformalminimal}, we need to give $f^*\T X$ an orthogonal, parallel, complex structure under which $\dif f$ is a $(1,0)$-form. Since each $\TT_i$ is a hermitian holomorphic vector bundle (Theorem \ref{thm_holo}), so is every subbundle in \eqref{eqn_homkl} (see the paragraph preceding Lemma \ref{lemma_orientation}), and the component $\ve\alpha_l$ is a $(1,0)$-form by Theorem \ref{thm_holo}. Also, the component $\ve g$ is a $(1,0)$-form since the metric $\ve{g}$ is conformal. On the other hand, the component $\ve\alpha_k^*$ is a $(0,1)$-form, but we can invert the complex structure on $\Hom_\R(\TT_k,\TT_{k-1})$ to make it a $(1,0)$-from. Thus, we conclude that $f^*\T X$ does admit a complex structure fulfilling the assumption of Lemma \ref{lemma_conformalminimal}, so we can apply the lemma to complete the proof.
\end{proof}

\subsection{$\ac{J}$-holomorphic curves in $\H^{4,2}$ as A-surfaces}\label{subsec_jholoA}
We now restrict to the case $n=3$ and consider the interplay between A-surfaces $\Sigma\subset\H^{4,2}$ and the almost complex structure $\ac{J}$ on  $\H^{4,2}$ (see \S \ref{subsec_almostcomplex}). 
By Lemma \ref{lemma_sphere} and the description of $X_{G_2'}$ in Corollary \ref{coro_xg2}, we immediately deduce from Theorem \ref{thm_gaussmap}:
\begin{corollary}\label{coro_gaussholo}
For any A-surface $\Sigma\subset\H^{4,2}$ which is not part of a totally geodesic $\H^2\subset\H^{4,2}$, the following conditions are equivalent to each other:
\begin{itemize}
\item the Gauss map $f:\Sigma\to X_{\SO_0(3,4)}$ takes values in the totally geodesic submanifold $X_{G_2'}\subset X_{\SO_0(3,4)}$ (hence is a conformal minimal immersion into $X_{G_2'}\,$);
\item the subbundle bundle $\TT_2\subset\T \H^{4,2}|_\Sigma$ is preserved by $\ac{J}$.
\end{itemize}
\end{corollary}

A particularly interesting subclass of A-surfaces satisfying these conditions are:
\begin{proposition}\label{prop_holomorphicA}
Let $\Sigma\subset\H^{4,2}$ be a spacelike $\ac{J}$-holomorphic curve with nowhere vanishing second fundamental form. Then $\Sigma$ is an A-surface if and only if the osculation line $\{\II(X,Y)\mid X,Y\in\T_z\Sigma\}\subset\ca{N}_z$ (see \S \ref{subsec_jholominimal}) is negative definite for all $z\in\Sigma$. In this case, the whole Frenet splitting $\T\H^{4,2}|_\Sigma=\TT_1\oplus\TT_2\oplus\TT_3$ is preserved by $\ac{J}$. In particular, $\Sigma$ satisfies the conditions in Corollary \ref{coro_gaussholo}.
\end{proposition}
\begin{proof}
If $\Sigma$ is a $\ac{J}$-holomorphic curve and an A-surface at the same time, then the fiber of $\TT_2$ at any $z\in \Sigma$ where $\II(\sth,\sth)$ is nonzero is exactly the osculation line at $z$, so the ``only if'' part is trivial. Conversely, assuming that the osculation lines are negative definite, we let $\TT_2$ be the subbundle of the normal bundle $\ca{N}$ formed by these complex lines and $\TT_3$ be the orthogonal complement of $\TT_2$ in $\ca{N}$. Since the almost complex structure $\ac{J}$ of $\H^{4,2}$ is orthogonal, the subbundles $\TT_1=\T\Sigma$, $\TT_2$ and $\TT_3$ of $\T\H^{4,2}|_\Sigma$ form a Frenet splitting preserved by $\ac{J}$. Also, $\ve\alpha_2(X)=\II(X,\sth):\TT_1\to\TT_2$ is conformal because $\II(\sth,\sth)$ is complex bilinear by Proposition \ref{prop_kahler} \ref{item_propkahler2}. Therefore, $\Sigma$ is an A-surface. This proves the ``if'' part and the ``In this case'' statement.
\end{proof}

\begin{remark}
The A-surfaces that we construct from cyclic $G_2'$-Higgs bundles in \S \ref{subsec_g2higgs} below belong to the above subclass. They also have the following extra properties by construction:
\begin{itemize}
	\item $\ve{\alpha}_3^+$ is nowhere zero as well; as a consequence, $\TT_1$, $\TT_2$ and $\TT_3$ identify with $\ca{K}^{-1}$, $\ca{K}^{-2}$ and $\ca{K}^{-3}$, respectively, as holomorphic line bundles; 
	\item the hermitian metrics $\ve{h}_1$, $\ve{h}_2$ and $\ve{h}_3$ on these line bundles satisfy $\ve{h}_3=\ve{h}_1\ve{h}_2/4$.
\end{itemize}
By using $G_2'$-moving frames, it can be shown that these properties are not ad hoc and are local in nature. Namely, they are satisfied by every A-surface $\Sigma$ as in Proposition \ref{prop_holomorphicA}. 
\end{remark}

\begin{remark}\label{rk_local}
	As the concluding remark for this section, we note that all the results till now are local in nature. Therefore, Theorems \ref{thm_holo}, \ref{thm_affinetoda} and Proposition \ref{prop_holomorphicA} still hold if the ambient space $\H^{p,q}$ is replaced by an orientable pseudo-Riemannian manifold locally modeled on $\H^{p,q}$, or more generally if $\Sigma$ is replaced by an \emph{A-surface immersion} from the universal cover $\wt{\Sigma}$ of an abstract surface $\Sigma$ to $\H^{p,q}$ which is equivariant with respect to some representation $\rho$ of $\pi_1(\Sigma)$ in $\SO_0(p,q+1)$ or $G_2'$. In particular, we can treat the structural data of such an immersion as defined on $\Sigma$ rather than $\wt\Sigma$ by virtue of the invariance under the $\pi_1(\Sigma)$ action. The definition and results concerning the Gauss map $f$ also hold in this equivariant setting, and $f$ is $\rho$-equivariant as well.
\end{remark}

\section{Infinitesimal rigidity and unique determination from Gauss map}\label{sec_infinitesimal}

In this section, we restrict to A-surfaces $\Sigma$ satisfying an extra hypothesis related to Dai-Li's work \cite{dai-li_minimal} and prove two result for them: the first is the core result of this paper that $\Sigma$ is infinitesimally rigid; the second is that $\Sigma$ is determined by its Gauss map up to the antipodal map of $\H^{p,q}$.

\subsection{Hypothesis \eqref{eqn_hypo} and a sufficient condition}\label{subsec_hypo}
Fix an oriented A-surface $\Sigma$ in an orientable pseudo-Riemannian manifold $M$ locally modeled on $\H^{p,q}$ with $n=\frac{p+q}{2}\geq3$. Also assume that $\Sigma$ is not locally contained in any pseudo-hyperbolic subspace of codimension $4$, so that we have a unique Frenet splitting $\T M|_\Sigma=\TT_1\oplus\cdots\oplus\TT_n$ by Theorem \ref{thm_holo} and Remark \ref{rk_local}, with each $\TT_i$ being a hermitian holomorphc line bundle. The extra hypothesis that we consider in this section can be simply stated as \emph{the first $n-1$ line bundles $\TT_1,\cdots,\TT_{n-1}$ have semi-negative Chern forms}. Here, the \emph{Chern form} is $\tfrac{\ima}{2\pi}$-times the curvature form of the Chern connection, which is an $\R$-valued $2$-form representing the first Chern class, and we call a $2$-form $\omega$ \emph{semi-negative} if the open set $\{\omega\neq0\}\subset\Sigma$ is dense and $\omega$ is opposite to the orientation of $\Sigma$ on this set. Although this notion of negativity depends on the choice of orientation for $\Sigma$, the hypothesis as a whole does not, because reversing the orientation of $\Sigma$ will also reverse the one of each $\TT_i$ with $2\leq i\leq n-1$ (see Remark \ref{rk_orientation}), hence will reverse the complex structure and Chern form of $\TT_i$ as well.

\begin{remark}
It is reasonable that the semi-negativity assumption is not imposed on $\TT_n$: unlike $\TT_i$ with $2\leq i\leq n-1$, whose orientation (and hence complex structure) is uniquely determined, we are free to choose the orientation of $\TT_n$, which just reflects the orientation of the ambient space $M$ (see Remark \ref{rk_orientation}), and such a choice is irrelevant to what we are doing.
\end{remark}

By virtue of Theorem \ref{thm_affinetoda}, we can reformulate the hypothesis as follows. Since the Chern form of $\TT_i$ is expressed locally as $\tfrac{\ima}{2\pi}\bpa\pa\log h_i=-\tfrac{1}{\pi}\pazbz\log h_i\dx\wedge\dy$, the hypothesis is equivalent to (see Theorem \ref{thm_affinetoda} and Remark \ref{rk_norm} for the notations):
\begin{equation}\tag{$\star$}\label{eqn_hypo}
\parbox{2.9cm}{on an open dense subset of $\Sigma$, }
	\begin{cases}
		\tfrac{1}{h_1}\pazbz\log h_1=\tfrac{1}{2}-\|\ve\alpha_2\|^2>0,\\[0.2cm]
		\tfrac{1}{h_1}\pazbz\log h_k=\|\ve\alpha_k\|^2-\|\ve\alpha_{k+1}\|^2>0\quad (2\leq k\leq n-2),\\[0.2cm]
	\tfrac{1}{h_1}\pazbz\log h_{n-1}=\|\ve\alpha_{n-1}\|^2-\|\ve\alpha_n^+\|^2-\|\ve\alpha_n^-\|^2>0.
	\end{cases}
\end{equation}

This clearly implies some restrictions on the zeros of the holomorphic forms $\ve\alpha_2,\cdots,\ve\alpha_{n-1},\ve\alpha_n^\pm$. For example, the last inequality in \eqref{eqn_hypo} forces every zero of $\ve\alpha_{n-1}$ to be a zero of $\ve\alpha_n^+$ and $\ve\alpha_n^-$ at the same time.
More specifically, letting $(\ve\alpha)$ denote the divisor\footnote{When $\ve\alpha\not\equiv0$, the divisor $(\ve\alpha)$ is just the $\Z_{\geq0}$-valued function on $\Sigma$ assigning to every $z\in\Sigma$ the multiplicity of zero of $\ve\alpha$ at $z$. When $\ve\alpha\equiv0$, we let $(\ve\alpha)$ be the constant function $+\infty$ by convention.}
 of any holomorphic section $\ve\alpha$, one may check that a necessary condition for \eqref{eqn_hypo} is
\begin{equation}\label{eqn_divisorleq}
	(\ve\alpha_2)\leq \cdots\leq (\ve\alpha_{n-1})\leq \min\big\{(\ve\alpha_n^+),(\ve\alpha_n^-)\big\},
\end{equation}
where ``$\leq$'' and ``$\min$'' are pointwise inequality and minimum of $\Z_{\geq0}\cup\{+\infty\}$-valued functions on $\Sigma$.

When $\Sigma$ is closed, we generalize the work of Dai and Li \cite{dai-li_minimal} to get the theorem below on the affine Toda system. It implies that a specific condition stronger than \eqref{eqn_divisorleq} is sufficient for \eqref{eqn_hypo}, although we do not know whether \eqref{eqn_divisorleq} itself is sufficient.
The condition is given by replacing every ``$\leq$'' in \eqref{eqn_divisorleq} by a stronger relation ``$\prec$'': here ``$D_1\prec D_2$'' means that the divisor $D_1$ is strictly less than that of $D_2$ at every point where the former is nonzero (which always holds when $D_1=\emptyset$).
\begin{theorem}\label{thm_daili}
	Let $\Sigma$ be a closed Riemann surface of genus $\geq2$, $\TT_1,\cdots,\TT_n$ ($n\geq2$) be holomorphic line bundles on $\Sigma$ and $\ve{h}_i$ be a hermitian metric on $\TT_i$. 
	Let $\ve\alpha_1$, $\ve\alpha_j$ ($2\leq j\leq n-1$) and $\ve\alpha_n^\pm$
	be holomorphic sections of the line bundles  $\ca{K}\!\TT_1$, $\ca{K}\!\TT_{j-1}^{-1}\TT_j$ and $\ca{K}\!\TT_{n-1}^{-1}\TT_n^{\pm1}$, respectively, with $\ve\alpha_1,\cdots,\ve\alpha_{n-1}\not\equiv0$. Suppose the following equations (with the same local notation as in Theorem \ref{thm_affinetoda}) hold on $\Sigma$:
	\begin{align}
		&\text{if $n=2$:}\quad
		\begin{cases}
			\pazbz\log h_1=h_1|\alpha_1|^2-\tfrac{h_2}{h_1}|\alpha_2^+|^2-\frac{1}{h_1h_2}|\alpha_2^-|^2,\\[0.2cm]
			\pazbz\log h_2=\tfrac{h_2}{h_1}|\alpha_2^+|^2-\frac{1}{h_1h_2}|\alpha_2^-|^2;
		\end{cases} \label{eqn_dailin2}\\
		&\text{if $n\geq3$:}\quad
		    \begin{cases}
			\pazbz\log h_1=h_1|\alpha_1|^2-\tfrac{h_2}{h_1}|\alpha_2|^2,\\[0.2cm]
			\pazbz\log h_k=\tfrac{h_k}{h_{k-1}}|\alpha_k|^2-\tfrac{h_{k+1}}{h_k}|\alpha_{k+1}|^2\quad \text{for $2\leq k\leq n-2$},\\[0.2cm]
			\pazbz\log h_{n-1}=\tfrac{h_{n-1}}{h_{n-2}}|\alpha_{n-1}|^2-\tfrac{h_{n}}{h_{n-1}}|\alpha_n^+|^2-\tfrac{1}{h_{n-1}h_n}|\alpha_n^-|^2,\\[0.2cm]
			\pazbz\log h_n=\tfrac{h_{n}}{h_{n-1}}|\alpha_n^+|^2-\tfrac{1}{h_{n-1}h_n}|\alpha_n^-|^2.\label{eqn_daili}	
		\end{cases}
	\end{align}
	Then the following statements hold.
	\begin{enumerate}[label=(\arabic*)]
		\item\label{item_daili1}
		If 
$(\ve\alpha_1)\prec(\ve\alpha_2)\cdots\prec (\ve\alpha_{n-1})\prec \min\big\{(\ve\alpha_n^+),(\ve\alpha_n^-)\big\}$,
		then we have $\pazbz\log h_i\geq0$ on $\Sigma$ for $i=1,\cdots, n-1$, with strict inequality away from the zeros of $\ve\alpha_i$.
		\item\label{item_daili2}
		If  $(\ve\alpha_n^+)\prec(\ve\alpha_n^-)\neq\emptyset$, then $\pazbz\log h_n\geq0$ on $\Sigma$, with strict inequality away from the zeros of $\ve\alpha_n^+$.
	\end{enumerate}
\end{theorem}

The proof is directly adapted from Dai-Li's and is postponed to the appendix. 
\begin{remark}
	For the sake of completeness, we have given in Theorem \ref{thm_daili} a statement that is stronger than required in two ways. First,
	the setting is more general than the one we are working on: the line bundle $\TT_1$ is not necessarily $\ca{K}^{-1}$, so an extra $\ve\alpha_1\in H^0(\Sigma,\ca{K}\TT_1)$ comes up correspondingly. Second, Parts \ref{item_daili1} and \ref{item_daili2} are about the left-hand side of the 1st to $(n-1)$th and the last equations in \eqref{eqn_daili}, respectively, but we only need the former. The original result of Dai-Li (see \cite[Lemma 5.4]{dai-li_minimal}) is the two parts combined (for $(\ve\alpha_1)=\cdots=(\ve\alpha_{n-1})=(\ve\alpha_n^+)=\emptyset$), without noticing that they are unrelated.
\end{remark}
\begin{remark}\label{rk_reduction}
Without affecting the statement, we can add to Theorem \ref{thm_daili} the extra assumption that $\ve\alpha_n^+\not\equiv0$. In fact, if $\ve\alpha_n^-\not\equiv0$ but $\ve\alpha_n^+\equiv0$, we can switch the roles of $\TT_n$ and $\TT_n^{-1}$, which interchanges $\ve\alpha_n^\pm$ and does affect the statement. On the other hand, if $\ve\alpha_n^+=\ve\alpha_n^-\equiv0$, then the statement just reduces from $n$ to the $n-1$ case, with $\ve\alpha_{n-1}^-\equiv0$ and $\ve\alpha_{n-1}^+$ being the original $\ve\alpha_{n-1}$, whereas the $n=2$ case reduces to a trivial statement. 
\end{remark}

\subsection{Infinitesimal rigidity}
An A-surface in a $4$-dimensional pseudo-Riemannian manifold is just a spacelike minimal surface with timelike normal bundle (see Remark \ref{rk_superminimalisminimal}). In this case, by Proposition \ref{prop_maximal}, $\Sigma$ locally maximizes the area as long as $M$ is negatively curved. We show the following generalization to higher dimensions (see Theorem \ref{thm_holo}, Remark \ref{rk_local} and \S \ref{subsec_variation} for notation and background):
\begin{theorem}\label{thm_Avariation}
	Let $\Sigma$ be an A-surface in an orientable pseudo-Riemannian manifold $M$ locally modeled on $\H^{p,q}$ with $n=\frac{p+q}{2}\geq3$. Suppose $\Sigma$ is not locally contained in any pseudo-hyperbolic subspace of codimension $4$ and its structural data satisfies \eqref{eqn_hypo}. Consider the spacelike part $\ca{N}^+:=\TT_3\oplus \TT_5\oplus \cdots$ and timelike part $\ca{N}^-:=\TT_2\oplus \TT_4\oplus \cdots$ of the normal bundle $\ca{N}$ of $\Sigma$. Then for any compactly supported variation of $\Sigma$ whose variational vector field $\xi$ is a section of $\ca{N}^+$ (resp.\ $\ca{N}^-$) and is not identically zero, the associated second variation of volume is positive (resp.\ negative). 
\end{theorem}
\begin{remark}\label{rk_h3}
Besides maximal surfaces in $\H^{2,q}$, minimal surfaces in hyperbolic $3$-manifolds have also been extensively studied (see e.g.\ \cite{huang-marcello-tarantello,taubes, uhlenbeck}). For these surfaces,
we cannot expect such kind of variational property to hold in general, because the term in the second variation formula given by the second fundamental form $\II$ has different sign from the Laplacian term and the curvature term, which makes the sign of the whole formula indefinite. 
The common feature of maximal surfaces and A-surfaces which frees us from this issue is that 
all osculation vectors are timelike (see Remark \ref{rk_secondff}), hence the $\II$-term has the correct sign.
\end{remark}
\begin{proof}
Recall from Theorem \ref{thm_variation} \ref{item_thmvariation1} that the second variation of volume is given by the integral
$$
\int_\Sigma\Tr_{\ve g}\Big(\langle \nabla_\sth^\ca{N} \xi,\nabla_\sth^\ca{N}\xi\rangle-\langle \ve{S}_\xi(\sth),\ve{S}_\xi(\sth)\rangle-\langle R_M(\sth,\xi)\sth,\xi\rangle\Big)\dif\vol_\ve{g}
$$
(where the first fundamental form $\ve{g}$ is the real part of the hermitian metric $\ve{h}_1$ on $\TT_1=\T\Sigma$). To prove the theorem, we shall show that the integrand satisfies
\begin{equation}\label{eqn_Avariation1}
	\Tr_{\ve g}\Big(\langle \nabla_\sth^\ca{N} \xi,\nabla_\sth^\ca{N}\xi\rangle-\langle \ve{S}_\xi(\sth),\ve{S}_\xi(\sth)\rangle-\langle R_M(\sth,\xi)\sth,\xi\rangle\Big)\ 
	\left\{\hspace{-0.1cm}
	\begin{array}{l}
	\geq0\ \text{ if $\xi\in C^\infty(\Sigma,\ca{N}^+)$}\\
	\leq0\ \text{ if $\xi\in C^\infty(\Sigma,\ca{N}^-)$}
	\end{array}
    \right.
\end{equation}
everywhere on $\Sigma$, with strict inequality on an open dense subset of $\{\xi\neq0\}$.

To this end, suppose the connection $\nabla^\ca{N}$ decomposes under the splitting  $\ca{N}=\ca{N}^+\oplus\ca{N}^-$ as
\begin{equation}\label{eqn_npm}
\nabla^\ca{N}=
\mat{\nabla^{\ca{N}^+}&\Gamma_0^-\\\Gamma_0^+&\nabla^{\ca{N}^-}},
\end{equation}
where $\nabla^{\ca{N}^\pm}$ is a connection on $\ca{N}^\pm$ and $\Gamma_0^\pm$ is a $\Hom(\ca{N}^\pm,\ca{N}^\mp)$-valued $1$-form. Then the left-hand side of \eqref{eqn_Avariation1} can be rewritten as
$$
\underbrace{\Tr_{\ve g}\langle \nabla_\sth^{\ca{N}^\pm} \xi,\nabla_\sth^{\ca{N}^\pm}\xi\rangle}_{\cir{1}}+\underbrace{\Tr_{\ve g}\langle \Gamma_0^\pm(\sth)\xi,\Gamma_0^\pm(\sth)\xi\rangle}_{\cir{2}}\ \underbrace{-\Tr_{\ve g}\langle \ve{S}_\xi(\sth),\ve{S}_\xi(\sth)\rangle}_{\cir{3}}\ \underbrace{-\Tr_{\ve g}\langle R_M(\sth,\xi)\sth,\xi\rangle}_{\cir{4}}~,
$$
where every ``$\pm$'' is taken to be $+$ or $-$ when $\xi$ is a section of $\ca{N}^+$ or $\ca{N}^-$, respectively. We shall prove \eqref{eqn_Avariation1} by showing that the terms \cir{1}, \cir{3} and \cir{4} have the correct sign for trivial reason (similarly as in the proof of Proposition \ref{prop_maximal}), whereas the term \cir{2} is controlled by \cir{4} by virtue of Hypothesis \eqref{eqn_hypo}.

For the term \cir{1}, since $\ca{N}^+$ and $\ca{N}^-$ are positive and negative definite respectively, we immediately have $\cir{1}\geq0$ and $\cir{1}\leq0$ for $\xi\in C^\infty(\Sigma,\ca{N}^+)$ and $\xi\in C^\infty(\Sigma,\ca{N}^-)$, respectively.

For the term \cir{3}, recall from  Remark \ref{rk_secondff} that the shape operator $\ve{S}_\xi\in C^\infty(\Sigma,\End(\T\Sigma))$ assigned to a normal vector field $\xi$ is identically zero when $\xi$ is orthogonal to $\TT_2$. In particular, we have $\ve{S}_\xi=0$ if $\xi\in C^\infty(\Sigma,\ca{N}^+)$. Therefore, given an orthonormal local frame $(e_1,e_2)$ of $\T\Sigma$, we have
$$
\cir{3}=-\sum_{i=1}^2\langle \ve{S}_\xi(e_i),\ve{S}_\xi(e_i)\rangle\ 
\left\{\hspace{-0.1cm}
\begin{array}{l}
=0\ \text{ if $\xi\in C^\infty(\Sigma,\ca{N}^+)$},\\
\leq 0\ \text{ if $\xi\in C^\infty(\Sigma,\ca{N}^-)$}.
\end{array}
\right.
$$

For the term  \cir{4}, note that at any $z\in\Sigma$ with $\xi(z)\neq0$, the sectional curvature of $M$ along the tangent $2$-plane spanned by $e_i$ and $\xi$ is 
$K(e_i\wedge\xi)=\frac{\langle R_M(e_i,\xi)e_i,\xi\rangle}{\langle \xi,\xi\rangle}=-1$. Therefore, at the point $z$, we have
$$
\cir{4}=-\sum_{i=1}^2\langle R_M(e_i,\xi)e_i,\xi\rangle=2\langle\xi,\xi\rangle.
$$
This clearly holds at the points where $\xi=0$ as well, hence holds throughout $\Sigma$. In particular, we have $\cir{4}\geq0$ and $\cir{4}\leq0$ for $\xi\in C^\infty(\Sigma,\ca{N}^+)$ and $\xi\in C^\infty(\Sigma,\ca{N}^-)$, respectively.

In the rest of the proof, we show that the absolute value of the term \cir{2} is controlled by that of \cir{4}:
\begin{equation}\label{eqn_Avariation2}
|\Tr_{\ve g}\langle \Gamma_0^\pm(\sth)\xi,\Gamma_0^\pm(\sth)\xi\rangle|\leq 2|\langle\xi,\xi\rangle|\ \text{ for any }\xi\in C^\infty(\Sigma,\ca{N}^\pm),
\end{equation}
and moreover the inequality is strict on an open dense subset of $\{\xi\neq0\}$. This implies the required inequality \eqref{eqn_Avariation1} and prove the theorem.

In order to show \eqref{eqn_Avariation2}, we first deduce from the decomposition \eqref{eqn_nablamatrix2} of the connection $\nabla$ on $\T M|_\Sigma$ the following more explicit expression of $\Gamma_0^+$ (which is just the $\Gamma$ in the proof of Theorem \ref{thm_gaussmap} with the first row and first column removed):
$$
\Gamma_0^+=
\scalebox{0.9}{$
\mat{
	\ve\alpha_3^*&&&\\
	\ve\alpha_4&\ve\alpha_5^*&&\\
	&\ddots&\ddots&\\
	&&\ddots&\ve\alpha_{n-1}^*\\
	&&&\ve\alpha_n}
$}
\text{ if $n$ is even,}\quad
\Gamma_0^+=
\scalebox{0.9}{$
\mat{
	\ve\alpha_3^*&&&&\\
	\ve\alpha_4&\ve\alpha_5^*&&&\\
	&\ddots&\ddots&&\\
	&&\ddots&\ve\alpha_{n-2}^*&\\
	&&&\ve\alpha_{n-1}&\ve\alpha_n^*}
$}
\text{ if $n$ is odd.}
$$

We proceed to compute $\Tr_\ve{g}\langle \Gamma_0^\pm(\sth)\xi, \Gamma_0^\pm(\sth)\xi\rangle$. Pick an orthonormal local frame $(u_i,v_i)$ of each $\TT_i$ as in the proof of Theorem \ref{thm_affinetoda}, so that $\ve\alpha_j$ can be viewed as a $2\times 2$-matrix of $1$-forms with the explicit expression given in that proof, while $\ve\alpha_j^*$ and $\Gamma_0^-$ are the transposes of $\ve\alpha_j$ and $\Gamma_0^+$ respectively. We also use the frame to view a section $\xi$ of $\ca{N}^+$ or $\ca{N}^-$ as a column vector of functions. Then we may write
$\langle\xi,\xi\rangle=\pm\transp\xi\,\xi$ and $\langle \Gamma_0^\pm(\sth)\xi,\,\Gamma_0^\pm(\sth)\xi\rangle=\mp \transp\xi\ \transp{\Gamma_0}^\pm(\sth)\ \Gamma_0^\pm(\sth)\xi$ for any $\xi\in C^\infty(\Sigma,\ca{N}^\pm)$.
Thus, when $n$ is even, we get by matrix computations that
$$
\langle \Gamma_0^+(\sth)\xi, \Gamma_0^+(\sth)\xi\rangle=-\transp\xi
\scalebox{0.9}{$
\mat{
	\scalebox{0.85}{\parbox{1.2cm}{$\ve\alpha_3\transp \ve\alpha_3$\\$+\transp \ve\alpha_4\ve\alpha_4$}}&\transp \ve\alpha_4\transp \ve\alpha_5&&\\[0.3cm]
	\ve\alpha_5\ve\alpha_4&\scalebox{0.85}{\parbox{1.2cm}{$\ve\alpha_5\transp \ve\alpha_5$\\$+\transp \ve\alpha_6\ve\alpha_6$}}&\ddots&\\[0.3cm]
	&\ddots&\ddots&\transp \ve\alpha_{n-1} \transp \ve\alpha_n\\[0.3cm]
	&&\hspace{-0.5cm} \ve\alpha_n\ve\alpha_{n-1}&\scalebox{0.85}{\parbox{1.6cm}{$\ve\alpha_{n-1}\transp \ve\alpha_{n-1}$\\$+\transp \ve\alpha_n\ve\alpha_n$}}
}$}
\xi\quad \text{ for }\xi\in C^\infty(\Sigma,\ca{N}^+),
$$
$$
\langle \Gamma_0^-(\sth)\xi, \Gamma_0^-(\sth)\xi\rangle=\transp\xi
\scalebox{0.9}{$
\mat{
	\transp \ve\alpha_3\ve\alpha_3&\transp \ve\alpha_3\transp \ve\alpha_4&&&\\[0.3cm]
	\ve\alpha_4\ve\alpha_3&\scalebox{0.85}{\parbox{1.2cm}{$\ve\alpha_4\transp\ve\alpha_4$\\$+\transp\ve\alpha_5\ve\alpha_5$}}&\ddots&&\\[0.3cm]
	&\ddots&\ddots&\ddots&\\[0.3cm]
	&&\ddots&\scalebox{0.85}{\parbox{1.9cm}{$\ve\alpha_{n-2}\transp\ve\alpha_{n-2}$\\$+\transp\ve\alpha_{n-1}\ve\alpha_{n-1}$}}&\transp \ve\alpha_{n-1} \transp \ve\alpha_n\\[0.3cm]
	&&&\ve\alpha_n\ve\alpha_{n-1}& \ve\alpha_n\transp\ve\alpha_n
}$}
\xi\quad \text{ for }\xi\in C^\infty(\Sigma,\ca{N}^-),
$$
where $\ve\alpha_3\transp\ve\alpha_3$ is understood as the $2\times 2$-matrix-valued $2$-tensor $\ve\alpha_3(\sth)\transp\ve\alpha_3(\sth)$, and similarly for the other terms in the matrices. When $n$ is odd, the results become
$$
\langle \Gamma_0^+(\sth)\xi, \Gamma_0^+(\sth)\xi\rangle=-\transp\xi
\scalebox{0.9}{$
\mat{
	\scalebox{0.85}{\parbox{1.2cm}{$\ve\alpha_3\transp\ve\alpha_3$\\$+\transp\ve\alpha_4\ve\alpha_4$}}&\transp \ve\alpha_4\transp \ve\alpha_5&&\\[0.3cm]
	\ve\alpha_5\ve\alpha_4&\ddots&\ddots&\\[0.3cm]
	&\ddots&\scalebox{0.85}{\parbox{1.9cm}{$\ve\alpha_{n-2}\transp\ve\alpha_{n-2}$\\$+\transp\ve\alpha_{n-1}\ve\alpha_{n-1}$}}&\transp \ve\alpha_{n-1} \transp \ve\alpha_n\\[0.3cm]
	&&\ve\alpha_n\ve\alpha_{n-1}& \ve\alpha_n\transp\ve\alpha_n
}$}
\xi\quad \text{ for }\xi\in C^\infty(\Sigma,\ca{N}^+),
$$
$$
\langle \Gamma_0^-(\sth)\xi, \Gamma_0^-(\sth)\xi\rangle=\transp\xi
\scalebox{0.9}{$
\mat{
	\ve\alpha_3\transp\ve\alpha_3&\transp \ve\alpha_3\transp \ve\alpha_4&&\\[0.3cm]
	\ve\alpha_4\ve\alpha_3&\scalebox{0.85}{\parbox{1.2cm}{$\ve\alpha_4\transp\ve\alpha_4$\\$+\transp\ve\alpha_5\ve\alpha_5$}}&\ddots&\\[0.3cm]
	&\ddots&\ddots&\transp \ve\alpha_{n-1} \transp \ve\alpha_n\\[0.3cm]
	&&\ve\alpha_n\ve\alpha_{n-1}& \scalebox{0.85}{\parbox{1.9cm}{$\ve\alpha_{n-1}\transp\ve\alpha_{n-1}$\\$+\transp\ve\alpha_n\ve\alpha_n$}}
}$}
\xi\quad \text{ for }\xi\in C^\infty(\Sigma,\ca{N}^-).
$$
The trace of each $2$-tensor in these matrices with respect to $\ve{g}$ (not to be confused with the matrix trace considered in the proof of Theorem \ref{thm_gaussmap}) is computed as follows:
\begin{lemma}\label{lemma_trggamma}
Let $I$ denote the $2\times 2$ identity matrix. For $2\leq i,j\leq n-1$, we have
\begin{align*}
	&\Tr_{\ve g}\big[\ve\alpha_i(\sth) \ve\alpha_j(\sth)\big]=\Tr_{\ve g}\big[\transp \ve\alpha_j(\sth) \transp \ve\alpha_i(\sth)\big]=0\quad \text{(the zero $2\times 2$ matrix)},\\
	&\Tr_{\ve g}\big[\transp \ve\alpha_i(\sth) \ve\alpha_i(\sth)\big]=\Tr_{\ve g}\big[\ve\alpha_i(\sth)\transp \ve\alpha_i(\sth)\big]=2\|\ve\alpha_i\|^2 I,\\
	&\Tr_{\ve g}\big[\ve\alpha_n(\sth)\ve\alpha_{i}(\sth)\big]=\Tr_{\ve g}\big[\transp \ve\alpha_i(\sth) \transp \ve\alpha_n(\sth)\big]=0,\\
	&\Tr_{\ve g}\big[\transp\ve\alpha_n(\sth)\ve\alpha_n(\sth)\big]=2(\|\ve\alpha_n^+\|^2+\|\ve\alpha_n^-\|^2)I,\\
	&\Tr_{\ve g}\big[\ve\alpha_n(\sth)\transp\ve\alpha_n(\sth)\big]=2\left[(\|\ve\alpha_n^+\|^2+\|\ve\alpha_n^-\|^2)I+\tfrac{2}{h_1h_{n-1}}
	\scalebox{0.9}{
	$\mat{\re(\alpha_n^+\cj{\alpha_n^-})&\im(\alpha_n^+\cj{\alpha_n^-})\\[0.1cm]\im(\alpha_n^+\cj{\alpha_n^-})&-\re(\alpha_n^+\cj{\alpha_n^-})}$}
	\right]
\end{align*}
(see Theorem \ref{thm_affinetoda} and Remark \ref{rk_norm} for the notations). Moreover, the eigenvalues of the right-hand side of the last equation are $2(\|\ve\alpha_n^+\|\pm\|\ve\alpha_n^-\|)^2$.
\end{lemma}
\begin{proof}
The expressions of the $\ve\alpha_j$'s given in the proof of Theorem \ref{thm_affinetoda} can be rewritten as 
\begin{align*}
	\ve\alpha_j&=\tfrac{\sqrt{h_j}}{\sqrt{h_{j-1}}}\left[
	\scalebox{0.9}{$
	\mat{\re(\alpha_j)&-\im(\alpha_j)\\[0.1cm]\im(\alpha_j)&\re(\alpha_j)}
	$}
	\dx+
	\scalebox{0.9}{$
	\mat{-\im(\alpha_j)&-\re(\alpha_j)\\[0.1cm]\re(\alpha_j)&-\im(\alpha_j)}
	$}
	\dy\right]\quad (2\leq j\leq n-1),\\
	\ve\alpha_{n-1}&=\ve\alpha_n^++\ve\alpha_n^-, \text{ where}\\	\ve\alpha_n^+&=\tfrac{\sqrt{h_n}}{\sqrt{h_{n-1}}}\left[
	\scalebox{0.9}{$
	\mat{\re(\alpha_n^+)&-\im(\alpha_n^+)\\[0.1cm]\im(\alpha_n^+)&\re(\alpha_n^+)}
	$}
	\dx+
	\scalebox{0.9}{$
	\mat{-\im(\alpha_n^+)&-\re(\alpha_n^+)\\[0.1cm]\re(\alpha_n^+)&-\im(\alpha_n^+)}
	$}
	\dy\right],\\
	\ve\alpha_n^-&=\tfrac{1}{\sqrt{h_{n-1}h_n}}\left[
	\scalebox{0.9}{$
	\mat{\re(\alpha_n^-)&-\im(\alpha_n^-)\\[0.1cm]-\im(\alpha_n^-)&-\re(\alpha_n^-)}
	$}
	\dx+
	\scalebox{0.9}{$
	\mat{-\im(\alpha_n^-)&-\re(\alpha_n^-)\\[0.1cm]-\re(\alpha_n^-)&\im(\alpha_n^-)}
	$}
	\dy\right].
\end{align*}
Taking account of the fact that
$$
\Tr_\ve{g}\big(\dx\otimes \dx\big)=\Tr_\ve{g}\big(\dy\otimes \dy\big)=\tfrac{1}{h_1}~,\quad \Tr_\ve{g}\big(\dx\otimes \dy\big)=\Tr_\ve{g}\big(\dy\otimes \dx\big)=0
$$
(because $\ve g=\re\ve h_1=h_1(\dx^2+\dy^2)$), we obtain the required equalities and the eigenvalues by direct computations.
\end{proof}
Using the expressions of $\langle \Gamma_0^\pm(\sth)\xi, \Gamma_0^\pm(\sth)\xi\rangle$ and Lemma \ref{lemma_trggamma}, we get expressions for $\Tr_\ve{g}\langle \Gamma_0^\pm(\sth)\xi, \Gamma_0^\pm(\sth)\xi\rangle$, from which it is clear that at any point of $\Sigma$, both of the quadratic form $\xi\mapsto -\Tr_\ve{g}\langle \Gamma_0^+(\sth)\xi, \Gamma_0^+(\sth)\xi\rangle$ and  $\xi\mapsto \Tr_\ve{g}\langle \Gamma_0^-(\sth)\xi, \Gamma_0^-(\sth)\xi\rangle$ have eigenvalues
belonging to the list
\begin{align*}
	&2\|\ve\alpha_3\|^2,\quad 2\|\ve\alpha_3\|^2+2\|\ve\alpha_4\|^2,\quad 2\|\ve\alpha_4\|^2+2\|\ve\alpha_5\|^2,\ \cdots\\
	&\quad\cdots,\ 2\|\ve\alpha_{n-2}\|^2+2\|\ve\alpha_{n-1}\|^2,\quad 2\|\ve\alpha_{n-1}\|^2+2\|\ve\alpha_n^+\|^2+2\|\ve\alpha_n^-\|^2,\quad 2(\|\ve\alpha_n^+\|\pm\|\ve\alpha_n^-\|)^2.
\end{align*}
But Hypothesis \eqref{eqn_hypo} implies that the following inequalities hold on an open dense set $U\subset\Sigma$:
\begin{equation}\label{eqn_Avariation3}
\tfrac{1}{2}(\|\ve\alpha_n^+\|\pm\|\ve\alpha_n^-\|)^2< \|\ve\alpha_n^+\|^2+\|\ve\alpha_n^-\|^2< \|\ve\alpha_{n-1}\|^2<\ \cdots\ <  \|\ve\alpha_3\|^2<  \|\ve\alpha_2\|^2< \tfrac{1}{2}.
\end{equation}
Therefore, all the eigenvalues are strictly less than $2$ on $U$. This implies the required inequality \eqref{eqn_Avariation2} and completes the proof.
\end{proof}

This theorem can be roughly understood as saying that A-surfaces satisfying the assumptions are saddle type critical points  for the area functional. Such critical points are nondegenerate. In fact, it is a simple linear algebraic fact that if a symmetric matrix $L$ has the form
$$
L=\mat{L_+&*\\ *&L_-}
$$
such that the blocks $L_+$ and $L_-$ are positive definite and negative definite respectively, then $L$ is nondegenerate.  Generalizing this to infinite dimensions and considering the Jacobi operator $L_\Sigma$ in place of $L$, we obtain:
\begin{corollary}\label{coro_Avariation1}
Under the hypotheses of Theorem \ref{thm_Avariation}, $\Sigma$ does not admit any nonzero compactly supported Jacobi field. 
\end{corollary}
\begin{proof}
	Let $\xi\in C^\infty(\Sigma,\ca{N})$ be a compactly supported Jacobi field and write $\xi=\xi_++\xi_-$, where $\xi_\pm$ is a compactly supported section of $\ca{N}^\pm$. Since  $L_\Sigma\xi=0$, we have
	$$
\langle L_\Sigma\xi_+,\,\xi_+\rangle+\langle L_\Sigma\xi_-,\,\xi_+\rangle=	\langle L_\Sigma\xi,\,\xi_+\rangle=0,
\quad \langle L_\Sigma\xi_-,\,\xi_-\rangle+\langle L_\Sigma\xi_+,\,\xi_-\rangle=\langle L_\Sigma\xi,\,\xi_-\rangle=0.
	$$ 
So the following equalities (between functions on $\Sigma$) hold:
	$$
	\langle L_\Sigma\xi_-,\,\xi_+\rangle=-\langle L_\Sigma\xi_+,\,\xi_+\rangle,\quad \langle L_\Sigma\xi_+,\,\xi_-\rangle=-\langle L_\Sigma\xi_-,\,\xi_-\rangle.
	$$ 
	The second variation formula (see Theorem \ref{thm_variation}) says that for any compactly supported $\eta\in C^\infty(\Sigma,\ca{N})$, the integral of $-\langle L_\Sigma\eta,\eta\rangle$ on $\Sigma$ is exactly the second variation of volume associated to $\eta$. Therefore, by Theorem \ref{thm_Avariation}, the integral of the first (resp.\ second) function above is $\geq0$ (resp.\ $\leq 0$), with strict inequality unless $\xi_+\equiv0$ (resp.\ $\xi_-\equiv0$). But the two integrals coincide because $L_\Sigma$ is self-adjoint (see \S \ref{subsec_jacobi}). So we conclude that $\xi_\pm$ are both identically zero.
\end{proof}

In particular, if $\Sigma$ is closed, then it does not admit any nonzero Jacobi field at all. Note that we still assume Hypothesis \eqref{eqn_hypo} here. Applying the sufficient condition for \eqref{eqn_hypo} in terms of divisors from Theorem \ref{thm_daili} \ref{item_daili1} (with $\ve\alpha_1=\frac{1}{\sqrt{2}}$), we obtain:
\begin{corollary}\label{coro_Avariation2}
	Let $\Sigma$ be a closed A-surface of genus $\geq2$ in an orientable pseudo-Riemannian manifold $M$ locally modeled on $\H^{p,q}$ with $n=\frac{p+q}{2}\geq3$. Suppose $\Sigma$ is not locally contained in any pseudo-hyperbolic subspace of codimension $4$ and the holomorphic forms $\ve\alpha_2,\cdots,\ve\alpha_{n-1},\ve\alpha_n^\pm$ in its structural data satisfy $(\ve\alpha_2)\prec\cdots\prec (\ve\alpha_{n-1})\prec \min\big\{(\ve\alpha_n^+),(\ve\alpha_n^-)\big\}$. Then $\Sigma$ does not admit any nonzero Jacobi field.
\end{corollary}

\subsection{Uniqueness of A-surface with given Gauss map}\label{subsec_uniqueness}
An A-surface $\Sigma\subset\H^{p,q}$ is not completely determined by its Gauss map $f:\Sigma\to X_{\SO_0(p,q+1)}$, as the antipodal surface $-\Sigma$ has the same Gauss map 
(more precisely, the latter is the composition of $f$ with the antipodal map $-\Sigma\to\Sigma$; this can be easily seen by using the setup in the proof of Theorem \ref{thm_gaussmap}). Again under Hypothesis \eqref{eqn_hypo}, we may show that this is the only ambiguity:

\begin{theorem}\label{thm_determination}
Let $\Sigma\subset\H^{p,q}$ be an A-surface, where $n=\frac{p+q}{2}\geq3$. Suppose $\Sigma$ is not contained in any pseudo-hyperbolic subspace of codimension $4$ and satisfies  \eqref{eqn_hypo}. Let $\iota':\Sigma\to\H^{p,q}$ be an A-surface immersion with the same Gauss map as $\Sigma$. Then $\iota'$ is the restriction of either the identity map or the antipodal map of $\H^{p,q}$ to $\Sigma$.
\end{theorem}

\begin{proof}
	Let $\iota,u_1,v_1,\cdots,u_n,v_n$ be locally defined maps from $\Sigma$ to $\R^{p,q+1}$ as in the proofs of Theorem \ref{thm_holo} and Theorem \ref{thm_gaussmap} (where $\iota$ is the inclusion), so that the matrix-valued function
	$$
	\wt{f}:=(u_1,\, v_1,\, u_3,\, v_3,\, \cdots,\ \iota,\, u_2,\, v_2,\, u_4,\, v_4)
	$$
	is a locally defined lift of the Gauss map $f$ of $\Sigma$ and we have
	\begin{equation}\label{eqn_proofdetermination1}
		\wt f^{-1}\dif\wt f=\mat{*&\transp\Gamma\\\Gamma&*},
	\end{equation}
	where $\Gamma$ has the expression \eqref{eqn_Gamma} in terms of the structural $1$-forms $\ve\alpha_2,\cdots,\ve\alpha_n$ of $\Sigma$.
	
	For any another A-surface immersion $\iota'$, we again have locally defined maps $u'_1,v'_1,\cdots,u'_n,v'_n$ from $\Sigma$ to $\R^{p,q+1}$ such that
	$$
	\wt f':=(u'_1,\,v'_1,\,u'_3,\,v'_3,\,\cdots,\ \iota',\,u'_2,\,v'_2,\,u'_4,\,v'_4)
	$$
	lifts the Gauss map $f'$ of $\iota'$ and that
	\begin{equation}\label{eqn_proofdetermination2}
		\wt f'^{-1}\dif\wt f'=\mat{*&\transp\Gamma'\\\Gamma'&*},
	\end{equation}
with $\Gamma'$ given by the same expression \eqref{eqn_Gamma} as $\Gamma$ except that each $\ve\alpha_j$ is replaced by the $1$-forms $\ve\alpha_j'$ coming from $\iota'$, which might be different from $\ve\alpha_j$ a priori.
	
	Now assume $f'=f$. Then there is a locally defined map $g=(g_1,g_2)$ from $\Sigma$ to $\SO(p)\times\SO(q+1)$ such that $\wt f'=\wt f g$, and it follows that $\wt f'^{-1}\dif\wt f'=g^{-1}(\wt f^{-1}\dif\wt f) g+g^{-1}\dif g$. Comparing with \eqref{eqn_proofdetermination1} and \eqref{eqn_proofdetermination2}, we infer 
	that $\Gamma'=g_2^{-1}\Gamma g_1$, and as a consequence, 
	\begin{equation}\label{eqn_proofdetermination3}
		\Tr_\ve{g}\big[\Gamma'(\sth)\,\transp\Gamma'(\sth)\big]=g_2^{-1}\Tr_\ve{g}\big[\Gamma(\sth)\,\transp\Gamma(\sth)\big]g_2,
	\end{equation}
	where we take the tensor traces of matrix-valued $2$-tensors with respect to the first fundamental form $\ve{g}$ (as in the proof of Theorem \ref{thm_Avariation}) to produce matrix-valued functions. These traces can be calculated similarly as the calculation of $\Tr_\ve{g}\langle \Gamma_0^\pm(\sth)\xi, \Gamma_0^\pm(\sth)\xi\rangle$ in the proof of Theorem \ref{thm_Avariation} (the $\Gamma_0^+$ therein is just $\Gamma$ with the first row and column removed). The result is
	\begin{equation}\label{eqn_proofdetermination4}
		\Tr_\ve{g}\big[\Gamma'(\sth)\,\transp\Gamma'(\sth)\big]=\mat{2&\\&\Lambda'},\quad \Tr_\ve{g}\big[\Gamma(\sth)\,\transp\Gamma(\sth)\big]=\mat{2&\\&\Lambda},
	\end{equation}
	where the matrix $\Lambda$ is diagonal except for a possible $2\times2$ block and its eigenvalues belong to the list
	\begin{align*}
		&2\|\ve\alpha_2\|^2+2\|\ve\alpha_3\|^2,\quad 2\|\ve\alpha_3\|^2+2\|\ve\alpha_4\|^2,\ \cdots\\
		&\quad\cdots,\ 2\|\ve\alpha_{n-2}\|^2+2\|\ve\alpha_{n-1}\|^2,\quad 2\|\ve\alpha_{n-1}\|^2+2\|\ve\alpha_n^+\|^2+2\|\ve\alpha_n^-\|^2,\quad 2(\|\ve\alpha_n^+\|\pm\|\ve\alpha_n^-\|)^2.
	\end{align*}
The matrix $\Lambda'$ is the same except that $\ve\alpha_2,\cdots,\ve\alpha_{n-1},\ve\alpha_n^\pm$ are replaced by $\ve\alpha_2',\cdots,\ve\alpha_{n-1}',\ve\alpha_n'{\!\!}^\pm$.
	
	We may now apply condition \eqref{eqn_hypo}, which implies inequalities \eqref{eqn_Avariation3} on an open dense set $U\subset\Sigma$. It follows that at every point in $U$, the sole largest eigenvalue of $\Tr_\ve{g}\big[\Gamma(\sth)\,\transp\Gamma(\sth)\big]$ is $2$. In view of \eqref{eqn_proofdetermination3} and \eqref{eqn_proofdetermination4}, this forces the $\SO(q+1)$-valued function $g_2$ to have the form 
	$$
	g_2=\mat{\pm1&\\&*}
	$$
on $U$. By continuity, it has the same form throughout $\Sigma$. Since $\wt f'=\wt f g$, we conclude that $\iota'=\pm\iota$.
\end{proof}

\section{Application to cyclic $\SO_0(n,n+1)$-Higgs bundles}\label{sec_cyclic}
In this section, we review some backgrounds on cyclic $\SO_0(n,n+1)$-Higgs bundle before showing in Theorem \ref{thm_higgsA} ($\approx$ Theorem \ref{thm_intro1}) that such Higgs bundles yield A-surfaces, then we use the infinitesimal rigidity result to prove Theorem \ref{thm_intro3} (restated as Theorem \ref{thm_main1}) and finally show in Proposition \ref{prop_g2A} that in the $G_2'$ case these A-surfaces belong to the type from Proposition \ref{prop_holomorphicA}.

\subsection{Cyclic $\SO_0(n,n+1)$-Higgs bundles and A-surfaces}\label{subsec_cyclic}
We refer to \cite{arroyo, collier, collier-tholozan-toulisse,2009arXiv0909.4487G} for the definition of $G$-Higgs bundles for a general real reductive group $G$ and the special case of $\SO_0(p,q)$-Higgs bundles. We only consider here the following particular type of $\SO_0(n,n+1)$-Higgs bundles:
\begin{definition}\label{def_cyclic}
Let $\Sigma$ be a closed Riemann surface. Given holomorphic line bundles $\TT_1,\cdots,\TT_n$ ($n\geq2$) on $\Sigma$ and holomorphic sections 
$\ve\alpha_1$, $\ve\alpha_j$ ($2\leq j\leq n-1$) and $\ve\alpha_n^\pm$ of the line bundles $\ca{K}\!\TT_1$, $\ca{K}\!\TT_{j-1}^{-1}\TT_j$ and $\ca{K}\!\TT_{n-1}^{-1}\TT_n^{\pm1}$, respectively, with $\ve\alpha_1,\cdots,\ve\alpha_{n-1},\ve\alpha_n^+\not\equiv0$, we definite the \emph{cyclic $\SO_0(n,n+1)$-Higgs bundle} associated to these data as the $\SO_0(n,n+1)$-Higgs bundle $(\ca{V},Q_\ca{V},\ca{W},Q_\ca{W},\eta)$ given by
\begin{align*}
\ca{V}&=\TT_{-n+1}\oplus \TT_{-n+3}\oplus\ \cdots\ \oplus\TT_{n-3}\oplus\TT_{n-1},\quad
Q_\ca{V}=
\scalebox{0.8}{$
\mat{&&1\\[-0.1cm]&\iddots&\\[-0.1cm]1&&}
$},\\
\ca{W}&=\TT_{-n}\oplus\TT_{-n+2}\oplus\ \cdots\ \oplus\TT_{n-2}\oplus\TT_n,\quad
 Q_\ca{W}=\scalebox{0.8}{$
\mat{&&1\\[-0.1cm]&\iddots&\\[-0.1cm]1&&}
$},\\
\eta&=
\scalebox{0.9}{$
\mat{*&&&*\\[-0.1cm]&\ddots&&\\[-0.1cm]&&*&}
$}\in H^0(\Sigma,\ca{K}\otimes\Hom(\ca{W},\ca{V})),
\end{align*}
where we set $\TT_0:=\ca{O}$, $\TT_{m}:=\TT_{-m}^{-1}$ for $m<0$, and let each ``$*$'' in the matrix expression of $\eta$ be one of the above holomorphic sections, namely the unique one which fits in that slot (for example, the first row is $(\ve\alpha_n^+,0,\cdots,0,\ve\alpha_n^-)$, the second is $(0,\ve\alpha_{n-2},0,\cdots,0)$, etc.).
\end{definition}

Alternatively, we may define the above $(\ca{V},Q_\ca{V},\ca{W},Q_\ca{W},\eta)$ as the $\SO_0(n,n+1)$-Higgs bundle whose underlying $\GL(2n+1,\C)$-Higgs bundle is 
\begin{align*} 
	&\ca{E}=\ca{V}\oplus\ca{W}=\TT_n^{-1}\oplus\cdots \oplus \TT_1^{-1}\oplus \ca{O}\oplus \TT_1\oplus  \cdots \oplus \TT_n,\\
	&\hspace{0.28cm}\Phi=
	\scalebox{0.85}{$
		\left(
		\begin{array}{ccccc|c|ccccc}
			0&&&&&&&&&\ve\alpha_n^-&\\
			\ve\alpha_n^+&0&&&&&&&&&\ve\alpha_n^-\\
			&\hspace{-0.1cm}\ve\alpha_{n-1}&\ddots&&&&&&&&\\
			&&\hspace{-0.2cm}\ddots&\ddots&&&&&&&\\
			&&&\hspace{-0.1cm}\ve\alpha_2&0&&&&&&\\
			\midrule
			&&&&\ve\alpha_1&0&&&&&\\
			\midrule
			&&&&&\ve\alpha_1&0&&&&\\
			&&&&&&\ve\alpha_2&\ddots&&&\\
			&&&&&&&\hspace{-0.2cm}\ddots&\ddots&&\\
			&&&&&&&&\hspace{-0.1cm}\ve\alpha_{n-1}&0&\\
			&&&&&&&&&\ve\alpha_n^+&0
		\end{array}
		\right)
		$}.
\end{align*}
We emphasize that each of the sub-diagonal holomorphic forms $\ve\alpha_1,\cdots,\ve\alpha_{n-1},\ve\alpha_n^+$ is assumed not to be identically zero, although $\ve\alpha_n^-$ can be zero.
Also note that the trivial line bundle $\ca{O}$ is a summand in $\ca{V}$ (resp.\ $\ca{W}$) when $n$ is odd (resp.\ even) and we can choose a global holomorphic section $\underline{1}$ of $\ca{O}$ with $Q_\ca{V}(\underline{1}\,,\,\underline{1})\equiv1$ (resp.\ $Q_\ca{W}(\underline{1}\,,\,\underline{1})\equiv1$). This section is unique up to sign. We let $\underline{\R}:=\R\underline{1}$ be the real line subbundle of $\ca{O}$ generated by it. 

In general, a polystable $G$-Higgs bundle admits a canonical $K$-reduction ($K\subset G$ denotes the maximal compact subgroup), namely the one solving the Hitchin equation, as well as a canonical flat $G$-connection $D$, namely the Hitchin connection. Applying the parallel transportation of $D$ to the $K$-reduction, we get a harmonic map $f$ from $\wt{\Sigma}$ to the symmetric space $X_G=G/K$ which is equivariant with respect to the holonomy representation of $D$. For the above particular $\SO_0(n,n+1)$-Higgs bundles, $f$ is a conformal minimal immersion and these objects have a more concrete description:
\begin{propdef}\label{propdef_harmonic}
The above $(\ca{E},\Phi)$ is stable if and only if either $\ve\alpha_n^-\not\equiv0$ or $\deg(\TT_n)<0$. In this case, there is a unique hermitian metric $\ve{h}_i$ on each $\TT_i$ such that the hermitian metric
$$
H:=\ve{h}_n^{-1}\oplus\cdots\oplus\ve{h}_1^{-1}\oplus1\oplus\ve{h}_1\oplus\cdots\oplus\ve{h}_n
$$ 
 on $\ca{E}$ (where the middle $1$ denotes the hermitian metric on $\ca{O}$ under which the section $\underline{1}$ has unit norm) satisfies the \emph{Hitchin equation}
\begin{equation}\label{eqn_hitchin}
F_{\nabla^H}+[\Phi,\Phi^*]=0.
\end{equation}
Here, $\nabla^H$ is the Chern connection of $(\ca{E},H)$ and $\Phi^*$ is the $\End(\ca{E})$-valued $(0,1)$-form $H$-adjoint to $\Phi$.

Eq.\eqref{eqn_hitchin} coincides with equations \eqref{eqn_dailin2} or \eqref{eqn_daili} in Theorem \ref{thm_daili}
and is equivalent to the condition that the real connection
$$
D^H:=\nabla^H+\Phi+\Phi^*
$$
on $\ca{E}$ is flat. The metric $H$ and the connection $D^H$ are called the \emph{harmonic metric} and \emph{Hitchin connection} of $(\ca{E},\Phi)$, respectively. We also refer to $\ve{h}_i$ as the harmonic metric on $\TT_i$. Further let
\begin{itemize}
\item $Q_\ca{E}$ be the holomorphic quadratic form $Q_\ca{V}\oplus (-Q_\ca{W})$ on $\ca{E}$;
\item 
 $\lambda_i$ be the anti-linear involution of $\TT_i^{-1}\oplus\TT_i$ given by
$$
\lambda_i(\ell_i^{-1})=\tfrac{1}{h_i}\ell_i,\quad \lambda_i(\ell_i)=h_i\ell_i^{-1}
$$
(where $\ell_i$ is a holomorphic local frame of $\TT_i$, $\ell_i^{-1}$ is the corresponding frame of $\TT_i^{-1}$, and $h_i:=\ve{h}_i(\ell_i,\ell_i)$; the definition of $\lambda_i$ is independence of the choice of $\ell_i$);
\item 
$\lambda$ be the anti-linear involution of $\ca{E}$ which is the direct sum of $\lambda_1,\cdots,\lambda_n$ and the anti-linear involution of $\ca{O}$ fixing $\underline{\R}\,$;
\item $\ca{E}_\R$, $\ca{V}_\R$, $\ca{W}_\R$ be the real subbundles formed by the fixed points of $\lambda$ in $\ca{E}$, $\ca{V}$, $\ca{W}$, respectively. 
\end{itemize}
Then the connection $D^H$ preserves both $Q_\ca{E}$ and $\ca{E}_\R$, whereas $Q_\ca{E}$ restricts to a real quadratic form of signature $(n,n+1)$ on $\ca{E}_\R$, denoted by $\langle\sth,\sth\rangle$, under which $\ca{V}_\R$ (resp.\ $\ca{W}_\R$) is a positive (resp.\ negative) definite subbundle of rank $n$ (resp.\ $n+1$).
 Moreover, if we choose a base point $z_0\in\Sigma$, let 
 $$
 \rho:\pi_1(\Sigma,z_0)\to \SO_0(\ca{E}_\R|_{z_0},\langle\sth,\sth\rangle)\cong\SO_0(n,n+1)
 $$ 
 be the holonomy representation of $D^H$ and 
\begin{align*}
f:\wt{\Sigma}\to X_{\SO_0(n,n+1)}&\cong\big\{\text{$n$-dimensional positive definite subspaces of $\ca{E}_\R|_{z_0}$}\big\}\\
&\cong\big\{\text{$(n+1)$-dimensional negative definite subspaces of $\ca{E}_\R|_{z_0}$}\big\}
\end{align*}
be the $\rho$-equivariant map given by translating the fiber of $\ca{V}_\R$ or $\ca{W}_\R$ at each point into $\ca{E}_\R|_{z_0}$ via the parallel transportation of $D^H$, then $f$ is a conformal minimal immersion.
\end{propdef}
\begin{proof}[Outline of proof]
 If we assume $(\ca{E},\Phi)$ to be stable, then every assertion can be either verified directly or deduced from the general theory of Higgs bundles \cite{corlette,donaldson,hitchin_selfduality,simpson} and cyclic Higgs bundles \cite{baraglia_dedicata,collier-li}.
For the stability statement at the beginning, the ``if'' part can be shown by the same argument as Lemma 4.5 and Lemma 5.11 in \cite{collier}; the ``only if'' part is because polystablity implies solvability of the Hitchin equation, but by the form \eqref{eqn_dailin2} or \eqref{eqn_daili} of the Hitchin equation, it is clear that if $\ve\alpha_n^-\equiv0$ then the Chern form $\frac{\ima}{2\pi}\bpa\pa\log h_n=-\tfrac{1}{\pi}\pazbz\log h_n\dx\wedge\dy$ of $\TT_n$  is negative away from the zeros of $\ve\alpha_n^+$, hence $\deg(\TT_n)=\int_\Sigma\frac{\ima}{2\pi}\bpa\pa\log h_n<0$ (this also shows that if $(\ca{E},\Phi)$ is polystable then it is stable).
\end{proof}

When $\TT_1=\ca{K}^{-1}$ and $\ve\alpha_1=\tfrac{1}{\sqrt{2}}\in H^0(\Sigma,\C)=H^0(\Sigma,\ca{K}\!\TT_1)$, the Hitchin equation \eqref{eqn_dailin2} \eqref{eqn_daili} are exactly the equations satisfied by the structural data of A-surfaces (see Theorem \ref{thm_affinetoda}). In this case, we show that $(\ca{E},\Phi)$ yields an A-surfaces in $\H^{p,q}$ (with $(p,q)$ given by $n$ as before; see \eqref{eqn_pqintro}) in a similar way as how maximal $\SO_0(2,q+1)$-Higgs bundles yield maximal surfaces in $\H^{2,q}$ (see \cite{collier-tholozan-toulisse}):

\begin{theorem}\label{thm_higgsA}
In the setting of Proposition/Definition \ref{propdef_harmonic}, further assume 
\begin{itemize}
	\item $\TT_1=\ca{K}^{-1}$ and $\ve\alpha_1=\tfrac{1}{\sqrt{2}}$;
	\item $(\ve\alpha_2)\prec\cdots\prec(\ve\alpha_{n-1})\prec\min\big\{(\ve\alpha_n^+),(\ve\alpha_n^-)\big\}$ (see \S \ref{subsec_hypo}).
\end{itemize}
Then there exists a unique antipodal pair $\pm\hf$ of $\rho$-equivariant A-surface immersions of $\wt{\Sigma}$ into $\H^{p,q}$ whose Gauss map is $f$. Moreover, after possibly swapping $\pm\hf$, the structural data of $\hf$ is exactly the holomorphic data $(\TT_1,\cdots,\TT_n,\ve\alpha_2,\cdots,\ve\alpha_{n-1},\ve\alpha_n^\pm)$ used to construct $(\ca{E},\Phi)$ along with the harmonic metric $\ve{h}_i$ on each $\TT_i$.
\end{theorem}
\begin{proof}
In this proof, when $n$ is odd, we multiply $Q_\ca{E}$ by $-1$ and swap the notations of $\ca{V}$ and $\ca{W}$. Under this new notation, $\ca{W}$ is the direct sum of  all $\TT_k$ with $k\in[-n,n]$ even, including $\TT_0=\ca{O}$. The signature of $\langle\sth,\sth\rangle$ is exactly $(p,q+1)$, while $\ca{V}_\R$ (resp.\ $\ca{W}_\R$) is a positive  (resp.\ negative) definite subbundle of rank $p$ (resp.\ $q+1$) in $\ca{E}_\R$.

Identifying $\H^{p,q}$ as the hypersurface in the fiber $\ca{E}_\R|_{z_0}$ formed by all $x\in\ca{E}_\R|_{z_0}$ with $\langle x,x\rangle=-1$, we let $\hf$ be the $\rho$-equivariant map from $\wt{\Sigma}$ to $\H^{p,q+1}$ given by translating the section $\underline{1}$ of $\underline{\R}$ at each point into $\ca{E}_\R|_{z_0}$ (similarly as the construction of $f$). 
We shall show that $\hf$ is an A-surface immersion with Gauss map $f$ and has the required structural data. Combining with the uniqueness result in Theorem \ref{thm_determination}, this implies the theorem.

Let $\TT'_i$ denote the real subbundle of $\TT^{-1}_i\oplus\TT_i$ fixed by $\lambda_i$. Then $\TT'_i$ is positive (resp.\ negative) definite with respect to $\langle\sth,\sth\rangle$ exactly when $i$ is odd (resp.\ even), and we have orthogonal splittings
$$
\ca{E}_\R=\underline{\R}\oplus \TT_1'\oplus\TT_2'\oplus\cdots\oplus\TT_n',\quad \ca{W}_\R=\underline{\R}\oplus\TT_2'\oplus\TT_4'\oplus\cdots.
$$
Choose a holomorphic local frame $\ell_i$ of $\TT_i$ and consider the following orthonormal local frame $(u_i,v_i)$ of $\TT_i'$ with respect to the positive definite metric $\langle\sth,\sth\rangle_i:=(-1)^{i+1}\langle\sth,\sth\rangle|_{\TT_i'}$:
$$
u_i:=\tfrac{1}{\sqrt{2}}\big(h_i^\frac{1}{2}\ell_i^{-1}+h_i^{-\frac{1}{2}}\ell_i\big),\quad 
v_i:=\tfrac{1}{\sqrt{2}}\big(-\ima h_i^\frac{1}{2}\ell_i^{-1}+\ima h_i^{-\frac{1}{2}}\ell_i\big).
$$
We may identify $\TT_i'$ with $\TT_i$ through the $\sqrt{2}$-scaled projection from $\TT_i^{-1}\oplus\TT_i$ to $\TT_i$, and view $\TT_i'$ as a hermitian holomorphic line bundle in this way.
Note that this identifies $\langle\sth,\sth\rangle_i$ with the real part of the hermitian metric $\ve{h}_i$ and identifies $(u_i,v_i)$ with the frame $h_i^{-\frac{1}{2}}(\ell_i,\ima \ell_i)$ of $\TT_i$.
In particular, $\TT_1'$ gets identified with $\TT_1=\ca{K}^{-1}$, which can be viewed as the tangent bundle $\T\Sigma$ (see the paragraph preceding Theorem \ref{thm_holo}), while the metric $\langle\sth,\sth\rangle_1\cong\re\ve{h}_1$ corresponds to the conformal Riemannian metric $\ve{g}=h_1|\dz|^2$ on $\Sigma$.

To prove the required statement about $\hf$, it suffices to show that the Hitchin connection $D^H$ (restricted to $\ca{E}_\R$) is expressed under the local frame $(\underline{1},u_1,v_1,\cdots,u_n,v_n)$ of $\ca{E}_\R$ as
\begin{equation}\label{eqn_proofhiggsA1}
D^H=\dif+\Omega,\quad \text{ where }\Omega=
\scalebox{0.9}{
$\mat{
	&\transp\theta&&&\\[0.1cm]
	\theta&\Omega_1&\transp\ve\alpha_2&&\\[0.1cm]
	&\ve\alpha_2&\Omega_2&\ddots&\\
	&&\hspace{-0.2cm}\ddots&\hspace{-0.2cm}\ddots&\transp\ve\alpha_n\\[0.1cm]
	&&&\hspace{-0.3cm}\ve\alpha_n&\hspace{-0.1cm}\Omega_n
}$}.
\end{equation}
Here, $\theta=\transp(\theta_1,\theta_2)$ is an orthonormal local frame of $\T^*\Sigma$ with respect to $\ve{g}$, $\Omega_i$ represents the Chern connection of $\TT_i'\cong\TT_i$, and we abuse the notation to let $\ve\alpha_2,\cdots,\ve\alpha_{n-1}$ and $\ve\alpha_n=\ve\alpha_n^++\ve\alpha_n^-$ be the matrices of $1$-forms given in the proof of Theorem \ref{thm_affinetoda}, which represent the holomorphic forms $\ve\alpha_2,\cdots,\ve\alpha_{n-1},\ve\alpha_n^\pm$ used to construct $(\ca{E},\Phi$).
In fact, once \eqref{eqn_proofhiggsA1} is established, by the parallel transportation process in the construction of $f$ and $\hf$, we may infer that there are locally defined maps $U_1,V_1,\cdots,U_n,V_n$ from $\wt\Sigma$ to $\R^{p,q+1}$ (given by parallel transporting $u_1,v_1,\cdots,u_n,v_n$) such that the matrix valued function $P:=(\hf,U_1,V_1,\cdots,U_n,V_n)$
satisfies $P^{-1}\dif P=\Omega$, whereas the function $(U_1,V_1,U_3,V_3,\cdots,\hf,U_2,V_2,U_4,V_4)$, which takes values in $\SO_0(p,q+1)$, is a lift of $f$. But this implies that $\hf$ has the required structural data and $f$ is the Gauss map of $\hf$ (c.f.\ the proof of Theorems \ref{thm_holo} and \ref{thm_gaussmap}).

To show \eqref{eqn_proofhiggsA1}, we start from the expression of $D^H$ under the frame $(\ell_n^{-1},\cdots,\ell_1^{-1},\underline{1},\ell_1,\cdots,\ell_n)$ of $\ca{E}$:
\begin{align}
&D^H=\nabla^H+\Phi+\Phi^*=\dif+\Omega', \text{ where}\nonumber\\
&\Omega'=\left(
\scalebox{0.85}{
$\begin{array}{ccccc|c|ccccc}
-\tfrac{\pa h_n}{h_n}&\tfrac{h_n\cj{\alpha_n^+}}{h_{n-1}}\dbz&&&&&&&&\alpha_n^-\dz&\\[0.3cm]
	\alpha_n^+\dz&-\tfrac{\pa h_{n-1}}{h_{n-1}}&\tfrac{h_{n-1}\cj{\alpha_{n-1}}}{h_{n-2}}\dbz&&&&&&&&\alpha_n^-\dz\\[0.3cm]
	&\alpha_{n-1}\dz&\ddots&\ddots&&&&&&&\\[0.3cm]
	&&\ddots&\ddots&\tfrac{h_2\cj{\alpha_2}}{h_1}\dbz&&&&&&\\[0.3cm]
	&&&\alpha_2\dz&-\tfrac{\pa h_1}{h_1}&\tfrac{h_1\dbz}{\sqrt{2}}&&&&&\\
	\midrule
	&&&&\tfrac{\dz}{\sqrt{2}}&0&\tfrac{h_1\dbz}{\sqrt{2}}&&&&\\
	\midrule
	&&&&&\tfrac{\dz}{\sqrt{2}}&\tfrac{\pa h_1}{h_1}&\tfrac{h_2\cj{\alpha_2}}{h_1}\dbz&&&\\[0.3cm]
	&&&&&&\alpha_2\dz&\ddots&\ddots&&\\[0.3cm]
	&&&&&&&\ddots&\ddots&\tfrac{h_{n-1}\cj{\alpha_{n-1}}}{h_{n-2}}\dbz&\\[0.3cm]
	\tfrac{\cj{\alpha_n^-}\,\dbz}{h_{n-1}h_n}&&&&&&&&\alpha_{n-1}\dz&\tfrac{\pa h_{n-1}}{h_{n-1}}&\tfrac{h_n\cj{\alpha_n^+}}{h_{n-1}}\dbz\\[0.3cm]
	&\tfrac{\cj{\alpha_n^-}\,\dbz}{h_{n-1}h_n}&&&&&&&&\alpha_n^+\dz&\tfrac{\pa h_n}{h_n}
\end{array}$
}
\right).\label{eqn_proofhiggsA2}
\end{align}
The frame for \eqref{eqn_proofhiggsA1} is  $(\underline{1},u_1,v_1,\cdots,u_n,v_n)=(\ell_n^{-1},\cdots,\ell_1^{-1},\underline{1},\ell_1,\cdots,\ell_n)R$ with change-of-basis matrix
$$
R=
\tfrac{1}{\sqrt{2}}
\left(
\scalebox{0.87}{
	$\begin{array}{c|cc|cc|cc|cc}
		&&&&&&&h_n^\frac{1}{2}&-\ima h_n^\frac{1}{2}\\
		&&&&&&\iddots&&\\[-0.1cm]
		&&&&&\iddots&&&\\
		&&&h_2^\frac{1}{2}&-\ima h_2^\frac{1}{2}&&&&\\
		&h_1^\frac{1}{2}&-\ima h_1^\frac{1}{2}&&&&&&\\
		\midrule
		\hspace{-0.2cm}\sqrt{2}&&&&&&&&\\
		\midrule
		&h_1^{-\frac{1}{2}}&\ima h_1^{-\frac{1}{2}}&&&&&&\\
		&&&h_2^{-\frac{1}{2}}&\ima h_2^{-\frac{1}{2}}&&&&\\
		&&&&&\ddots&&&\\[-0.1cm]
		&&&&&&\ddots&&\\
		&&&&&&&h_n^{-\frac{1}{2}}&\ima h_n^{-\frac{1}{2}}
	\end{array}$
}
\right).
$$
So the matrix of $D^H$ under this frame is $R^{-1}\Omega' R+R^{-1}\dif R$, and we only need to show that it equals $\Omega$ in \eqref{eqn_proofhiggsA1}. 
This is a straightforward but lengthy matrix computation. We give below the details of some key components of the computation:
\begin{itemize}
	\item The $\TT_i'$-diagonal block of $R^{-1}\Omega' R+R^{-1}\dif R$ is
	$$
	\scalebox{0.85}{
		$\mat{h_i^\frac{1}{2}&-\ima h_i^\frac{1}{2}\\[0.2cm]h_i^{-\frac{1}{2}}&\ima h_i^{-\frac{1}{2}}}^{-1}$
	}
	\mat{-\tfrac{\pa h_i}{h_i}&\\&\tfrac{\pa h_i}{h_i}}
	\scalebox{0.85}{
		$\mat{h_i^\frac{1}{2}&-\ima h_i^\frac{1}{2}\\[0.2cm]h_i^{-\frac{1}{2}}&\ima h_i^{-\frac{1}{2}}}$
	}
	+
	\scalebox{0.85}{
		$\mat{h_i^\frac{1}{2}&-\ima h_i^\frac{1}{2}\\[0.2cm]h_i^{-\frac{1}{2}}&\ima h_i^{-\frac{1}{2}}}^{-1}$
	}
	\hspace{-0.3cm}
	\scalebox{0.85}{
		$\dif\mat{h_i^\frac{1}{2}&-\ima h_i^\frac{1}{2}\\[0.2cm]h_i^{-\frac{1}{2}}&\ima h_i^{-\frac{1}{2}}}$
	}
	=\mat{&\tfrac{\dc h_i}{2h_i}\\-\tfrac{\dc h_i}{2h_i}&},
	$$
	where $\dc h:=\ima(\pa-\bpa)h=\pa_yh\dx-\pa_xh\dy$. The result is exactly the matrix expression of the Chern connection of $(\TT_i,\ve{h}_i)$ under the frame $h_i^{-\frac{1}{2}}(\ell_i,\ima\ell_i)$, as required.
	
	\item The $\underline{\R}$-to-$\TT_1'$ block of $R^{-1}\Omega' R+R^{-1}\dif R$  is 
	$$
	\scalebox{0.85}{
		$\mat{h_1^\frac{1}{2}&-\ima h_1^\frac{1}{2}\\[0.2cm]h_1^{-\frac{1}{2}}&\ima h_1^{-\frac{1}{2}}}^{-1}$
	}
	\mat{\tfrac{h_1\dbz}{\sqrt{2}}\\[0.2cm]\frac{\dz}{\sqrt{2}}}\cdot \sqrt{2}=h_1^\frac{1}{2}\mat{\dx\\\dy}
	$$
	(the term $R^{-1}\dif R$ is block-diagonal and does not contribute here; similarly below). This is indeed an orthonormal local frame of $\T^*\Sigma$ under $\ve{g}=h_1(\dx^2+\dy^2)$, as required.
	\item The $\TT_{j-1}'$-to-$\TT_j'$ block ($2\leq j\leq n-1$) and the $\TT_{n-1}'$-to-$\TT_n'$ block are respectively
	\begin{align*}
		&\scalebox{0.85}{
			$\mat{h_j^\frac{1}{2}&-\ima h_j^\frac{1}{2}\\[0.2cm]h_j^{-\frac{1}{2}}&\ima h_j^{-\frac{1}{2}}}^{-1}$
		}
		\mat{\tfrac{h_j\cj{\alpha_j}}{h_{j-1}}\dbz&\\&\alpha_j\dz}
		\scalebox{0.85}{
			$\mat{h_{j-1}^\frac{1}{2}&-\ima h_{j-1}^\frac{1}{2}\\[0.2cm]h_{j-1}^{-\frac{1}{2}}&\ima h_{j-1}^{-\frac{1}{2}}}$
		}=\tfrac{\sqrt{h_j}}{\sqrt{h_{j-1}}}
		\scalebox{0.9}{$
			\mat{\re(\alpha_j\dz)&-\im(\alpha_j\dz)\\[0.1cm]\im(\alpha_j\dz)&\re(\alpha_j\dz)}
			$},\\
		&\scalebox{0.85}{
			$\mat{h_n^\frac{1}{2}&-\ima h_n^\frac{1}{2}\\[0.2cm]h_n^{-\frac{1}{2}}&\ima h_n^{-\frac{1}{2}}}^{-1}$
		}
		\scalebox{0.85}{
			$\mat{\tfrac{h_n\cj{\alpha_n^+}}{h_{n-1}}\dbz&\alpha_n^-\dz\\[0.2cm]\tfrac{\cj{\alpha_n^-}\,\dbz}{h_{n-1}h_n}&\alpha_n^+\dz}$
		}
		\scalebox{0.85}{
			$\mat{h_{n-1}^\frac{1}{2}&-\ima h_{n-1}^\frac{1}{2}\\[0.2cm]h_{n-1}^{-\frac{1}{2}}&\ima h_{n-1}^{-\frac{1}{2}}}$
		}\\
		&\hspace{0.5cm}=\tfrac{\sqrt{h_n}}{\sqrt{h_{n-1}}}
		\scalebox{0.9}{$
			\mat{\re(\alpha_n^+\dz)&-\im(\alpha_n^+\dz)\\[0.1cm]\im(\alpha_n^+\dz)&\re(\alpha_n^+\dz)}
			$}
		+\tfrac{1}{\sqrt{h_{n-1}h_n}}
		\scalebox{0.9}{$
			\mat{\re(\alpha_n^-\dz)&-\im(\alpha_n^-\dz)\\[0.1cm]-\im(\alpha_n^-\dz)&-\re(\alpha_n^-\dz)}
			$}.
	\end{align*}
	These are exactly the matrices representing $\ve\alpha_j$ and $\ve\alpha_n=\ve\alpha_n^++\ve\alpha_n^-$ under the frame $h_i^{-\frac{1}{2}}(\ell_i,\ima\ell_i)$ (see the proof of Theorem \ref{thm_affinetoda}), as required.
\end{itemize}
In summary, we obtain the expression \eqref{eqn_proofhiggsA1} of $D^H$ by computations and this completes the proof.
\end{proof}

\subsection{Collier's components}\label{subsec_collier}
We briefly review here the construction of Collier's components \cite{collier}, which leads to the statement of Theorem \ref{thm_intro3}, and then give the proof.

We henceforth fix integers $g\geq2$ and $n\geq3$. 
Given $k\geq1$, the Hodge bundle $\ca{H}_k$ is the holomorphic vector bundle over the Teichm\"uller space ${\ca{T}}_g$ whose fiber $\ca{H}_k|_{[\Sigma]}$ at any $[\Sigma]\in{\ca{T}}_g$ is the vector space $H^0(\Sigma,\ca{K}^k)$ of holomorphic $k$-differentials on the Riemann surface $\Sigma$ (marked by $S_g$). Given $0<d\leq (2g-2)n$, we also define a fiber bundle $\ca{F}_d$ over ${\ca{T}}_g$ in such a way that the fiber $\ca{F}_d|_{[\Sigma]}$ is the moduli space $\wt{\ca{F}}_d|_{[\Sigma]}\big/\sim$, where
\begin{itemize}
	\item $\wt{\ca{F}}_d|_{[\Sigma]}$ is the space of triples $(\ca{M},\ve\mu,\ve\nu)$ consisting of a holomorphic line bundle $\ca{M}$ of degree $d$ on $\Sigma$ and holomorphic sections $\ve\mu\in H^0(\Sigma,\ca{K}^n\!\!\!\ca{M}^{-1})\setminus\{0\}$, $\ve\nu\in H^0(\Sigma,\ca{K}^n\!\!\!\ca{M})$;
	\item the equivalence relation $(\ca{M},\ve\mu,\ve\nu)\sim(\ca{M}',\ve\mu',\ve\nu')$ means that there are line bundle isomorphisms $\phi_0:\ca{K}^{-n+1}\to\ca{K}^{-n+1}$ and $\phi_1:\ca{M}\to\ca{M}'$ which, along with the isomorphism $\phi_{-1}:\ca{M}^{-1}\to\ca{M}'{}^{-1}$ induced by $\phi_1$, fit into a commutative diagram
 \begin{center}
	\includegraphics[width=4.8cm]{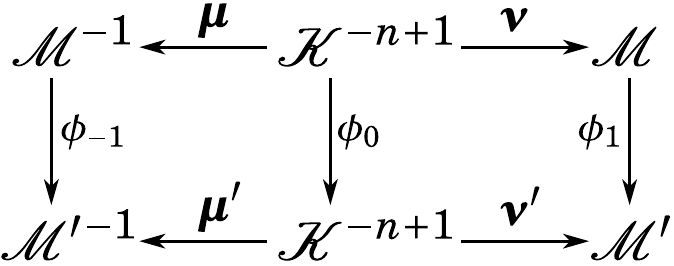}
 \end{center}
   Here, an horizontal arrow such as $\TT_1\overset{\ve\alpha}\to\TT_2$ means that $\ve\alpha$ is a $\Hom(\TT_1,\TT_2)$-valued holomorphic $1$-form, or equivalently a holomorphic section of $\ca{K}\!\TT_1^{-1}\!\TT_2$.
\end{itemize}

\begin{remark}
The $\wt{\ca{F}}_d|_{[\Sigma]}$ here is slightly different from the one in \cite{collier}\footnote{Note that everything in \cite{collier} is based on a fix Riemann surface $\Sigma$, so the notation ``$\,\sth\,|_{[\Sigma]}$'' is not included therein.} in that our $\ca{M}$ is a concrete line bundle rather than a point in $\mathrm{Pic}^d(\Sigma)$, but the resulting quotient $\ca{F}_d|_{[\Sigma]}$ is the same. As noticed in \cite{collier}, $\ca{F}_d|_{[\Sigma]}$ identifies with the total space of a rank $d+(2n-1)(g-1)$ holomorphic vector bundle over the symmetric product $\Sym^{(2g-2)n-d}\Sigma$.
\end{remark}

The direct sum fiber bundle $\ca{H}_2\oplus\ca{H}_4\oplus\cdots\oplus\ca{H}_{2n-2}\oplus\ca{F}_d$ over ${\ca{T}}_g$ can be viewed as a data space whose each point determines an equivalence class of stable $\SO_0(n,n+1)$-Higgs bundles.
More specifically, 
from $\big(\ve\phi_2,\cdots,\ve\phi_{2n-2},(\ca{M},\ve\mu,\ve\nu)\big)\in\big(\ca{H}_2\oplus\ca{H}_4\oplus\cdots\oplus\ca{H}_{2n-2}\oplus\wt{\ca{F}}_d\big)|_{[\Sigma]}$, we get the following $\SO_0(n,n+1)$-Higgs bundle $(\ca{V},Q_\ca{V},\ca{W},Q_\ca{W},\eta)$ on $\Sigma$:
\begin{align*}
	&\ca{V}=\ca{K}^{n-1}\oplus\ca{K}^{n-3}\oplus\cdots\oplus\ca{K}^{-n+3}\oplus\ca{K}^{-n+1},\quad 
	Q_\ca{V}=
	\scalebox{0.9}{$
		\mat{&&1\\[-0.1cm]&\iddots&\\[-0.1cm]1&&}
		$},\\
	&\ca{W}=\ca{M}\oplus\ca{K}^{n-2}\oplus\ca{K}^{n-4}\oplus\cdots\oplus\ca{K}^{-n+4}\oplus\ca{K}^{-n+2}\oplus\ca{M}^{-1},\quad
	 Q_\ca{W}=\scalebox{0.8}{$
		\mat{&&1\\[-0.1cm]&\iddots&\\[-0.1cm]1&&}
		$},\\
	&\eta=
	\scalebox{0.85}{$
		\mat{\ve\mu&\ve\phi_2&\ve\phi_4&\cdots&\hspace{-0.15cm}\ve\phi_{2n-2}&\ve\nu\\&1&\ddots&\ddots&\vdots&\\[-0.1cm]&&\ddots&\ddots&\ve\phi_4&\\&&&1&\ve\phi_2&\\&&&&\tfrac{1}{\sqrt{2}}&}$} \in H^0(\Sigma,\ca{K}\otimes\Hom(\ca{W},\ca{V}));
\end{align*}
while replacing $(\ca{M},\ve\mu,\ve\nu)$ by an equivalent triple  $(\ca{M}',\ve\mu',\ve\nu')$ yields an equivalent Higgs bundle (the $\frac{1}{\sqrt{2}}$ here can be replaced by $1$ as mentioned in the introduction).

Let $\ca{X}:=\Hom(\pi_1(S_g),\SO_0(n,n+1))/\SO_0(n,n+1)$ be the $\SO_0(n,n+1)$-character variety. Combining the above construction with the non-abelian Hodge correspondence, 
which assigns to every equivalence class of stable $\SO_0(n,n+1)$-Higgs bundles a point of $\ca{X}$, 
we obtain a map
$$
\Psi:\ca{H}_2\oplus\ca{H}_4\oplus\cdots\oplus\ca{H}_{2n-2}\oplus\ca{F}_d\to \ca{X}.
$$
For any fixed Riemann surface $\Sigma$, Collier \cite{collier} studied the restriction of $\Psi$ to the fiber of the domain bundle at $[\Sigma]$ and found that it parametrizes certain remarkable components of $\ca{X}$:
\begin{theorem}[\cite{collier}]\label{thm_collier}
Given $d\in(0,(2g-2)n]$ and a Riemann surface $\Sigma$, the map $\Psi$ restricts to a diffeomorphism from $\big(\ca{H}_2\oplus\ca{H}_4\oplus\cdots\oplus\ca{H}_{2n-2}\oplus\ca{F}_d\big)|_{[\Sigma]}$ to a connected component $\ca{X}_d$ of $\ca{X}$.
\end{theorem}

On the other hand, the motivation behind Labourie's conjecture is to give certain components of character varieties \emph{mapping-class-group-equivariant} parametrizations. For the component $\ca{X}_d$ from Theorem \ref{thm_collier}, the major candidate for such a parametrization is the restriction of $\Psi$ to the domain fiber bundle with the summand $\ca{H}_2$ removed. However, although the existing techniques \cite{labourie_cross,breyrer-pozzetti,guichard-labourie-wienhard} imply that this restriction is surjective, the injectivity is a difficult problem, especially in view of the new twist \cite{markovic,markovic-sagman-smillie} in the study of Labourie's conjecture for hermitian-type Lie groups. Our goal here is to provide the following partial result which extends Labourie's work \cite{labourie_cyclic} for Hitchin components. The proof is closely modeled on the argument in \cite[\S 7.3.4]{labourie_cyclic}, only with the infinitesimal rigidity of cyclic surfaces replaced by that of A-surfaces. 
\begin{theorem}\label{thm_main1}
Given $d\in(0,(2g-2)n]$, $\Psi$ restricts to an immersion from the fiber bundle $\ca{F}_d$ to $\ca{X}_d$.
\end{theorem}
\begin{proof}
Fix $q=([\Sigma],\ca{M},\ve\mu,\ve\nu)\in \wt{\ca{F}}_d$ and consider a smooth variation $q_t=([\Sigma_t],\ca{M}_t,\ve\mu_t,\ve\nu_t)$ with $q_0=q$, $t\in(-\epsilon,\epsilon)$. 
Let $(\ca{E}_t,\Phi_t)$ be the Higgs bundle constructed from $q_t$ as above, with  $\ve\phi_2,\cdots,\ve\phi_{2n-2}=0$. Then $(\ca{E}_t,\Phi_t)$ is the special case of the Higgs bundles studied in \S \ref{subsec_cyclic} with $\TT_i=\ca{K}^{-i}$ for $i\leq n-1$, $\TT_n=\ca{M}^{-1}$, $\ve\alpha_1=\frac{1}{\sqrt{2}}$, $\ve\alpha_2=\cdots=\ve\alpha_{n-1}=1$, $\ve\alpha_n^+=\ve\mu_t$ and $\ve\alpha_n^-=\ve\nu_t$. By Prop./Def.\ \ref{propdef_harmonic} and Theorem \ref{thm_higgsA}, we get from $(\ca{E}_t,\Phi_t)$ a holonomy representation $\rho_t:\pi_1(S_g)=\pi_1(\Sigma_t)\to\SO_0(n,n+1)$ and a $\rho_t$-equivariant A-surface immersion $\hf_t:\wt\Sigma_t\to\H^{p,q}$, while the map $\Psi$ is given by $\Psi([q_t])=[\rho_t]\in\ca{X}_d$. To prove the theorem, we assume that the path $([\rho_t])_{t\in(-\epsilon,\epsilon)}$ in $\ca{X}_d$ is stationary at $t=0$, and only need to show that the path $([q_t])$ in $\ca{F}_d$ is stationary at $t=0$ as well.

After conjugating each $\rho_t$ by some element of $\SO_0(n,n+1)$, we may assume that $\rho_t$ itself (rather than the conjugacy class $[\rho_t]$) is stationary, namely $\dt{\rho}_0=0$. As a consequence, taking the derivative in $t$ of the $\rho_t$-equivariance relation
$\hf_t(\gamma.z)=\rho_t(\gamma).\hf_t(z)$ ($\gamma\in\pi_1(\Sigma_t)$, $z\in\wt{\Sigma}$), we obtain, at $t=0$,
$$
	\dt{\hf}_0(\gamma.z)=\rho_0(\gamma).\dt{\hf}_0(z)\ \  \text{ for all }\gamma\in\pi_1(\Sigma),\ z\in\wt{\Sigma}.
$$
Namely, $\dt{F}_0$ is $\rho_0(\Sigma)$-invariant. In particular, the normal component $\dt{\hf}_0^\ca{N}$ of $\dt{F}_0$, which can be viewed as a Jacobi field on the A-surface $F_0(\wt{\Sigma})\subset\H^{p,q}$ (see \S \ref{subsec_variation}), descents to a Jacobi field on the closed A-surface $F_0(\wt{\Sigma})/\rho_0(\pi_1(\Sigma))$ in the quotient  $M:=\H^{p,q}/\rho_0(\pi_1(\Sigma))$. Therefore, we may apply Corollary \ref{coro_Avariation2} to conclude that $\dt{\hf}_0^\ca{N}\equiv0$, or in other words, $\dt{\hf}_0$ is tangent to $\hf$.

Since the first fundamental form of the immersion $\hf_t$ is conformal to the Riemann surface structure $\Sigma_t$, by the first variation formula of first fundamental form (see e.g.\ \cite{spivak}), the variation of this Riemann surface structure at $t=0$ is given by its Lie derivative with respect to the tangent vector field of $\Sigma$ corresponding to $\dt{F}_0$. It follows that the path $([\Sigma_t])$ in ${\ca{T}}_g$ is stationary at $t=0$. 
Therefore, we may assume from the beginning that $\Sigma_t$ is the fixed Riemann surface $\Sigma$ and only the triple $(\ca{M}_t,\ve\mu_t,\ve\nu_t)$ on $\Sigma$ vary.
By Lemma \ref{lemma_variationconnection}, the connection $\nabla^{\ca{N}_t}$ on the normal bundle $\ca{N}_t\subset\hf_t^*\T\H^{p,q}$ is stationary at $t=0$. But on the other hand, by definition of A-surfaces, $(\ca{N}_t,\nabla^{\ca{N}_t})$ is isomorphic to $\ca{K}^{-2}\oplus\cdots\oplus\ca{K}^{-n+1}\oplus\ca{M}^{-1}_t$ (as a real vector bundle) endowed with the connection 
$$
	\nabla=
	\scalebox{0.9}{$
		\mat{
			\nabla_2&1^*&&&\\[0.1cm]
			1&\nabla_3&\ddots&&\\
			&\ddots&\ddots&1^*&\\[0.2cm]
			&&1&\nabla_{n-1}&\hspace{-0.1cm}\ve\mu_t^*+\ve\nu_t^*\\[0.1cm]
			&&&\hspace{-0.2cm}\ve\mu_t+\ve\nu_t&\hspace{-0.4cm}\nabla_n
		}$}
$$
(where each $1\in \C\cong H^0(\Sigma,\ca{K}\otimes\Hom_\C(\ca{K}^{-j+1},\ca{K}^{-j}))$ is viewed as a $\Hom_\R(\ca{K}^{-j+1},\ca{K}^{-j})$-valued $1$-form). In particular, the equivalence class of the holomorphic chain 
  \begin{center}
	\includegraphics[width=7.5cm]{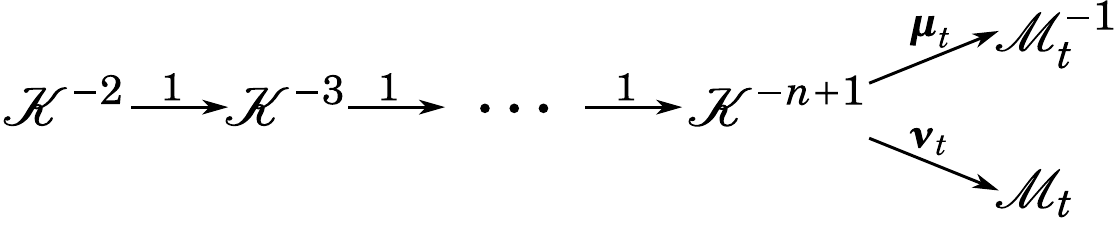}
\end{center}
is determined by the connection $\nabla^{\ca{N}_t}$. Therefore, we conclude that this equivalence class is stationary at $t=0$. By definition of the equivalence relation in $\wt{\ca{F}}_d|_{[\Sigma]}$, this means that the point in $\ca{F}_d|_{[\Sigma]}=\wt{\ca{F}}_d|_{[\Sigma]}/\sim$ represented by $(\ca{M}_t,\ve\mu_t,\ve\nu_t)$ is stationary at $t=0$, as required.
\end{proof}

\subsection{$G_2'$-Hitchin component}\label{subsec_g2higgs}
We now turn to the special case of Def.\ \ref{def_cyclic} and Prop./Def.\ \ref{propdef_harmonic} with 
\begin{equation}\label{eqn_caseg2}
n=3,\quad \TT_i=\ca{K}^{-i}\ (i=1,2,3),\quad \ve\alpha_1=\tfrac{1}{\sqrt{2}},\quad \ve\alpha_2=\ve\alpha_3^+=1.
\end{equation}
The $(\ca{E},\Phi)$'s in this case are parametrized by $\ve\alpha_n^-\in H^0(\Sigma,\ca{K}^6)$ and are exactly the Higgs bundles in the $G_2'$-Hitchin component constructed from the principal $3$-dimensional subalgebra $\frak{s}\subset\frak{g}_2^\C$ in \S \ref{subsec_g2complex}. By using the description in Remark \ref{rk_p} of the $\SO(4)$-module $\frak{p}\subset\frak{g}_2'$, one may check that $(\ca{E},\Phi)$ is indeed a $G_2'$-Higgs bundles, so the Hitchin connection $D^H$ on $\ca{E}=\ca{K}^3\oplus\cdots\oplus\ca{K}^{-3}$ is a  $G_2'$-connection. We omit the details here, as it involves the lengthy definition of $G$-Higgs bundles for a general $G$. Alternatively, we now use Proposition \ref{prop_g2realform} to identify $D^H$ as a $G_2'$-connection in a direct way.

Recall from Proposition \ref{prop_g2complex} that $\frak{g}_2^\C$ can be identified as the Lie algebra of matrices infinitesimally preserving the $3$-form $\varpi:=\ima\big(\vae^{147}+\vae^{246}-\vae^{345}-\scalebox{0.8}{$2\sqrt{2}$}\,\vae^{156}-\tfrac{1}{\sqrt{2}}\vae^{237}\big)\in\bigwedge^3\C^{7*}$. But each term $\vae^{ijk}$ in $\varpi$ (note that $i+j+k=12$) can be viewed as a holomorphic section of $\bigwedge^3\ca{E}^*$ through
$$
\vae^{ijk}(\xi_1,\xi_2,\xi_3):=\sum_{\sigma\in\frak{S}_3}\mathrm{sign}(\sigma)\ \xi_{\sigma(1)}^{(i)}\otimes\xi_{\sigma(2)}^{(j)}\otimes\xi_{\sigma(3)}^{(k)}
$$
for any holomorphic local sections $\xi_1,\xi_2,\xi_3$ of $\ca{E}$,
where $\frak{S}_3$ is the permutation group of the indices $\{1,2,3\}$ and the superscript ``$(i)$'' means taking the $i$th component in $\ca{E}=\ca{K}^3\oplus\cdots\oplus\ca{K}^{-3}$, namely the $\ca{K}^{4-i}$-component; the tensor product yields a holomorphic section of $\ca{K}^{4-i}\otimes\ca{K}^{4-j}\otimes\ca{K}^{4-k}=\ca{O}$, which is just a holomorphic function. Therefore, we may consider $\varpi$ as a holomorphic section of $\bigwedge^3\ca{E}^*$ as well. It allows us to deduce from Proposition \ref{prop_g2realform} and Prop./Def.\ \ref{propdef_harmonic}:
\begin{corollary}\label{coro_g2hitchin}
In the case \eqref{eqn_caseg2}, the following statements about the objects in Prop./Def.\ \ref{propdef_harmonic} hold:
\begin{itemize}
	\item 
the harmonic metric $H=\ve{h}_3^{-1}\oplus\ve{h}_2^{-1}\oplus\ve{h}_1^{-1}\oplus1\oplus\ve{h}_1\oplus\ve{h}_2\oplus\ve{h}_3$ satisfies
$\ve{h}_3=\ve{h}_1\ve{h}_2/4$;
\item the Hitchin connection $D^H$ preserves $\varpi\in H^0(\Sigma,\bigwedge^3\ca{E}^*)$;
\item $\varpi$ restricts to a real $3$-form on each fiber of $\ca{E}_\R$ which, together with the metric $\langle\sth,\sth\rangle$, makes the fiber isomorphic to $\R^{3,4}$ endowed with the $3$-form $\varphi$ associated to the cross product $\times$ (see \S \ref{subsec_crossproduct}). 
\end{itemize}
As a consequence, the representation $\rho$ takes values in the subgroup $G_2'$ of $\SO_0(\ca{E}_\R|_{z_0},\langle\sth,\sth\rangle)\cong\SO_0(3,4)$ formed by linear transformations of $\ca{E}_\R|_{z_0}$ preserving $\langle\sth,\sth\rangle$ and the $3$-form, whereas the minimal immersion $f$ takes values in the corresponding submanifold $X_{G_2'}$ of $X_{\SO_0(3,4)}$.
 \end{corollary}
\begin{proof}
It is a general fact for Higgs bundles in the Hitchin component (see \cite{hitchin_lie,labourie_cubic}) that the Cartan involution $\rho$ (induced by the harmonic metric) on each fiber of the underlying Lie algebra bundle commutes with the holomorphic involution $\sigma$ (induced by the principal $3$-dimensional subalgebra), and in the cyclic case $\rho$ furthermore preserves $\frak{h}$ . Therefore, the statement on $H$ follows from Proposition \ref{prop_g2realform}. For the statement on $D^H$, we note that in the current setting, the expression \eqref{eqn_proofhiggsA2} of $D^H$ under the local frame $(\ell_n^{-1},\cdots,\ell_1^{-1},\underline{1},\ell_1,\cdots,\ell_n)$ of $\ca{E}$ (here $n=3$ and we may take $\ell_i=\dz^{-i}$) becomes
\begin{equation}\label{eqn_DHg2}
D^H=\dif+\left(
\scalebox{0.9}{
	$\begin{array}{ccc|c|ccc}
	-\tfrac{\pa h_3}{h_3}&\tfrac{h_3}{h_2}\dbz&&&&\alpha_3^-\dz&\\[0.3cm]
	\dz&-\tfrac{\pa h_2}{h_2}&\tfrac{h_2}{h_1}\dbz&&&&\alpha_3^-\dz\\[0.3cm]
	&\dz&-\tfrac{\pa h_1}{h_1}&\tfrac{h_1\dbz}{\sqrt{2}}&&&\\[0.3cm]
	\midrule
	&&\tfrac{\dz}{\sqrt{2}}&0&\tfrac{h_1\dbz}{\sqrt{2}}&&\\
    \midrule
	&&&\tfrac{\dz}{\sqrt{2}}&\tfrac{\pa h_1}{h_1}&\tfrac{h_2}{h_1}\dbz&\\[0.3cm]
	\tfrac{\cj{\alpha_3^-}\dbz}{h_2h_3}&&&&\dz&\tfrac{\pa h_2}{h_2}&\tfrac{h_3}{h_2}\dbz\\[0.3cm]
	&\tfrac{\cj{\alpha_3^-}\dbz}{h_2h_3}&&&&\dz&\tfrac{\pa h_3}{h_3}	
	\end{array}$
}
\right), \text{ where }h_3=\tfrac{h_1h_2}{4}.
\end{equation}
This matrix has the form \eqref{eqn_g2complex} in Proposition \ref{prop_g2complex}, hence $D^H$ preserves $\varpi$. The statement on the restriction of $\varpi$ to $\ca{E}_\R$ is again contained in Proposition \ref{prop_g2realform}. These statements imply that $D^H$ is a $G_2'$-connection on $\ca{E}_\R$, whence the statement on $\rho$ and $f$ follows.
\end{proof}

The A-surfaces underlying $(\ca{E},\Phi)$ belong to the subclass discussed in Proposition \ref{prop_holomorphicA}:
\begin{proposition}\label{prop_g2A}
In the above setting, the A-surface immersion $\hf:\wt{\Sigma}\to\H^{4,2}$ given in Theorem \ref{thm_higgsA} is $\ac{J}$-holomorphic. As a consequence, $-F$ is $\ac{J}$-anti-holomorphic and the images of $\pm F$ are
$\ac{J}$-holomorphic curve with nowhere vanishing second fundamental form and timelike osculation lines.
\end{proposition}
\begin{proof}
Consider the holomorphic local frame $(\ell_3^{-1},\ell_2^{-1},\ell_1^{-1},\underline{1},\ell_1,\ell_2,\ell_3)=(\dz^3,\dz^2,\cdots,\dz^{-3})$ of $\ca{E}$. In the proof of Theorem \ref{thm_higgsA}, we have shown that under the frame $(\underline{1},u_1,v_1,u_2,v_2,u_3,v_3)=(\ell_3^{-1},\cdots,\ell_3)R$ of  $\ca{E}_\R$ given by
$$
R=\tfrac{1}{\sqrt{2}}
\left(
\scalebox{0.9}{$
	\begin{array}{c|cc|cc|cc}
		&&&&&h_3^\frac{1}{2}&-\ima h_3^\frac{1}{2}\\[0.2cm]
		&&&h_2^\frac{1}{2}&-\ima h_2^\frac{1}{2}&&\\[0.2cm]
		&h_1^\frac{1}{2}&-\ima h_1^\frac{1}{2}&&&&\\
		\midrule
		\hspace{-0.2cm}\scalebox{0.9}{$\sqrt{2}$}&&&&&&\\
		\midrule
		&h_1^{-\frac{1}{2}}&\ima h_1^{-\frac{1}{2}}&&&&\\[0.2cm]
		&&&h_2^{-\frac{1}{2}}&\ima h_2^{-\frac{1}{2}}&&\\[0.2cm]
		&&&&&h_3^{-\frac{1}{2}}&\ima h_3^{-\frac{1}{2}}
	\end{array}
	$}
\right)
$$	
(note that $h_3=\frac{h_1h_2}{4}$ by Corollary \ref{coro_g2hitchin}), the Hitchin connection is represented by the matrix 
$$
\Omega=
\scalebox{0.9}{
	$\mat{
		&\transp\theta&&\\[0.1cm]
		\theta&\Omega_1&\transp\ve\alpha_2&\\[0.1cm]
		&\ve\alpha_2&\Omega_2&\transp\ve\alpha_3\\[0.1cm]
		&&\ve\alpha_3&\hspace{-0.1cm}\Omega_n
	}$};
$$
and as a consequence, there are maps $U_i,V_i$ ($i=1,2,3$) from $\wt\Sigma$ to $\R^{3,4}$ such that the matrix-valued map $P:=(F,U_1,V_1,U_2,V_2,U_3,V_3)$ satisfies $P^{-1}\dif P=\Omega$. In particular, $\dif F_z:\T_z\wt\Sigma\to \T_{F(z)}\H^{4,2}\cong F(z)^\perp$ sends the unit tangent vectors $h_1^{-\frac{1}{2}}\pa_x,h_1^{-\frac{1}{2}}\pa_y\in\T_z\wt\Sigma$ to $U_1$ and $V_1$, respectively.

On the other hand, by the proof of Proposition \ref{prop_g2realform}, the modified frame $(\underline{1},u_2,v_2,u_1,v_1,-v_3,u_3)$ (which equals $(\ell_3^{-1}\,\cdots,\ell_3)B$ with $B$ from that proof) is a $G_2'$-frame of $\ca{E}_\R$ in the sense of Definition \ref{def_g2}, so the matrix $(F,U_2,V_2,U_1,V_1,-V_3,U_3)$ is in $G_2'$. Thus, if we re-denote the column vectors of this matrix by $e_1,\cdots,e_7$ successively, then their cross products are given by Table \ref{table_crossproduct}. In particular, we have $F\times U_1=e_1\times e_4=e_5=V_1$. By definition of the almost complex structure $\ac{J}$ (see \S \ref{subsec_almostcomplex}), this means that $\ac{J}$ sends $U_1(z)\in F(z)^\perp=\T_{F(z)}\H^{4,2}$ to $V_1(z)$ for all $z\in\wt{\Sigma}$. Combining with the above description of $\dif F_z$, we conclude that $\dif F_z\circ\ac{j}=\ac{J}\circ\dif F_z$ (where $\ac{j}$ denotes the complex structure on $\wt\Sigma$), namely $F$ is $\ac{J}$-holomorphic.
\end{proof}

\appendix
\section{Proof of Theorem \ref{thm_daili}}
The proof here is adapted from \cite[\S 5]{dai-li_minimal}. In order to make it easier for the reader to compare with the original work, we restate the theorem as follows, where the assumption from Remark \ref{rk_reduction} has been added and the notation has been reset (see the remark below) so as to be consistent with \cite{dai-li_minimal}.
\begin{theorem*}
	Let $\Sigma$ be a closed Riemann surface of genus $\geq2$, $\TT_1,\cdots,\TT_n$ ($n\geq2$) be holomorphic line bundles on $\Sigma$ and $\ve{h}_i$ be a hermitian metric on $\TT_i$ ($i=1,\cdots,n$). Let $\ve{\gamma}_j$ ($j=1,\cdots, n-1$), $\ve{\gamma}_n$ and $\ve{\beta}$ 
	be holomorphic sections of the line bundles $\ca{K}\!\TT_j^{-1}\TT_{j+1}$, $\ca{K}\!\TT_n^{-1}$ and $\ca{K}\!\TT_1\TT_2$, respectively, such that none of $\ve{\gamma}_1,\cdots,\ve{\gamma}_n$ is identically zero and the following equations hold on $\Sigma$:
	\begin{align}
		&\text{if $n=2$: }		
		\begin{cases}
			\pazbz\log h_1=-\tfrac{h_2}{h_1}|\gamma_1|^2+h_1h_2|\beta|^2,\\[0.2cm]
			\pazbz\log h_2=\tfrac{h_2}{h_1}|\gamma_1|^2-\tfrac{1}{h_2}|\gamma_2|^2+h_1h_2|\beta|^2;
		\end{cases}
		\label{eqn_appthm1}
		\\
		&\text{if $n\geq3$: }	
		\begin{cases}
			\pazbz\log h_1=-\tfrac{h_2}{h_1}|\gamma_1|^2+h_1h_2|\beta|^2,\\[0.2cm]
			\pazbz\log h_2=\tfrac{h_2}{h_1}|\gamma_1|^2-\tfrac{h_3}{h_2}|\gamma_2|^2+h_1h_2|\beta|^2,\\[0.2cm]
			\pazbz\log h_k=\tfrac{h_k}{h_{k-1}}|\gamma_{k-1}|^2-\tfrac{h_{k+1}}{h_k}|\gamma_k|^2\quad (3\leq k\leq n-1),\\[0.2cm]
			\pazbz\log h_n=\tfrac{h_n}{h_{n-1}}|\gamma_{n-1}|^2-\tfrac{1}{h_n}|\gamma_n|^2.
		\end{cases}
		\label{eqn_appthm2}
	\end{align}
	Then the following statements hold.
	\begin{enumerate}[label=(\arabic*)]
		\item\label{item_appthm2}
		If $(\ve\gamma_1)\prec(\ve\beta)\neq \emptyset$, then $\pazbz\log h_1\leq0$ on $\Sigma$, with strict inequality away from the zeros of $\ve\gamma_1$. 
		\item\label{item_appthm1}
		If $(\ve{\gamma}_n)\prec\cdots\prec (\ve{\gamma}_2)\prec \min\big\{(\ve{\gamma}_1),(\ve{\beta})\big\}$, then
		$\pazbz\log h_i\leq0$ on $\Sigma$ for $i=2,\cdots,n$, with strict inequality away from the zeros of $\ve\gamma_i$.
	\end{enumerate}
\end{theorem*}
\begin{remark*}
The hermitian line bundle $(\TT_i,\ve{h}_i)$ here corresponds to $(\TT_{n+1-i}^{-1},\ve{h}_{n+1-i}^{-1})$ in Theorem \ref{thm_daili}. In particular, $\log h_1,\log h_2,\cdots$ correspond to $-\log h_n,-\log h_{n-1},\cdots$ therein, respectively. Meanwhile, $\ve\gamma_n,\cdots,\ve\gamma_2,\ve\gamma_1,\ve\beta$ correspond respectively to $\ve\alpha_1,\cdots,\ve\alpha_{n-1},\ve\alpha_n^+,\ve\alpha_n^-$. In fact, while the equations in Theorem \ref{thm_daili} are the Hitchin equation for the Higgs bundle $(\ca{E},\Phi)$ exhibited in \S \ref{def_cyclic}, the above \eqref{eqn_appthm1} and \eqref{eqn_appthm2} are the Hitchin equation for the Higgs bundle 
\begin{align*} 
	&\hspace{0.1cm}\ca{E}=\TT_1\oplus\TT_2\cdots \oplus \TT_n\oplus \ca{O}\oplus \TT_n^{-1}\oplus  \cdots \TT_2^{-1}\oplus \TT_1^{-1},\\
	&\Phi=
	\scalebox{0.83}{$
		\left(
		\begin{array}{cccc|c|cccc}
			0&&&&&&&\ve\beta&\\
			\ve\gamma_1&0&&&&&&&\ve\beta\\
			&\ddots&\ddots&&&&&&\\
			&&\hspace{-0.2cm}\ve\gamma_{n-1}&0&&&&&\\
			\midrule
			&&&\ve\gamma_n&0&&&&\\
			\midrule
			&&&&\ve\gamma_n&0&&&\\
			&&&&&\ve\gamma_{n-1}&\ddots&&\\
			&&&&&&\ddots&0&\\
			&&&&&&&\ve\gamma_1&0
		\end{array}
		\right)
		$}~.
\end{align*}
The original result in \cite{dai-li_minimal} deals with case $\TT_i=\ca{K}^{n+1-i}$, $\ve\gamma_1=\cdots=\ve\gamma_n=1$.
\end{remark*}
\begin{proof}
	When $\ve\beta\equiv0$, \ref{item_appthm2} is trivial and \ref{item_appthm1} can be proved by a simple adaptation of the argument below for $\ve\beta\not\equiv0$. So let us assume $\ve\beta\not\equiv0$ from now on and only outline the adaption for $\ve\beta\equiv0$ at the end.
	
	Pick a background conformal metric $\ve{g}=h_0|\dz|^2$ on $\Sigma$. After being divided by $h_0$, each term in  \eqref{eqn_appthm2} is a globally defined function on $\Sigma$. The terms on the right-hand side are
	$$
	\|\ve{\gamma}_j\|^2=\tfrac{h_{j+1}}{h_0h_j}|\gamma_j|^2\quad (j=1,\cdots,n-1),\quad
	\|\ve\gamma_n\|^2=\tfrac{1}{h_0h_n}|\gamma_n|^2,\quad
	\|\ve{\beta}\|^2=\tfrac{h_1h_2|\beta|^2}{h_0}.
	$$
	In what follows, we extend the definition of the real logarithm and exponential functions by putting $\log 0:=-\infty$ and $e^{-\infty}:=0$. Consider the functions $u_0,\cdots, u_n:\Sigma\to\{-\infty\}\cup\R$ given by
	$$
	u_0:=\log\|\ve\beta\|^2,\quad u_i:=\log\|\ve\gamma_i\|^2\quad (i=1,\cdots,n).
	$$
	Note that $u_0$ (resp.\ $u_i$) is smooth away from the discrete set $\{u_0=-\infty\}$ (resp.\ $\{u_i=-\infty\}$), which is just the zero locus of $\ve\beta$ (resp.\ $\ve\gamma_i$).
	
	Let $\Delta=\frac{4}{h_0}\pazbz$ be the Laplacian associated to the metric $\ve{g}$.
	Since $\pazbz\log |\phi|^2=0$ for any holomorphic function $\phi(z)$, by \eqref{eqn_appthm1} and the first two equations in \eqref{eqn_appthm2}, away from $\{u_0=-\infty\}$, we have
	\begin{align*}
		\tfrac{1}{4}\Delta u_0&=\tfrac{1}{h_0}\pazbz\big(\log|\beta|^2+\log h_1+\log h_2-\log h_0\big)=\tfrac{1}{h_0}\pazbz\log h_1+\tfrac{1}{h_0}\pazbz\log h_2+\tfrac{\kappa}{2}\\
		&=\big(-\|\ve\gamma_1\|^2+\|\ve\beta\|^2\big)+\big(\|\ve\gamma_1\|^2-\|\ve\gamma_2\|^2+\|\beta\|^2\big)+\tfrac{\kappa}{2}\\
		&=2\|\ve\beta\|^2-\|\ve\gamma_2\|^2+\tfrac{\kappa}{2}=2e^{u_0}-e^{u_2}+\tfrac{\kappa}{2}
	\end{align*}
	where $\kappa=-\frac{2}{h_0}\pazbz\log h_0$ is the curvature of $\ve{g}$. By similar computations, we obtain
	$$
	\text{if $n=3$: }
	\begin{cases}
		\tfrac{1}{4}\Delta u_0=2e^{u_0}-e^{u_2}+\tfrac{\kappa}{2}~,\\
		\tfrac{1}{4}\Delta u_1=2e^{u_1}-e^{u_2}+\tfrac{\kappa}{2};\\
		\tfrac{1}{4}\Delta u_2=-e^{u_0}-e^{u_1}+e^{u_2}+\tfrac{\kappa}{2};\\
	\end{cases}
	\
	\text{if $n\geq3$: }
	\begin{cases}
		\tfrac{1}{4}\Delta u_0=2e^{u_0}-e^{u_2}+\tfrac{\kappa}{2}~,\\
		\tfrac{1}{4}\Delta u_1=2e^{u_1}-e^{u_2}+\tfrac{\kappa}{2}~,\\
		\tfrac{1}{4}\Delta u_2=-e^{u_0}-e^{u_1}+2e^{u_2}-e^{u_3}+\tfrac{\kappa}{2}~,\\
		\tfrac{1}{4}\Delta u_k=-e^{u_{k-1}}+2e^{u_k}-e^{u_{k+1}}+\tfrac{\kappa}{2}\quad (3\leq k\leq n-1),\\
		\tfrac{1}{4}\Delta	u_n=-e^{u_{n-1}}+e^{u_n}+\tfrac{\kappa}{2}.
	\end{cases}
	$$
	Taking differences, we get equations for $\tfrac{1}{4}\Delta(u_i-u_{i+1})$ away from $\{u_i=-\infty\}\cup\{u_{i+1}=-\infty\}$:
	\begin{align}
		\text{if $n=2$: }&
		\begin{cases}
			\tfrac{1}{4}\Delta(u_0-u_1)=2e^{u_0}-2e^{u_1},\\
			\tfrac{1}{4}\Delta(u_1-u_2)=e^{u_0}+3e^{u_1}-2e^{u_2};
		\end{cases}
		\text{if $n=3$: }
		\begin{cases}
			\tfrac{1}{4}\Delta(u_0-u_1)=2e^{u_0}-2e^{u_1},\\
			\tfrac{1}{4}\Delta(u_1-u_2)=e^{u_0}+3e^{u_1}-3e^{u_2}+e^{u_3},\\
			\tfrac{1}{4}\Delta(u_2-u_3)=-e^{u_0}-e^{u_1}+3e^{u_2}-2e^{u_3};
		\end{cases}
		\label{eqn_appthmproof4}\\
		&\text{if $n\geq4$: }
		\begin{cases}
			\tfrac{1}{4}\Delta(u_0-u_1)=2e^{u_0}-2e^{u_1},\\
			\tfrac{1}{4}\Delta(u_1-u_2)=e^{u_0}+3e^{u_1}-3e^{u_2}+e^{u_3},\\
			\tfrac{1}{4}\Delta(u_2-u_3)=-e^{u_0}-e^{u_1}+3e^{u_2}-3e^{u_3}+e^{u_4},\\
			\tfrac{1}{4}\Delta(u_{k-1}-u_k)=-e^{u_{k-2}}+3e^{u_{k-1}}-3e^{u_k}+e^{u_{k+1}}\quad (4\leq k\leq n-1),\\
			\tfrac{1}{4}\Delta(u_{n-1}-u_n)=-e^{u_{n-2}}+3e^{u_{n-1}}-2e^{u_n}.
		\end{cases}
		\label{eqn_appthmproof5}
	\end{align}
	
	We are now ready to carry out the proof of statements \ref{item_appthm2} and \ref{item_appthm1}. 
	
	\ref{item_appthm2}
	The assumption $(\ve\gamma_1)\prec(\ve\beta)$ implies that $e^{u_0-u_1}=\|\ve\beta\|^2/\|\ve\gamma_1\|^2$ attains its maximum $A_1>0$ on a set $P_1\subset\Sigma$ not containing the zeros of $\ve\beta$ (note that the condition $(\ve\gamma_1)\leq(\ve\beta)$ is sufficient for $\|\ve\beta\|^2/\|\ve\gamma_1\|^2$ to attain maximum, but in this case a maximal point might be a common zero of $\ve\gamma_1$ and $\ve\beta$ with the same multiplicity). Since a maximal point of $e^{u_0-u_1}$ is also a maximal point of $u_0-u_1$, by the first equation in \eqref{eqn_appthmproof4} or \eqref{eqn_appthmproof5} and Maximum Principle, we obtain the following inequality, where every function is evaluated at a given $p\in P_1$:
	$$
	\text{at $p\in P_1$:}\quad 0\geq \tfrac{1}{4}\Delta (u_0-u_1)=2e^{u_0}-2e^{u_1}=2e^{u_1}(e^{u_0-u_1}-1)=2e^{u_1}(A_1-1).
	$$
	It follows that $A_1\leq 1$ because $e^{u_1}=\|\ve\gamma_1\|^2>0$ at $p$.
	
	We now improve this to the strict inequality $A_1<1$ by using E. Hopf's Strong Maximum Principle (see e.g.\ \cite[Theorem 3.5]{gilbarg-trudinger}), which says that any function $u$ on a disk in $\C$ satisfying $\pazbz u\geq cu$ for a constant $c>0$ cannot achieve nonnegative maximum unless it is constant. Suppose by contradiction that $A_1=1$. Given $p\in P_1$, on any disk $U$  centered at $p$ and avoiding the zeros of $\ve\beta$, we have
	$$
	\tfrac{1}{h_0}\pazbz(u_0-u_1)=\tfrac{1}{4}\Delta(u_0-u_1)=2e^{u_1}(e^{u_0-u_1}-1)\geq e^{u_1}(u_0-u_1).
	$$
	Since the maximum $\log A_1=0$ of $u_0-u_1$ on $U$ is achieved at $p$, we can use Strong Maximum Principle to infer that $u_0-u_1$ is identically $0$ on $U$. Therefore, the nonempty closed set $P_1$ is also open, hence is the whole $\Sigma$. This is impossible because $P_1$ does not intersect the zero locus of $\ve\beta$, which is nonempty by assumption. Thus, we have shown $A_1<1$. The required statement follows:
	$$
	\tfrac{1}{h_0}\pazbz\log h_1=\|\ve\beta\|^2-\|\ve\gamma_1\|^2=
	\begin{cases}
		0 &\text{ on $\{\ve\gamma_1=0\}$},\\
		\|\ve\gamma_1\|^2\big(\tfrac{\|\ve\beta\|^2}{\|\ve\gamma_1\|^2}-1\big)\leq \|\ve\gamma_1\|^2(A_1-1)<0 &\text{ on $\{\ve\gamma_1\neq0\}$}.
	\end{cases}
	$$

	\ref{item_appthm1} Assuming $(\ve\gamma_n)\prec\cdots\prec(\ve\gamma_2)\prec\min\big\{(\ve\gamma_1),(\ve\beta)\big\}$, we first observe that the divisors $(\ve\gamma_1)$ and $(\ve\beta)$ cannot be both empty: otherwise we have $(\ve\gamma_2)=\cdots=(\ve\gamma_n)=\emptyset$ as well, and hence the holomorphic line bundles $\ca{K}\!\TT_j^{-1}\TT_{j+1}$, $\ca{K}\!\TT_n^{-1}$ and $\ca{K}\!\TT_1\TT_2$ are all trivial; but the triviality of the first two implies that $\TT_i=\ca{K}^{n+1-i}$, so $\ca{K}\!\TT_1\TT_2=\ca{K}^{2n}$ is not trivial (as $\Sigma$ has genus $\geq2$), a contradiction. Therefore, after possibly replacing $(\TT_1,\ve{h}_1)$ by $(\TT_1^{-1},\ve{h}_1^{-1})$ and interchanging $\ve\gamma_1$ with $\ve\beta$ (which does not affect the equations), we may assume that $(\ve\beta)\neq\emptyset$.
	
	We also infer from the assumption that $e^{u_0-u_2}+e^{u_1-u_2}=\tfrac{\|\ve\gamma_1\|^2+\|\ve\beta\|^2}{\|\ve\gamma_2\|^2}$ attains its maximum $A_2>0$ on a set $P_2\subset\Sigma$ not containing the zeros of $\ve\gamma_2$, while $e^{u_{k-1}-u_k}=\|\ve\gamma_{k-1}\|^2/\|\ve\gamma_k\|^2$ ($3\leq k\leq n$) attains its maximum $A_k>0$ on $P_k\subset\Sigma$ not containing the zeros of $\ve\gamma_k$. We proceed in the following steps:
	
	\textbf{Step 1.} Use Maximum Principle to produce a set of algebraic inequalities on $A_2,\cdots,A_n$.
	
	First consider the case $n\geq4$. Using \eqref{eqn_appthmproof5} and Maximum Principle, we get
	\begin{align*}
		\text{at $p\in P_3$:}\quad 0&\geq\tfrac{1}{4}e^{-u_3}\Delta(u_2-u_3)\\
		&=e^{-u_3}(-e^{u_0}-e^{u_1}+3e^{u_2}-3e^{u_3}+e^{u_4})\\
		&=-e^{u_2-u_3}(e^{u_0-u_2}+e^{u_1-u_2}-1)+2(e^{u_2-u_3}-1)-(1-e^{u_4-u_3})\\
		&\geq-A_3(A_2-1)+2(A_3-1)-\big(1-\tfrac{1}{A_4}\big),
	\end{align*}
	where the multiplier $e^{-u_3}$ in the first line is allowed because $P_3$ does not contain the zeros of $\ve\gamma_3$. Similarly, we have
	\begin{align*}
		\text{at $p\in P_k$ ($4\leq k\leq n-1$):}\quad 0&\geq \tfrac{1}{4}e^{-u_k}\Delta(u_{k-1}-u_k)\\
		&=e^{-u_k}(-e^{u_{k-2}}+3e^{u_{k-1}}-3e^{u_k}+e^{u_{k+1}})\\
		&=-e^{u_{k-1}-u_{k}}(e^{u_{k-2}-u_{k-1}}-1)+2(e^{u_{k-1}-u_k}-1)-(1-e^{u_{k+1}-u_k})\\
		&\geq -A_k(A_{k-1}-1)+2(A_k-1)-\big(1-\tfrac{1}{A_{k+1}}\big);\\
		\text{at $p\in P_n$:}\quad 0&\geq \tfrac{1}{4}e^{-u_n}\Delta(u_{n-1}-u_n)\\
		&=e^{-u_n}(e^{-u_{n-2}}+3e^{u_{n-1}}-2e^{u_n})\\
		&=-e^{u_{n-1}-u_n}(e^{u_{n-2}-u_{n-1}}-1)+2(e^{u_{n-1}-u_n}-1)\\
		&\geq -A_n(A_{n-1}-1)+2(A_n-1).
	\end{align*}
	As for $P_2$, we use the inequality
	$\Delta\log(e^f+1)=\tfrac{e^f}{e^f+1}\Delta f+\tfrac{e^f}{(e^f+1)^2}\|\nabla f\|^2\geq \tfrac{e^f}{e^f+1}\Delta f$, which holds for any $f\in C^\infty(\Sigma)$, to get
	\begin{align}
		\text{at $p\in P_2$:}\quad0&\geq \tfrac{1}{4}\Delta \log\big(e^{u_0-u_2}+e^{u_1-u_2}\big)\label{eqn_A2}\\
		&=\tfrac{1}{4}\Delta \log \big(e^{u_0-u_1}+1\big)+\tfrac{1}{4}\Delta (u_1-u_2)\nonumber\\
		&\geq\frac{e^{u_0-u_1}}{e^{u_0-u_1}+1}\cdot\tfrac{1}{4}\Delta (u_0-u_1)+\tfrac{1}{4}\Delta(u_1-u_2)\nonumber\\
		&=\frac{e^{u_0-u_1}}{e^{u_0-u_1}+1}\cdot 2(e^{u_0}-e^{u_1})+(e^{u_0}+3e^{u_1}-3e^{u_2}+e^{u_3})\nonumber\\
		&=\frac{e^{-u_1}(e^{u_0}-e^{u_1})^2}{e^{u_0-u_1}+1}+(e^{u_0}-e^{u_1})+(e^{u_0}+3e^{u_1}-3e^{u_2}+e^{u_3})\nonumber\\
		&\geq (e^{u_0}-e^{u_1})+(e^{u_0}+3e^{u_1}-3e^{u_2}+e^{u_3})\nonumber\\
		&=e^{u_2}\left[2\big(e^{u_0-u_2}+e^{u_1-u_2}-1\big)-\big(1-e^{u_3-u_2}\big)\right]\nonumber\\
		&\geq e^{u_2}\left[2(A_2-1)-\big(1-\tfrac{1}{A_3}\big)\right]\nonumber.
	\end{align}
	Thus, we have obtained
	\begin{equation}\label{eqn_elementary}
		\text{ if $n\geq4$: }
		\begin{cases}
			2(A_2-1)-\big(1-\tfrac{1}{A_3}\big)\leq 0,\\
			-A_3(A_2-1)+2(A_3-1)-\big(1-\tfrac{1}{A_4}\big)\leq0,\\
			-A_k(A_{k-1}-1)+2(A_k-1)-\big(1-\tfrac{1}{A_{k+1}}\big)\leq0\quad (4\leq k\leq n-1),\\
			-A_n(A_{n-1}-1)+2(A_n-1)\leq0.
		\end{cases}
	\end{equation}
	
	When $n=2,3$, the same argument yields
	
	\begin{align}
		&\text{ if $n=2$: } A_2-1\leq0;\\
		&\text{ if $n=3$: }
		\begin{cases}
			2(A_2-1)-\big(1-\tfrac{1}{A_3}\big)\leq 0,\\
			-A_3(A_2-1)+2(A_3-1)\leq0.\label{eqn_elementaryn=3}
		\end{cases}
	\end{align}
	
	\textbf{Step 2.} Infer from these inequalities that $A_2,\cdots,A_n$ are either all equal to $1$ or all less than $1$.
	
	If $n=2$, there is nothing to prove. If $n=3$, we can rewrite \eqref{eqn_elementaryn=3} as
	$4\big(1-\tfrac{1}{A_3}\big)\leq 2(A_2-1)\leq 1-\tfrac{1}{A_3}$, whence
	the required conclusion follows easily. Now assume that $n\geq4$ and $A_2,\cdots,A_n$ are not all $1$. Put $B_k:=\big(1-\tfrac{1}{A_k}\big)-(A_{k-1}-1)$ for $3\leq k\leq n$. Then \eqref{eqn_elementary} can be rewritten as
	\begin{equation}\label{eqn_elementary2}
		\begin{cases}
			A_2-1\leq B_3,\\
			A_3B_3\leq B_4,\\
			A_kB_k\leq B_{k+1}\quad (4\leq k\leq n-1),\\
			A_nB_n\leq -(A_n-1).
		\end{cases}
	\end{equation}
	
	Let us first show $A_n\leq 1$. Applying the definition of the $B_k$'s repeatedly, we get
	\begin{align*}
		1-\tfrac{1}{A_n}&=B_n+A_{n-1}-1=B_n+A_{n-1}\big(1-\tfrac{1}{A_{n-1}}\big)\\
		&=B_n+A_{n-1}(B_{n-1}+A_{n-2}-1)=B_n+A_{n-1}B_{n-1}+A_{n-1}A_{n-2}\big(1-\tfrac{1}{A_{n-2}}\big)\\
		&=B_n+A_{n-1}B_{n-1}+A_{n-1}A_{n-2}(B_{n-2}+A_{n-3}-1)=\cdots\\
		&=B_n+A_{n-1}B_{n-1}+A_{n-1}A_{n-2}B_{n-2}+\ \cdots\ +A_{n-1}\cdots A_3B_3+A_{n-1}\cdots A_3(A_2-1)
	\end{align*}
	The result has $n-1$ terms and each term is no larger than $B_n$ by  \eqref{eqn_elementary2}. Therefore, we have
	$$
	1-\tfrac{1}{A_n}\leq (n-1) B_n\leq -(n-1)\big(1-\tfrac{1}{A_n}\big),
	$$
	which implies $A_n\leq 1$, as required.
	
	Next, we show $A_2<1$. Suppose by contradiction that $A_2\geq1$. Then the inequalities in \eqref{eqn_elementary2} imply, successively for $3\leq k\leq n$, that $B_k\geq0$; or equivalently,  $1-\tfrac{1}{A_k}\geq A_{k-1}-1$. This in turn implies successively that $A_k\geq1$ for $3\leq k\leq n$. Since we already know that $A_n\leq 1$, it follows that $A_n=1$. This allows us to deduce again from \eqref{eqn_elementary2} that $B_k\leq 0$ for $3\leq k\leq n$ (in reversed order). Therefore, we have $B_3=\cdots=B_n=0$, which implies $A_2=\cdots=A_n=1$, a contradiction.
	
	Finally, we show that $A_3,\cdots,A_n<0$ in a similar way as follows. Suppose by contradiction that $A_{l-1}<1$ but $A_l\geq1$ for some $3\leq l\leq n$. Then we have $B_l:=\big(1-\tfrac{1}{A_l}\big)-(A_{l-1}-1)>0$. Similarly as in the above proof of $A_2<1$, we infer that $B_k>0$ and hence $A_k\geq1$ for $l\leq k\leq n$. So again we have $A_n=1$, which implies $B_n\leq0$, a contradiction.

	\textbf{Step 3.} Exclude the case $A_2=\cdots=A_n=1$ by using Stronger Maximum Principle.
	
	Suppose $A_2=\cdots=A_n=1$ by contradiction. Observe that all the inequalities in \eqref{eqn_A2} other than the first and the last hold outside of the zeros of $(\ve\beta)$ or $(\ve\gamma_1)$ (namely, they do not require the functions to be evaluated at $p\in P_2$). So we have the following estimates for  $w:=\log\big(e^{u_0-u_2}+e^{u_1-u_2}\big)=\log\tfrac{\|\ve\gamma_1\|^2+\|\ve\beta\|^2}{\|\ve\gamma_2\|^2}$ away from the zeros of $(\ve\beta)$ or $(\ve\gamma_1)$:
	\begin{align*}
		\tfrac{1}{4}\Delta w&\geq e^{u_2}\left[2\big(e^{u_0-u_2}+e^{u_1-u_2}-1\big)+\big(e^{-(u_2-u_3)}-1\big)\right]\\
		&\geq 2e^{u_2}\big(e^{u_0-u_2}+e^{u_1-u_2}-1\big)\\
		&=2e^{u_2}(e^w-1)\geq 2e^{u_2}w,
	\end{align*}
	where the second inequality is because $e^{u_2-u_3}\leq A_3=1$.
	Since $w$ achieves its maximum $\log A_2=0$ on $P_2$, we may use Strong Maximum Principle in the same way as in Part \ref{item_appthm2} to conclude that $P_2=\Sigma$, or in other words, $\tfrac{\|\ve\gamma_1\|^2+\|\ve\beta\|^2}{\|\ve\gamma_2\|^2}=1$ throughout $\Sigma$. By the assumption $(\ve\gamma_2)\prec\min\big\{(\ve\gamma_1),(\ve\beta)\}$, it follows that $\ve\gamma_1$ and $\ve\beta$ do not have any common zero and $\ve\gamma_2$ is nowhere zero.
	
	In order to derive a contradiction, we consider the maximum $A_2'$ of $e^{u_1-u_2}=\|\ve\gamma_1\|^2/\|\ve\gamma_2\|^2$. By the second equation in \eqref{eqn_appthmproof4} or \eqref{eqn_appthmproof5}, at any point where $A_2'$ is achieved (which cannot be a zero of $e^{u_2}=\|\ve\gamma_2\|^2$ because $(\ve\gamma_2)\prec(\ve\gamma_1)$), we have
	\begin{align*}
		0&\geq\tfrac{1}{4}e^{-u_2}\Delta(u_1-u_2)=e^{-u_2}\big(e^{u_0}+3e^{u_1}-3e^{u_2}+e^{u_3}\big)\\
		&\geq e^{u_0-u_2}+e^{u_1-u_2}+2e^{u_1-u_2}-3+e^{-(u_2-u_3)}\\
		&=2e^{u_1-u_2}-2+e^{-(u_2-u_3)}\\
		&\geq 2A_2'-2+\tfrac{1}{A_3}=2A_2'-1.
	\end{align*}
	It follows that $A_2'\leq \tfrac{1}{2}$. But this is impossible because $\|\ve\gamma_1\|^2/\|\ve\gamma_2\|^2=1$ on the zero locus of $\ve\beta$ (which is nonempty as observed at the beginning). This contradiction shows that $A_2=\cdots=A_n=1$ cannot hold. Therefore, we have $A_2,\cdots, A_n<1$.
	
	\textbf{Step 4.} Deduce the required statement from the inequality $A_2,\cdots,A_n<1$ that we just obtained:
	since $(\ve\gamma_n)\leq\cdots\leq(\ve\gamma_2)\leq\min\big\{(\ve\gamma_1),(\ve\beta)\big\}$, we have
	$$
	\tfrac{1}{h_0}\pazbz\log h_2=\|\ve\gamma_1\|^2-\|\ve\gamma_2\|^2+\|\ve\beta\|^2\\
	=\begin{cases}
		0&\text{on $\{\ve\gamma_2=0\}$},\\
		\|\ve\gamma_2\|^2\big(\tfrac{\|\ve\gamma_1\|^2+\|\ve\beta\|^2}{\|\ve\gamma_2\|^2}-1\big)\leq \|\ve\gamma_2\|^2(A_2-1)<0&\text{on $\{\ve\gamma_2\neq0\}$};
	\end{cases}
	$$
	$$
	\tfrac{1}{h_0}\pazbz\log h_k=\|\ve\gamma_{k-1}\|^2-\|\ve\gamma_k\|^2=
	\begin{cases}
		0&\text{on $\{\ve\gamma_k=0\}$},\\
		\|\ve\gamma_k\|^2\big(\tfrac{\|\ve\gamma_{k-1}\|^2}{\|\ve\gamma_k\|^2}-1\big)\leq \|\ve\gamma_k\|^2(A_k-1)<0&\text{on $\{\ve\gamma_k\neq0\}$}.
	\end{cases}
	$$
	
	\vspace{0.4cm}
	
	Thus, we have proven statement \ref{item_appthm1} in the case $\ve\beta\not\equiv0$. To adapt this proof to $\ve\beta\equiv0$, we basically only need to erase all the terms containing $u_0$ or $\|\ve\beta\|^2$ and take into account some simplications along the way. For example, when $n\geq4$, the equations that we started with become
	$$
	\begin{cases}
		\tfrac{1}{4}\Delta(u_1-u_2)=3e^{u_1}-3e^{u_2}+e^{u_3}\geq 3e^{u_1}-3^{u_2},\\
		\tfrac{1}{4}\Delta(u_2-u_3)=-e^{u_1}+3e^{u_2}-3e^{u_3}+e^{u_4},\\
		\tfrac{1}{4}\Delta(u_{k-1}-u_k)=-e^{u_{k-2}}+3e^{u_{k-1}}-3e^{u_k}+e^{u_{k+1}}\quad (4\leq k\leq n-1),\\
		\tfrac{1}{4}\Delta(u_{n-1}-u_n)=-e^{u_{n-2}}+3e^{u_{n-1}}-2e^{u_n}.
	\end{cases}
	$$
	Putting $A_k:=\max_\Sigma e^{u_{k-1}-u_k}$ for $k=2,\cdots,n$ (only $A_2$ differs from the above one), we may use Maximum Principle to deduce almost the same inequalities on $A_2,\cdots, A_n$ as \eqref{eqn_elementary}, except that the first one is replaced by the simpler inequality $A_2-1\leq0$. Then, in the same way as above, we may show that $A_2,\cdots,A_n<1$, whence the required statement follows. 
\end{proof}

\bibliographystyle{siam} \bibliography{frenet_bib}
\end{document}